\documentclass[11pt]{amsart}

\usepackage{epsf,amssymb,amsmath,overpic,amscd,psfrag,stmaryrd,times,color}
\usepackage[all]{xy}
\usepackage{hyperref}

\newtheorem{thm}{Theorem}[section]
\newtheorem{prop}[thm]{Proposition}
\newtheorem{lemma}[thm]{Lemma}
\newtheorem{cor}[thm]{Corollary}

\newtheorem{claim}[thm]{Claim}
\newtheorem{fact}[thm]{Fact}

\numberwithin{equation}{subsection}
\numberwithin{thm}{subsection}

\theoremstyle{definition}
\newtheorem{defn}[thm]{Definition}

\theoremstyle{remark}

\newtheorem{rmk}[thm]{Remark}

\newtheorem{convention}[thm]{Convention}
\newtheorem{warning}[thm]{Warning}

\DeclareMathAlphabet{\mathpzc}{OT1}{pzc}{m}{it}

\newcommand{\C}{\mathbb{C}}

\newcommand{\R}{\mathbb{R}}
\newcommand{\Z}{\mathbb{Z}}
\newcommand{\Q}{\mathbb{Q}}

\newcommand{\bdry}{\partial}
\newcommand{\s}{\vskip.1in}
\newcommand{\n}{\noindent}

\newcommand{\vp}{\varphi}

\newcommand{\be}{\begin{enumerate}}
\newcommand{\ee}{\end{enumerate}}
\newcommand{\op}{\operatorname}
\newcommand{\bs}{\boldsymbol}

\newcommand{\cb}{\color{black}}

\begin{document}

\title[Definition of cylindrical contact homology in dimension three]
{Definition of cylindrical contact homology in dimension three}

\author{Erkao Bao}
\address{University of California, Los Angeles, Los Angeles, CA 90095}
\email{bao@math.ucla.edu}

\author{Ko Honda}
\address{University of California, Los Angeles, Los Angeles, CA 90095}
\email{honda@math.ucla.edu} \urladdr{http://www.math.ucla.edu/\char126 honda}

\date{This version: July 18, 2018.}



\begin{abstract} \cb
In this paper we give a rigorous definition of cylindrical contact homology for contact $3$-manifolds that admit nondegenerate contact forms with no contractible Reeb orbits, and show that the cylindrical contact homology  is an invariant of the contact structure. \cb
\end{abstract}

\maketitle

\setcounter{tocdepth}{1}
\tableofcontents

\section{Introduction}

The goal of this paper is to give a rigorous definition of cylindrical contact homology for contact $3$-manifolds that admit nondegenerate \cb {\em hypertight} contact forms, i.e., contact forms with no contractible Reeb orbits. \cb
By ``defining cylindrical contact homology'' we mean the following:

\begin{thm}\label{main thm}
\cb Let $(M,\xi)$ be a closed oriented contact $3$-manifold that admits nondegenerate hypertight contact forms.
\be
\item There is an assignment of a $\Q$-vector space $HC(\mathcal{D})$ to each auxiliary data
$$\mathcal{D}=(\alpha,\{L_i\},\{\vp_i\},\{J_i\},\{\overline{J}_i\}),$$
where $\alpha$ is a nondegenerate contact form for $\xi$ with no contractible Reeb orbits, $L_i\to \infty$, $i=1,2,\dots$, $\vp_i:M\to \R^+$ is a smooth function, $\vp_i\alpha$ has no elliptic orbits of action $< L_i$, and the rest of $\mathcal{D}$ is described in Section~\ref{subsection: def}.
\item $HC(\mathcal{D})$ is the direct limit of homologies of chain complexes generated by good Reeb orbits of $\vp_i\alpha$ of action $<L_i$ whose differential counts Fredholm index $1$ holomorphic cylinders.
\item Given two data $\mathcal{D}_1$ and $\mathcal{D}_2$ for $(M,\xi)$ there exists a natural isomorphism
$$\Phi_{\mathcal{D}_2\mathcal{D}_1}: HC(\mathcal{D}_1)\xrightarrow\sim HC(\mathcal{D}_2)$$
satisfying the property $\Phi_{\mathcal{D}_3\mathcal{D}_2}\circ \Phi_{\mathcal{D}_2\mathcal{D}_1}= \Phi_{\mathcal{D}_3\mathcal{D}_1}$.
\ee
\cb
\end{thm}

The notion of contact homology was proposed by Eliashberg-Givental-Hofer in \cite{EGH} over a decade ago, but a rigorous definition had not been written down yet, even for cylindrical contact homology for contact $3$-manifolds. This is starting to change with our work and also the recent work of Hutchings-Nelson~\cite{HN} towards defining cylindrical contact homology in dimension three.

There are earlier ``abstract perturbation'' approaches which are much broader in scope and are likely to give a definition of contact homology:  polyfolds of Hofer-Wysocki-Zehnder~\cite{HWZ3}, Kuranishi structures of Fukaya-Ono~\cite{FO}, and work of Liu-Tian~\cite{LT} and Ruan~\cite{Ru}.  Our approach is quite different (and closer in spirit to \cite{HN}) in that we do not use any type of abstract perturbation theory: we try to minimize the analysis by more carefully using asymptotic eigenfunctions in the spirit of Hutchings-Taubes~\cite{HT1,HT2}. \cb We expect that cylindrical contact homology, once it is defined via abstract perturbation theory, will be isomorphic to the one in this paper. \cb

\begin{rmk}
The assumption of the theorem is designed to simplify the possible degenerations that we need to analyze. Taking away the assumption would lead to the next level of difficulty: defining full contact homology in dimension three.
\end{rmk}

\begin{rmk}
Since the first version of this paper was submitted, other approaches (cf.\ Pardon \cite{Pa} and Bao-Honda \cite{BH}) have appeared which define full contact homology in all dimensions using abstract perturbation theory.  One of the usual properties of abstract perturbation theory is that, when perturbations are not needed (i.e., the relevant moduli spaces are already transversely cut out) the counts of holomorphic curves without perturbations agrees with the counts with perturbations.  Hence the definition of cylindrical contact homology in this paper agrees with those of \cite{Pa} and \cite{BH}.
\end{rmk}

We also make some remarks on computations:

\s\n
(1) Currently the only case that can be explicitly computed using the definition of cylindrical contact homology given in this paper is the unit cotangent bundle of a hyperbolic surface. This is because the definition relies on replacing elliptic orbits by hyperbolic ones.  It is an interesting problem to determine how to explicitly replace elliptic orbits by a countable collection of hyperbolic orbits (and holomorphic curves connecting them).

\cb
\s\n (2) If we want to compute $HC(\mathcal{D})$ up to a given action $L_1$, we take a contact form $\varphi_1\alpha$ that is nice up to $L_1$, take $\varphi_i\alpha$, $i>1$, such that $\varphi_i=\varphi_1$ on fixed neighborhoods of the Reeb orbits of $\varphi_1\alpha$ of action $\leq L_1$ and $\varphi_i$ is $C^0$-close to $\varphi_1$, and form the direct limit of chain complexes induced by cobordisms from $\varphi_i\alpha$ to $\varphi_{i+1}\alpha$. Then the cobordism maps are induced by isomorphisms of chain complexes up to $L_1$ since we can take the almost complex structure to be very close to an $\R$-invariant one.

\cb

\begin{proof}[Outline of proof.]  Starting with a nondegenerate contact form $\alpha$ with no contractible Reeb orbits for $(M,\xi)$, it is possible to eliminate all its elliptic orbits up to a given action $L>0$ by taking a small perturbation $\vp\alpha$ of $\alpha$.  Such a perturbed contact form (with some extra normalizations near the hyperbolic orbits) will be called {\em $L$-supersimple}. The elimination of elliptic orbits will be reviewed in  Section~\ref{section: elimination of elliptic orbits}.

The advantage of using an $L$-supersimple contact form is that, if $v$ is an $m$-fold branched cover of a finite energy holomorphic map $u$ with $b$ simple branch points, then the Fredholm index $\op{ind}$, given by Equation~\eqref{eqn: Fredholm index of u}, satisfies
$$\op{ind}(v)=m\op{ind}(u)+b;$$
see Lemma~\ref{lemma: multiplicative}.
In particular, if $\op{ind}(u)\geq 0$, then $\op{ind}(v)\geq 0$.

\begin{rmk}
Using $L$-supersimple contact forms is mostly a matter of convenience, used to reduce the number of possible cases that we need to consider.  It is expected (although not worked out in this paper) that elliptic orbits can be treated in a similar manner.
\end{rmk}

The chain groups $CC^L(M,\vp \alpha,J)$ are generated by the good Reeb orbits of action $<L$ for $\vp \alpha$ and the differential $\bdry$ counts $J$-holomorphic maps of $\op{ind}=1$ as usual. Here we require $(\vp \alpha,J)$ to be an {\em $L$-supersimple pair}; see Section~\ref{subsection: L-supersimple pairs} for the definition. Its significance will be explained later in this section.

We need to verify the following:
\begin{enumerate}
\item[(i)] $\bdry^2=0$;
\item[(ii)] an exact symplectic cobordism gives rise to a chain map; and
\item[(iii)] a homotopy of cobordisms gives rise to a chain homotopy.
\end{enumerate}
\cb The contact homology group $HC(\mathcal{D})$ is then defined as the direct limit of groups $HC^{L_i}(\vp_i\alpha,J_i)$ as $L_i\to \infty$. For supersimple contact forms, (i) and (ii) are not difficult with the aid of automatic transversality techniques of Wendl~\cite{We}. \cb Automatic transversality will be reviewed in Section~\ref{section: automatic transversality}, and (i) and (ii) will be proven in Section~\ref{section: def of HC}.

In order to prove (iii), we need to make one type of obstruction bundle calculation using the setup of Hutchings-Taubes~\cite{HT2}, which takes up the rest of the paper. The prototypical gluing problem is the following (there are a few variations, but all of them can be understood in the same way, as explained in Sections~\ref{section: chain homotopy} and \ref{section: gluing}):

\s\n
{\em Prototypical gluing problem.}
Let $\overline{J}^\tau$, $\tau\in[0,1]$, be a $1$-parameter family of almost complex structures adapted to a $1$-parameter family of completed exact symplectic cobordisms $(\widehat{X}^\tau,\widehat{\alpha}^\tau)$, $\tau\in[0,1]$, and let $J_\pm$ be the adapted\footnote{More precisely, ``tame'', which is defined in Section~\ref{subsection: alpha-tame almost complex structures}.} almost complex structures which agree with $\overline{J}^\tau$ at the positive/negative symplectization ends.
Let $v_0\cup v_1$ be a two-level SFT building, arranged from bottom to top as we go from left to right,\footnote{This will be our usual convention.} where:
\begin{enumerate}
\item[(C1)] \label{Condition1} $v_1$ is a holomorphic cylinder from $\gamma$ to $\gamma''$ and $v_0$ is a holomorphic cylinder from $\gamma''$ to $\gamma'$;\footnote{By a {\em curve from $\gamma_+$ to $\gamma_-$} we mean a curve which is asymptotic to $\gamma_+$ at the positive end and to $\gamma_-$ at the negative end.} we assume that $\gamma''$ is negative hyperbolic and $\gamma'$ is positive hyperbolic;
\item[(C2)] $v_0$ maps to a cobordism $(X^{\tau_0},\alpha^{\tau_0},\overline{J}^{\tau_0})$ for some $\tau_0\in (0,1)$ and $v_1$ maps to a symplectization;
\item[(C3)] $\op{ind}(v_0)=-k$, $\op{ind}(v_1)=k$, $k>1$;
\item[(C4)] $v_0$ is a $k$-fold unbranched cover of a transversely cut out (in a $1$-parameter family) cylinder $u_0$ with $\op{ind}(u_0)=-1$ and $v_1$ is regular; and we write $v_0=u_0\circ \pi$, where $\pi$ is the covering map.
\end{enumerate}
 The curves $v_0$ and $v_1$ are also equipped with asymptotic markers at the positive and negative ends; this will be described more precisely later.
Let $\mathcal{M}=\mathcal{M}_{J_+}$ be the moduli space of $v_1$'s from $\gamma$ to $\gamma''$ satisfying the above. We want to glue $v_0$ to $\mathcal{M}/\R$ (or its compactification $\overline{\mathcal{M}/\R}$).  See Figure~\ref{fig: degeneration}.

\begin{figure}[ht]
\begin{center}
\psfragscanon
\psfrag{A}{\tiny $\gamma$}
\psfrag{B}{\tiny $\gamma''$}
\psfrag{C}{\tiny $\gamma'$}
\psfrag{D}{\tiny $0$}
\psfrag{E}{\tiny $k$}
\psfrag{F}{\tiny $-k$}
\psfrag{G}{\tiny \mbox{symplectization}}
\psfrag{H}{\tiny \mbox{cobordism}}
\psfrag{I}{\tiny $v_1$}
\psfrag{J}{\tiny $v_0$}
\includegraphics[width=7cm]{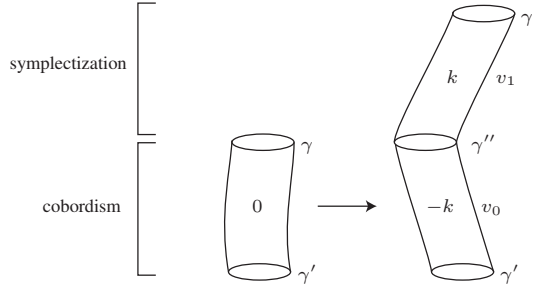}
\end{center}
\caption{The degeneration to the limit $v_0\cup v_1$.  The numbers on the holomorphic curves represent Fredholm indices.} \label{fig: degeneration}
\end{figure}

\s
By a slight modification of \cite{HT2}, there is an obstruction bundle
$$\mathcal{O}\to [R,\infty)\times \overline{\mathcal{M}/\R},\quad  R\gg 0,$$
whose fiber over $(T,v_1)$ is
$$\mathcal{O}(T,v_1)=\op{Hom}(\op{Ker}((D_{v_0}^N)^*)/\R\langle Y\rangle,\R),$$
and a section $\mathfrak{s}$ of the bundle whose zeros we are trying to count \cb as they correspond to curves in $\overline{\mathcal M/\R}$ that are ``gluable'' to $v_0$. \cb Here we assume that $v_0$ is immersed for notational convenience, $ D^N_{v_0}$ is the linearized normal $\overline\bdry$-operator of $v_0$ (i.e., the linearized $\overline \bdry$-operator projected to the normal direction), $(D^N_{v_0})^*$ is its $L^2$-adjoint, and  $Y$ is a nonzero element of $\op{coker}(D^N_{v_0})$ which satisfies the following:  The linearizations of the operators $\overline\bdry_{\overline{J}^\tau} (\exp_{u_0} \xi_0)$ and $\overline\bdry_{\overline{J}^\tau} (\exp_{v_0} \xi)$, projected to the normal direction, can be written as $D^N_{u_0}\xi_0 + (\tau-\tau_0)Y'_0$ and $D^N_{v_0}\xi + (\tau-\tau_0)Y'$, respectively, where $Y'=\pi^* Y'_0$.  If $\Pi_{u_0}$ and $\Pi_{v_0}$ are orthogonal projections to $\ker (D^N_{u_0})^*$ and $\ker (D^N_{v_0})^*$, then $Y_0=\Pi_{u_0} Y_0'$, $Y=\Pi_{v_0} Y'$, and $Y=\pi^* Y_0$.

\begin{rmk} \cb
In general, if $v_0$ is not immersed,  we must replace $D^N_{v_0}$ by the full linearized $\overline \bdry $-operator, which is usually denoted by $D_{v_0}$.  For immersed $v_0$, $D_{v_0}$ and $D^N_{v_0}$ have isomorphic kernels and cokernels and have Fredholm index equal to $\op{ind}(v_0)$ as given by Equation~\eqref{eqn: Fredholm index of u}.  
\end{rmk}

 Suppose $\gamma''$ is an $m(\gamma'')$-fold cover of a simple orbit $\gamma''_0$.  Let $\gamma''_0\times D^2$ be a neighborhood of $\gamma''_0$ and let $(\R/\Z)\times D^2\to \gamma''_0\times D^2$ be its $m(\gamma'')$-fold cover with coordinates $(t,z=x+iy)$ such that $\{z=0\}$ corresponds to $\gamma''$.  Also let $\R\times (\R/\Z)\times D^2$ be the cylinder over $(\R/\Z)\times D^2$ with coordinates $(s,t,z=x+iy)$. We parametrize $\R/\Z$ such that the asymptotic marker of $v_1$ at the negative end corresponds to $t=0$.

We consider the negative end of $v_1\in\mathcal{M}$ and write it as a graph $\eta_1(s,t)$ over a subset of $\R\times (\R/\Z)$; \cb a map of the form $(s,t)\mapsto (s,t,\eta_1(s,t))$ will be called {\em graphical over $\gamma''$}. \cb  Let
$$A=-j_0{\bdry\over \bdry t}-S(t): W^{1,2}(\R/\Z, \R^2) \to L^2(\R/\Z,\R^2)$$
be the asymptotic operator of $\gamma''$, where $j_0=\begin{pmatrix} 0 & -1 \\ 1 & 0 \end{pmatrix}$ and $S(t)$ is a family of symmetric matrices.  The sign convention for $A$ is consistent with \cite{HWZ1,HWZ2} but opposite to that of \cite{HT1,HT2}.  Asymptotic eigenfunctions will be discussed in more detail in Sections~\ref{section: automatic transversality} and \ref{section: evaluation map}.

In an idealized situation, $\eta_1\in \mathcal{M}$ admits a ``Fourier expansion''
$$\eta_1(s,t)=\sum_{i\in \Z-\{0\}} c_i(s) e^{\lambda_i s} f_i(t),$$
where $f_i(t)$ is an eigenfunction of the asymptotic operator $A$ corresponding to $\gamma''$ with unit $L^2$-norm, $\lambda_i$ is the corresponding eigenvalue,
\begin{equation} \label{eqn: ordering eigenvalues 1}
\dots \leq \lambda_{-2}\leq \lambda_{-1} < 0 < \lambda_1\leq \lambda_2\leq \dots,
\end{equation}
$\{f_i(t)\}_{i\in\Z-\{0\}}$ forms an orthonormal basis of $L^2(\R/\Z;\R^2)$, and $c_i(s)$ limits to a constant $c_i$ as $s\to -\infty$ with $c_i=0$ for $i<0$.  (If there are multiple eigenvalues, then we will make specific, convenient choices for $f_i(t)$.) Hence there would be an evaluation map
$$ev^k: \mathcal{M} \to \R^k,$$
$$v_1\mapsto (c_1,\dots, c_k),$$
\cb which can be shown to miss the origin,
and a corresponding quotient
$$\widetilde{ev}^k:\mathcal{M}/\R\to (\R^k -\{0\})/\R^+\simeq S^{k-1}$$
given by $\R$-translation. \cb

In reality, unless the $\overline\bdry_{J_+}$-equation is linear, i.e., of the form
$${\bdry \eta \over \bdry s} +j_0 {\bdry \eta \over \bdry t} + S(t) \eta=0$$ for curves that are close to and graphical over $\R\times \gamma''$, the nonlinear terms seem to interfere with the definition of the higher-order terms in the evaluation map and perhaps the best one can do is Siefring's asymptotic analysis~\cite{Si}. This now brings us to the {\em key reason} for using $L$-supersimple pairs $(\vp \alpha,J)$: For $(\vp \alpha,J)$ supersimple, $\overline{\bdry}_J$ is linear near $\gamma''$ and we are able to define the evaluation map; in fact $c_i(s)=c_i$ for all $s\ll 0$.

\cb The following lemma, proved in Section~\ref{subsection: proof of lemma}, is a consequence of automatic transversality: \cb

\begin{lemma} \label{lemma: good basis for coker}
Suppose (C1)--(C4) hold. Then there exists a basis $\{\sigma_1,\dots,\sigma_k\}$ for $\ker (D_{v_0}^N)^*$, such that the positive ends of $\sigma_i$, $i=1,\dots, k$, are of the form
\begin{equation} \label{eqn: good basis}
\sigma_i(s,t)= e^{-\lambda_i s} f_i(t) \quad \mbox{modulo $f_{k+1}, f_{k+2},\cdots$}
\end{equation}
and $\sigma_k=Y$ modulo $f_{k+1}, f_{k+2}, \cdots$ (up to a nonzero constant multiple).
\end{lemma}

Note that $D^N_{v_0}$ is given locally by ${\bdry\over \bdry s}-A$ and $(D_{v_0}^N)^*$ by ${\bdry \over \bdry s}+A$, when $v_0$ is close to and graphical over $\R\times \gamma''$.

The section $\mathfrak{s}$ of the obstruction bundle $\mathcal{O}$ is homotopic to a section $\mathfrak{s}_0$ which is given by:
\begin{equation} \label{eqn: eqn for s_0}
\mathfrak{s}_0(T,v_1)(\sigma_i) = e^{-2\lambda_i T} c_i,
\end{equation}
where $ev^k(v_1)=(c_1,\dots,c_k)$, $i=1,\dots,k-1.$
In other words, on each slice $\{T\}\times \overline{\mathcal{M}/\R}$,  $\mathfrak{s}_0$ is more or less an evaluation map.  It is important that the zeros of the homotopy $\mathfrak{s}_\zeta$, $\zeta\in[0,1]$, stay away from the boundary $[R,\infty)\times \bdry(\mathcal{M}/\R)$ as we go from $\mathfrak{s}=\mathfrak{s}_1$ to $\mathfrak{s}_0$.  This is proved in Section~\ref{subsection: linearized section s_0}.

\begin{figure}[ht]
\begin{center}
\psfragscanon
\psfrag{A}{\tiny $(0,\dots, 0,1)$}
\psfrag{B}{\tiny $(0,\dots,0,-1)$}
\psfrag{C}{\tiny $S^{k-1}$}
\psfrag{D}{\tiny $\nu$}
\includegraphics[width=5cm]{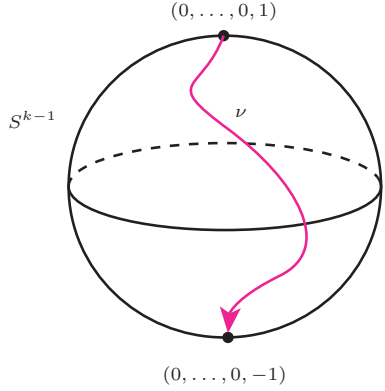}
\end{center}
\caption{The path $\nu\subset S^{k-1}$.} \label{fig: sphere}
\end{figure}

In Section~\ref{section: evaluation map} we prove the transversality of $\widetilde{ev}^k$ for generic $J$ (Theorem~\ref{thm: transversality of evaluation map}). By the transversality of $\widetilde{ev}^k$ and Equation~\eqref{eqn: eqn for s_0}, the zero set $(\mathfrak{s}_0)^{-1}(0)$ is given by $[R,\infty)\times (\widetilde{ev}^k)^{-1}(\{(0,\dots,0,\pm 1)\})$.   Now, let $\nu$ be a generic embedded path in $S^{k-1}$ from $(0,\dots,0,1)$ to $(0,\dots,0,-1)$.  An analysis of $(\widetilde{ev}^k)^{-1}(\nu)$ yields a chain homotopy term $K'\circ \bdry$, where $K'$ is part of the full chain homotopy $K$.  This will be carried out in Section~\ref{section: chain homotopy}.
\end{proof}

\n
{\em Acknowledgements.} The second author thanks Michael Hutchings for various discussions over the years and Vincent Colin and Paolo Ghiggini for ideas leading to using only hyperbolic orbits.  Also the authors owe an enormous mathematical debt to the work of Hutchings-Taubes~\cite{HT2} which is used throughout this paper. Finally, we thank Michael Hutchings again for pointing out some errors in earlier versions of the paper \cb and the referees for extensive lists of comments. \cb

\section{Elimination of elliptic orbits} \label{section: elimination of elliptic orbits}

\begin{convention}
In this paper an ``orbit'' or ``Reeb orbit'' is a closed Reeb orbit, unless stated otherwise.  Also, a Reeb orbit may be a multiple cover of a simple orbit.
\end{convention}

In what follows, $M$ is a closed oriented $3$-manifold. Given a contact form $\alpha$ on $M$, we denote its Reeb vector field by $R_\alpha$.  The {\em $\alpha$-action} of an orbit $\gamma$ is given by $\mathcal{A}_\alpha(\gamma)=\int_\gamma\alpha$.

\cb We first review the Conley-Zehnder index and the classification of Reeb orbits into three types: positive hyperbolic, negative hyperbolic, and elliptic.  Let $\gamma$ be an orbit of $R_\alpha$ with $\alpha$-action $\ell$ and $\tau$ be a framing for $\xi$ along $\gamma$.  The {\em Conley-Zehnder index $\mu_\tau(\gamma)$} of $\gamma$ with respect to $\tau$ is given as follows: Let $\phi_t$, $t\in[0,\ell]$, be the time-$t$ flow of $R_\alpha$ and let $p$ be a point on $\gamma$.  We consider the path $d\phi_t(p): \xi_p\stackrel\sim\to \xi_{\phi_t(p)}$, $t\in[0,\ell]$, of $d\alpha$-symplectic linear maps and in particular the first return map $d\phi_\ell(p):\xi_p\stackrel\sim\to \xi_p$.
\be
\item If the eigenvalues of $d\phi_\ell(p)$ are positive real numbers and $d\phi_t(v)$ winds $k$ times around the origin with respect to $\tau$, where $v\not=0\in \xi_p$ is an eigenvector of $d\phi_\ell(p)$,  then $\mu_\tau(\gamma)=2k$.  Such an orbit $\gamma$ is {\em positive hyperbolic}.
\item If the eigenvalues are negative real numbers and $d\phi_t(v)$ winds $k$ and a half times around the origin with respect to $\tau$, where $v\not=0\in \xi_p$ is an eigenvector of $d\phi_\ell(p)$, then $\mu_\tau(\gamma)=2k+1$. Such an orbit $\gamma$ is {\em negative hyperbolic}.
\item If the eigenvalues are not real and $d\phi_t(w)$ winds between $k$ and $k+1$ times around the origin with respect to $\tau$ for any $w\not=0\in \xi_p$, then $\mu_\tau(\gamma)=2k+1$.  Such an orbit $\gamma$ is {\em elliptic}.
\ee
The definition of $\mu_\tau(\gamma)$ is independent of the choice of $p$. \cb

The starting point of the work is the following theorem:

\begin{thm}[Elimination of elliptic orbits]\label{thm: elimination}
Let $\alpha$ be a nondegenerate contact form for $(M,\xi)$. Then for any \cb finite \cb $L>0$ and $\varepsilon>0$ there exists a smooth function $\vp :M\rightarrow \R^+$ such that:
\begin{enumerate}
\item $\vp$ is $\varepsilon$-close to $1$ with respect to a fixed $C^1$-norm;
\item all the orbits of $R_{\vp \alpha}$ of $\vp \alpha$-action less than $L$ are hyperbolic;
\item[(3)] each positive hyperbolic orbit $\gamma$ \cb of  $\vp \alpha$-action less than $L$ \cb has a neighborhood $(\R/\Z)\times D^2_{\delta_0}$ with coordinates $t,x,y$ such that
\begin{enumerate}
\item $D^2_{\delta_0}=\{x^2+y^2 \leq \delta_0\}$, $\delta_0>0$ small,
\item $\vp \alpha= Hdt +\beta$,
\item $H=c(\gamma)-\varepsilon xy$ with $c(\gamma),\varepsilon>0$ and $c(\gamma)\gg \varepsilon$,
\item $\beta=2xdy+ydx$, and
\item $\gamma=\{x=y=0\}$; and
\end{enumerate}
\item[(4)] each negative hyperbolic orbit $\gamma$ \cb of  $\vp \alpha$-action less than $L$ \cb has a neighborhood $([0,1]\times D^2_{\delta_0})/\sim$ with coordinates $t,x,y$ and identification $(1,x,y)\sim (0,-x,-y)$ satisfying (a)--(e).
\end{enumerate}
\end{thm}

\begin{proof}[Idea of proof.]
\cb (1) and (2) are proved in \cite[Theorem~2.5.2]{CGH1}.  We will explain the idea behind the proof.

Given a simple elliptic orbit $\gamma$ with $\mathcal{A}_\alpha(\gamma)=\ell< L$, there exists a sufficiently small neighborhood $N(\gamma)$ of $\gamma$ and a perturbation of $\alpha$ (which we still call $\alpha$) such that:
\be
\item[(i)] $N(\gamma)\simeq D^2_\delta \times \R/\ell \Z$ with cylindrical coordinates $(r,\theta,z)$, where $\delta>0$ is small;
\item[(ii)] $\alpha= f(r) dz + \tfrac{1}{2} r^2d\theta$ and $f(r)= 1-{1\over 2} a  r^2$ for some $a>0$;
\item[(iii)] $R_\alpha$ is parallel to $X_0=\tfrac{\bdry}{\bdry z} +a{\bdry\over \bdry \theta}$, and the first return map is a rotation through an angle $a\ell\in (0,2\pi)$.
\ee
Note that perturbing $\alpha$ is equivalent to multiplying it by a positive function.

We then modify $\alpha$ on $int(N(\gamma))$ so that $\gamma$ is replaced by a simple negative hyperbolic orbit of action approximately equal to $\ell$ and a simple elliptic orbit of action approximately equal to $2\ell$, and no orbits of action $\leq 1.5 \ell$ are created (besides the negative hyperbolic orbit).  We will not explicitly give the modification, but indicate what happens to a vector field parallel to the Reeb vector field:  First we change $X_0=\tfrac{\bdry}{\bdry z} +a{\bdry\over \bdry \theta}$ to $X_1=\tfrac{\bdry}{\bdry z}+ Y_1$ so that $Y_1=a{\bdry \over \bdry\theta}$ on $\bdry D^2_{\delta}$ and $Y_1= ({\pi\over \ell} +\varepsilon'){\bdry \over \bdry\theta}$ on $D^2_{\delta/2}$, where $\varepsilon'>0$ is small.  This means that the first return map of $X_1$ is a rotation slightly larger than $\pi$ on $D^2_{\delta/2}$.  Next view $D^2_{\delta/2}\times \R/\ell\Z$ as $(D^2_{\delta/2}\times[0,\ell])/ (r,\theta,\ell)\sim (r,-\theta,0)$, also with coordinates $(r,\theta,z)$.  Then $X_1= {\bdry \over \bdry z}+ Y'_1$, $Y'_1= \varepsilon'{\bdry\over \bdry \theta}$, with respect to the latter identification.  Finally modify $Y'_1$ to $Y'_2$ as given in Figure~\ref{fig: modification}.

\begin{figure}[ht]
\begin{center}
\psfragscanon
\psfrag{e}{\tiny $e$}
\psfrag{h}{\tiny $h$}
\includegraphics[width=7cm]{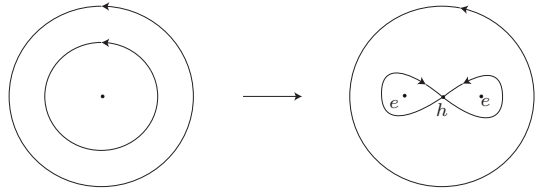}
\end{center}
\caption{Modifying $Y'_1$ on the left-hand side to $Y'_2$ on the right-hand side. What we have drawn are integral curves of $Y'_1$ and $Y'_2$ and their directions. $h$ corresponds to the negative hyperbolic orbit and $e$ the elliptic orbit with (roughly) doubled action.} \label{fig: modification}
\end{figure}

(3) and (4) are standard modifications that can be done in a neighborhood of any hyperbolic orbit.  \cb
\end{proof}

\begin{defn}
Let $L>0$. Then a contact form $\alpha$ is:
\begin{enumerate}
\item {\em $L$-nondegenerate} if all the Reeb orbits of action $<L$ are nondegenerate;
\item {\em $L$-supersimple} if $\alpha$ is $L$-nondegenerate and all the Reeb orbits of action $<L$ are hyperbolic with neighborhoods that are given in Theorem~\ref{thm: elimination}(3) and (4);
\cb
\item {\em $L$-hypertight} if $\alpha$ is $L$-nondegenerate and has no contractible Reeb orbits of action $<L$; and
\item {\em $L$-monotone} if $\alpha$ is $L$-nondegenerate and has no contractible Reeb orbits of action $<L$ and Conley-Zehnder index $\leq 3$ with respect to any bounding disk.
\cb
\end{enumerate}
\end{defn}

\cb Given a nondegenerate contact form $\alpha$  for $(M,\xi)$, we choose positive increasing sequences $L_1, L_2, \dots$ and $\varepsilon_1,\varepsilon_2,\dots$ such that $\lim_i L_i=\infty$ and $\lim_i \varepsilon_i=\varepsilon$ for some $\varepsilon>0$.  Then there exists a sequence $\vp_1,\vp_2,\dots$ of functions $\vp_i:M\to \R^+$ where $\vp_i$ is $\varepsilon_i$-close to $1$ and $\vp_i \alpha$ is $L_i$-supersimple.\footnote{Unfortunately, in general the limit $\lim_i \vp_i$ is not a smooth function.}
The sequences $\{L_i\}_i$ and $\{\varphi_i\}_i$ are parts of additional data that we choose in order to define cylindrical contact homology. \cb

\section{Almost complex structures and moduli spaces} \label{section: almost complex structures}

Let $\alpha$ be a contact form and $R_\alpha$ be its Reeb vector field.  We write $\boldsymbol\gamma=(\gamma_1,\dots, \gamma_l)$ for some ordered tuple of (not necessarily simple) Reeb orbits of $R_\alpha$.

\begin{warning}
Unlike the case of embedded contact homology~\cite{Hu}, $\gamma_i$ may be multiply-covered.
\end{warning}

\subsection{$\alpha$-tame almost complex structures} \label{subsection: alpha-tame almost complex structures}

The almost complex structures on $\R\times M$ that we use are slight variants of almost complex structures which are commonly known as ``adapted to'' or ``compatible with'' a contact form $\alpha$ for the contact manifold $(M,\xi)$.

\begin{defn} \label{def: almost cplx str}
An almost complex structure $J$ on $\R\times M$ with coordinates $(s,x)$ is {\em $\alpha$-tame} if
\begin{itemize}
\item $J$ is $\R$-invariant,
\item $J\bdry_s=g(x) R_\alpha$, where $g$ is a smooth positive function on $M$ and
\item $J(\xi')=\xi'$ for some oriented $2$-plane field $\xi'$ of $M$ on which $d\alpha$ is symplectic and $d\alpha(v,Jv)>0$ for nonzero $v\in \xi'$. \cb {\em Note that we are not requiring $\xi'=\xi$.} \cb
\end{itemize}
\end{defn}

Note that $d\alpha|_{\xi'}$ being symplectic is equivalent to $R_\alpha$ being positively transverse to $\xi'$. \cb  Also observe that if $J$ is $\alpha$-tame, then $d(e^{\delta s} \alpha)$ is tamed by $J$ for sufficiently small $\delta>0$. \cb

\begin{rmk}[Failure of maximum principle] \label{maximal principle} \cb
An $\alpha$-tame $J$ does not necessarily satisfy the maximum principle for $J$-holomorphic curves. However the $\alpha$-tameness of $J$ is sufficient to guarantee that the moduli space of ``finite energy'' $J$-holomorphic curves satisfies SFT compactness; see Section \ref{subsection: compactness}. \cb
\end{rmk}

The space of all $\alpha$-tame $J$ in the class $C^\infty$ will be denoted by $\mathcal{J}(\alpha)$ or by $\mathcal{J}$ if $\alpha$ is understood.

\subsection{$L$-supersimple pairs} \label{subsection: L-supersimple pairs}

When $ \alpha$ is $L$-supersimple, we choose a $\alpha$-tame $J$ which satisfies the following on each neighborhood $(\R/\Z)\times D^2_{\delta_0}$ of a positive hyperbolic orbit of action $< L$ given in Theorem~\ref{thm: elimination} (3) (and an analogous condition for each neighborhood of a negative hyperbolic orbit):
\begin{enumerate}
\item[(J1)] $\xi'=\xi$ on $(\R/\Z)\times (D^2_{\delta_0} - D^2_{2\delta_0/3})$, $\xi'=TD^2_{\delta_0}$ on $(\R/\Z)\times D^2_{\delta_0/3}$, and $\xi'$ is \cb $C^0$-close to $\xi$ on $(\R/\Z)\times (D^2_{2\delta_0/3} - D^2_{\delta_0/3})$; \cb
\item[(J2)] $J: \bdry_s\mapsto \bdry_t +X_H=\bdry_t - \varepsilon( x\bdry_x-y\bdry_y)=H R_{ \alpha}$ and $\bdry_x\mapsto \bdry_y$ on  $(\R/\Z)\times D^2_{\delta_0/3}$.
\end{enumerate}
The subset of such $J$ will be denoted by $\mathcal{J}_{\star_L}(\alpha)$ or $\mathcal{J}_\star( \alpha)$ when $L$ is understood.  Finally, a pair $( \alpha,J)$ consisting of an $L$-supersimple $ \alpha$ and a tame $J$ satisfying the above will be called an {\em $L$-supersimple pair}.

\subsection{Moduli spaces} \label{subsection: moduli spaces}

Let $\alpha$ be a nondegenerate contact form on $M$ and let $J$ be an $\alpha$-tame almost complex structure.

\begin{defn}
Let $\boldsymbol\gamma_+=(\gamma_{+,1},\dots,\gamma_{+,l_+})$ and $\boldsymbol\gamma_-=(\gamma_{-,1},\dots,\gamma_{-,l_-})$ be ordered tuples for $\alpha$ and let $J$ be $\alpha$-tame.
We say a $J$-holomorphic map $u:\dot F\to \R\times M$ is {\em from $\boldsymbol\gamma_+$ to $\boldsymbol\gamma_-$} if it is asymptotic to $\boldsymbol\gamma_+$ at the positive end and to $\boldsymbol\gamma_-$ at the negative end. Here $F$ is a closed Riemann surface, $\mathbf{p}=\mathbf{p}_+\sqcup\mathbf{p}_{-}$ is a finite ordered subset of $F$ (whose points are called {\em punctures}), $\mathbf{p}_+$ corresponds to $\boldsymbol \gamma_+$ in order,  $\mathbf{p}_-$ corresponds to $\boldsymbol \gamma_-$ in order, and $\dot F= F-\mathbf{p}$.
\end{defn}

\n
{\em Fredholm index.}
The Fredholm index of $u$ is given by:
\begin{equation}\label{eqn: Fredholm index of u}
\op{ind}(u)= -\chi(\dot F) + \mu_\tau(\boldsymbol\gamma_+)- \mu_\tau(\boldsymbol\gamma_-) + 2c_1(u^*\xi,\tau),
\end{equation}
where $\tau$ is a framing for $\xi$ defined along $\boldsymbol\gamma_+$ and $\boldsymbol\gamma_-$,  $\mu_\tau(\boldsymbol\gamma_\pm):= \sum_i \mu_\tau(\gamma_{\pm, i})$, and $c_1(u^*\xi,\tau)$ is the first Chern class of $u^*\xi$ with respect to $\tau$. (Here we are viewing $\xi$ as a $2$-plane field of $M$ and also $\R\times M$.) Here $\op{ind}$ keeps track of both the variations of complex structures of $\dot F$ and the infinitesimal automorphisms of $\dot F$.

\s
The most important aspect of working with $L$-supersimple $\alpha$ is the following:

\begin{lemma} \label{lemma: multiplicative}
Let $u$ be a $J$-holomorphic map from $\boldsymbol\gamma_+$ to $\boldsymbol\gamma_-$, where all the orbits of $\boldsymbol\gamma_+$ and $\boldsymbol\gamma_-$ are hyperbolic.
If $v$ is a $k$-fold branched cover of $u$ with total branching multiplicity $b$, then $\op{ind}(v)= k \op{ind}(u) +b$.
\end{lemma}

Here $b$ is the sum over all the branch points of the order of multiplicity minus one; in particular, if all the branch points are double points, then $b$ is the number of branch points.

\begin{proof}
Follows immediately from observing that the Conley-Zehnder indices of hyperbolic orbits behave multiplicatively when we take multiple covers of Reeb orbits.
\end{proof}

\s\n {\em Definition of $\mathcal{M}_J(\boldsymbol\gamma_+,\boldsymbol\gamma_-)$.}
Pick a point $x_\gamma$ on each simple Reeb orbit $\gamma$ of $R_\alpha$. Let $(u,{\bf r})$ be a pair consisting of a $J$-holomorphic map $u:\dot F\to \R\times M$ from $\boldsymbol\gamma_+$ to $\boldsymbol\gamma_-$ and an ordered set ${\bf r}={\bf r}_+\sqcup {\bf r}_-$ of asymptotic markers, where
$${\bf r}_+=(r_{+,1},\dots,r_{+,l_+})\quad \mbox{and} \quad {\bf r}_-=(r_{-,1},\dots,r_{-,l_-})$$
correspond to punctures ${\bf p}_+$ and ${\bf p}_-$, the marker $r_{\pm, i}$ is ``mapped to'' $x_{\gamma^s_{\pm,i}}$, and $\gamma^s_{\pm,i}$ is the simple orbit corresponding to $\gamma_{\pm,i}$.  Here an {\em asymptotic marker} at a puncture $z$ of $F$ is an element of $(T_z F-\{0\})/\R^+$.
The moduli space $\mathcal{M}_J(\boldsymbol\gamma_+,\boldsymbol\gamma_-)$ is the space of $(u,{\bf r})$, modulo biholomorphisms of the domain that take markers to markers.

For convenience we will suppress ${\bf r}$ from $(u,{\bf r})$ when there is no confusion.

\s
If $\boldsymbol\gamma_\pm=(\gamma_{\pm,1})$, then we also write $\boldsymbol\gamma_\pm=\gamma_\pm$ and $\mathcal{M}_J(\boldsymbol\gamma_+,\boldsymbol\gamma_-)=\mathcal{M}_J(\gamma_+,\gamma_-)$; similarly, if ${\bf r}_\pm=(r_{\pm,1})$, then we also write ${\bf r}_\pm=r_\pm$.  We write $\mathcal{M}^*_J(\boldsymbol\gamma_+,\boldsymbol\gamma_-)$ to denote the subset of $\mathcal{M}_J(\boldsymbol\gamma_+,\boldsymbol\gamma_-)$ satisfying $*$.  In particular, $\op{ind}=k$ means ``Fredholm index $k$'', $s$ means ``simple (= non-multiply-covered)", $A$ means ``in the homology class $A\in H_2(M;\Z)$'', $\op{sing}$ means ``singular'', i.e., non-immersed, and $\mbox{cyl}$ means we only count cylinders.

A generic $J\in\mathcal{J}$ is {\em regular}, i.e, the moduli spaces $\mathcal{M}_J^s(\boldsymbol\gamma_+,\boldsymbol\gamma_-)$ are transversely cut out for all $\boldsymbol\gamma_+,\boldsymbol\gamma_-$. Let $\mathcal{J}^{\tiny \mbox{reg}}\subset\mathcal{J}$ be the subset of regular $J$ and let $\mathcal{J}^{<L,\tiny \mbox{reg}}\subset\mathcal{J}$ be the subset of $J$ for which $\mathcal{M}_J^s(\boldsymbol\gamma_+,\boldsymbol\gamma_-)$ is transversely cut out for all $\boldsymbol\gamma_+,\boldsymbol\gamma_-$ with action $< L$. The space $\mathcal{J}^{<L,\tiny \mbox{reg}}_{\star}=\mathcal{J}^{<L,\tiny \mbox{reg}}_{\star_L}$ of regular $L$-supersimple almost complex structures (with respect to a fixed $\alpha$) is defined similarly. \cb We sometimes use the superscript $\delta_0$ on $\mathcal J$ to indicate the radius of $D^2$ in (J1) and (J2). \cb

\begin{rmk}
We sometimes say that a curve is {\em regular} if it is transversely cut out \cb (in the sense of including variations of the complex structure on the domain). \cb This should not be confused with a curve being singular, which means the curve is not an immersion.
\end{rmk}

Finally, we use the coherent orientation system for all $\mathcal{M}_J(\boldsymbol\gamma_+,\boldsymbol\gamma_-)$ as described in Bourgeois-Mohnke~\cite{BM}.

\subsection{Compactness} \label{subsection: compactness}

\cb Since the class of $\alpha$-tame almost complex structures that we use is slightly different from the one in \cite{BEHWZ}, i.e., is not stable Hamiltonian, the SFT compactness results (Theorem 10.1, 10.2, and 10.3 in \cite{BEHWZ}) for our case do not follow automatically from their paper.  In this subsection we show that they still hold in our setting by a slight modification of their proofs. \cb

\begin{thm} \label{thm: SFT compactness fo modified J}
Suppose that $J$ is an $\alpha$-tame almost complex structure. Then the moduli space $\mathcal{M}_J(\boldsymbol\gamma_+,\boldsymbol\gamma_-)$ can be compactified by adding holomorphic buildings (see Section 8 and 9 of \cite{BEHWZ} for the definition of the holomorphic buildings and the topology of the compactified space).
\end{thm}
\cb

\begin{proof}[Sketch of proof]
The proof follows the \cb general outline of \cite{Ho, HWZ1, BEHWZ} \cb with the following modifications to adjust to the slightly different choice of almost complex structures.

\cb Let $K>0$ and let $\mathcal{C}_K$ be the set of ``slow growth'' smooth functions $\phi:\R\to [1,2]$ such that $\phi(s)=1$ for $s\ll 0$, $\phi(s)=2$ for $s\gg 0$, and $0\leq \phi'(s)\leq K$ for all $s$.

\begin{lemma}\label{lemma: J-nonnegative}
There exists $K>0$ small such that, all $\phi \in\mathcal{C}_K$, $d(\phi \alpha)(v,Jv)\geq 0$ for all $v$ and $d(\phi \alpha)(v,Jv)> 0$ for all nonzero $v$ on the region $\phi '(s)>0$.
\end{lemma}

\begin{proof}[Proof of Lemma~\ref{lemma: J-nonnegative}]
Write $v= X+ a\bdry_s + b J\bdry_s=X+ a\bdry_s + b gR_\alpha$, where $X\in \xi'$. Then we have $Jv=JX-b \bdry_s +agR_\alpha$ and
\begin{align}\label{eqn: J-nonnegative}
d(\phi \alpha)(v,Jv) & = \phi'(s)( ds\wedge \alpha) (v,Jv) + \phi d\alpha(v,Jv)\\
\nonumber &= \phi'(s) ( a^2 g+ a\alpha(JX)+ b^2 g+ b \alpha(X)) + \phi d\alpha(X,JX).
\end{align}
Here $d\alpha(X,JX)>0$ if $X\not=0$.  Also note that $d\alpha(X,JX)$ is bounded below by  $C_0 \cdot |X|^2$ for some constant $C_0>0$; $|\alpha(X)|$ and $|\alpha(JX)|$ are bounded above by  $C_1 \cdot |X|$ for some constant $C_1>0$, where the norm $|\cdot|$ is measured with respect to some fixed Riemannian metric on $M$; and $g \geq \kappa_0>0$.
Therefore, by the Cauchy-Schwarz lemma we have
\begin{align}\label{eqn: J-nonnegative2}
d(\phi \alpha)(v,Jv) & = \phi'(s) ( (a^2 + b^2) g + a\alpha(JX)+ b \alpha(X)) + \phi d\alpha(X,JX)\\
\nonumber & \geq \phi'(s) ( (a^2 + b^2) \kappa_0 -  \tfrac{1}{2}(\kappa_0 a^2 + \tfrac{1}{\kappa_0}C_1^2 |X|^2 ) \\
\nonumber & \qquad \qquad -  \tfrac{1}{2}(\kappa_0 b^2 + \tfrac{1}{\kappa_0} C_1^2|X|^2 )) + \phi C_0 |X|^2\\
\nonumber & = \phi'(s) (a^2 + b^2) \tfrac{\kappa_0}{2} +( \phi C_0 - \tfrac{1}{\kappa_0}\phi'(s)C_1^2)|X|^2.
\end{align}
Since $\phi \geq 1$, the lemma follows by taking $0\leq \phi'(s)<K$ for $K$ small.
\end{proof}

\cb
Let $u: \dot F\to \R\times M$ be a $J$-holomorphic map.  The {\em $d\alpha$-energy of $u$} is defined as usual by
$$E_{d\alpha}(u)=\int_{\dot F} u^*d\alpha$$
and {\em the $\alpha$-energy of $u$} is defined slightly differently by
$$E_\alpha(u)=\sup_{\phi\in \mathcal{C}_K}\int_{\dot F}u^* (\phi'(s) ds\wedge \alpha).$$
where $K>0$ is the constant given by Lemma \ref{lemma: J-nonnegative}. The {\em total energy} is defined as
$$E(u) = \sup_{\phi\in \mathcal{C}_K}\int_{\dot F}u^* d(\phi \alpha),$$
\cb and we have the energy bounds
\begin{equation}\label{eqn: energy bounds}
E_{d\alpha}(u) + E_{\alpha}(u)\leq E(u) \leq 2\cdot E_{d\alpha}(u) + E_{\alpha}(u).
\end{equation}

\begin{rmk}[Energy bounds]
If $u$ is a $J$-holomorphic map from $\boldsymbol\gamma_+$ to $\boldsymbol\gamma_-$, then $$E(u)=2\mathcal{A}_\alpha(\boldsymbol\gamma_+)-\mathcal{A}_\alpha(\boldsymbol\gamma_-) \quad \mbox{and} \quad E_{d\alpha}(u)=\mathcal{A}_\alpha(\boldsymbol\gamma_+)-\mathcal{A}_\alpha(\boldsymbol\gamma_-).$$
Then Equation~\eqref{eqn: energy bounds} implies:
\begin{equation}
\mathcal{A}_\alpha(\boldsymbol\gamma_-)  \leq E_\alpha(u)\leq \mathcal{A}_\alpha(\boldsymbol\gamma_+).
\end{equation}
\end{rmk}

With our choice of $J$ and energy (the curves $u$ under consideration will have finite $E(u)$ and $E_{d\alpha}(u)$, or, equivalently, finite $E_\alpha(u)$ and $E_{d\alpha}(u)$), the following key arguments in \cite{Ho, HWZ1, BEHWZ} carry over:
\be
\item The removal of singularities and monotonicity follow from the ``$J$-tameness" of the $2$-forms $d(\phi \alpha)$ for $\phi\in \mathcal{C}_K$, which is the content of Lemma~\ref{lemma: J-nonnegative}.
\item The ends of a holomorphic curve $u:\dot F\to \R\times M$ with finite $E_{d\alpha}(u)$ and $E(u)$ limit to cylinders over closed orbits (discussed in the next paragraph).
\item In particular, the maximum principle is not necessary (cf.\  Remark~\ref{maximal principle}), since the control of the $d\alpha$-energy $E_{d\alpha}(u)$ and (2) imply that there are no holomorphic curves with no positive end.
\ee

We will sketch the proof of (2), following \cite[Section~3.1]{Ho}.  If $E_{d\alpha}(u)=0$, then $u$ maps to a cylinder over a (not necessarily closed) Reeb trajectory $\gamma$. If $\dot F=\C$ in addition, then the little Picard theorem applied to $\widetilde u: \C\to \R\times \widetilde\gamma$ and the finiteness of $E(u)$ imply that $u$ is a constant map. (This is a simplification of the argument of \cite[Lemma~28]{Ho}.) Here $\widetilde\gamma$ is the universal cover of $\gamma$ and $\widetilde u$ is the lift of $u:\C\to\R\times\gamma$.  Note that using a smaller class $\mathcal{C}_K$ of ``slow growth'' test functions does not alter this argument or any other argument where $E_\alpha$ is used.

Next, as in the proof of \cite[Proposition~27]{Ho}, $u$ with finite $E(u)$ and $E_{d\alpha}(u)$ does not have gradient blow-up at the ends of $\dot F$, since otherwise we can construct a nonconstant finite energy plane with $E_{d\alpha}=0$ using the ``Hofer lemma'' (\cite[Lemma~26]{Ho}) and the taming of $J$ by $d(\phi \alpha)$, $\phi \in\mathcal{C}_K$, on compact regions. (Here the gradient is measured with respect to some $\R$-invariant Riemannian metric on $\R\times M$ and a cylindrical metric on the ends of $u$.) The gradient bound implies that if $E_{d\alpha}(u)=0$, $\dot F = \R\times S^1$, and $u$ is nonconstant, then $u$ maps to a cylinder over a closed orbit $\gamma$ and
$$u:\dot F\to \R\times \gamma=\R\times (\R/\Z),$$
$$(s',t')\mapsto (s'+c, mt'+d),$$
where $c,d\in\R$ and $m\in \Z^+$ (this follows from the fact that an entire holomorphic function with bounded derivative is a linear map).  Finally, the gradient bound, the $E_\alpha$-bound, and the taming of $J$ together imply that $u|_{\mathcal{E}}$ (here $\mathcal{E}$ is an end) converges either to a cylinder over a Reeb orbit or to a removable singularity.

Having controlled the behavior of holomorphic curves at the ends (and analogously on the necks), the rest of the SFT compactness argument from \cite[Section 10.2]{BEHWZ} then carries over.  This completes the sketch of proof of Theorem~\ref{thm: SFT compactness fo modified J}.\cb
\end{proof}

\section{Automatic transversality and applications} \label{section: automatic transversality}

\subsection{Calculation of asymptotic eigenvalues and eigenfunctions} \label{subsection: calculation of asymptotic eigenvalues}

For simplicity of notation, in this section we only look at the orbits of action $1$.  Let $S^1=\R/\Z$ with coordinate $t$.  In this subsection we explicitly calculate the eigenvalues and eigenfunctions for certain asymptotic operators
$$A: W^{1,2}(\R/\Z,\R^2)\to L^2 (\R/\Z,\R^2),$$
$$A=-j_0{\bdry\over \bdry t}-S(t),$$
where $j_0=\begin{pmatrix} 0 & -1 \\ 1 & 0 \end{pmatrix}$ and $S(t)$ is a family of symmetric matrices.  This is mostly for reference (especially the elliptic case, which is never used) and is completely standard.

\subsubsection{Elliptic case} \label{subsubsection: elliptic case}

Let $S(t)=\begin{pmatrix} \varepsilon & 0 \\ 0 & \varepsilon \end{pmatrix}$, where $0<\varepsilon<2\pi$.  If we view $\R^2=\C$, then $J$ is multiplication by $i$ and $S(t)$ is multiplication by $\varepsilon$. If $f$ is an eigenfunction of $A$ with eigenvalue $\lambda$, then
$$Af=-i{\bdry f\over \bdry t}-\varepsilon  f=\lambda f$$
and $ f'(t)=i(\lambda+\varepsilon) f$. Hence $\lambda+\varepsilon=2\pi n$ for $n\in\Z$ and $\lambda=2\pi n- \varepsilon$; the corresponding asymptotic eigenfunction is $ f(t)= ce^{2\pi in t}$, $c\in\C$. We write
$$ f_{2n-1}(t)= e^{2\pi i nt},  f_{2n}(t)=ie^{2\pi i nt},$$
$$\lambda_{2n-1}=\lambda_{2n}=2\pi n -\varepsilon,$$
for $n\geq 1$ and
$$ f_{2n-2}(t)= e^{2\pi i nt},  f_{2n-1}(t)=ie^{2\pi i nt},$$
$$\lambda_{2n-2}=\lambda_{2n-1}=2\pi n -\varepsilon,$$
for $n\leq 0$.\footnote{We choose this slightly strange numbering so it is consistent with Equation~\eqref{eqn: ordering eigenvalues 1}.}
Note that $\lambda_{2n-1}=\lambda_{2n}$ is a multiple eigenvalue and the corresponding eigenspace is a complex vector space $\C \langle  f_{2n-1}(t)\rangle$.

\subsubsection{Positive hyperbolic case} \label{subsubsection: positive hyperbolic case}

Let $S(t)=\begin{pmatrix} 0 & \varepsilon \\ \varepsilon & 0 \end{pmatrix}$, where $\varepsilon$ is a small positive number.  Then $A f=\lambda f$ can be written as:
$$ f'(t)= \begin{pmatrix} 0 & -1 \\ 1 & 0 \end{pmatrix} \begin{pmatrix} \lambda & \varepsilon \\ \varepsilon & \lambda \end{pmatrix} f
=\begin{pmatrix} -\varepsilon & -\lambda \\ \lambda & \varepsilon \end{pmatrix}  f .$$

We diagonalize the matrix $P=\begin{pmatrix} -\varepsilon & -\lambda \\ \lambda & \varepsilon \end{pmatrix}$: We solve for $\mu$ in
$$\det \begin{pmatrix} -\varepsilon-\mu & -\lambda \\ \lambda & \varepsilon-\mu \end{pmatrix} = \mu^2+(\lambda^2-\varepsilon^2)=0.$$

\s\n
{\bf Case $|\lambda|>\varepsilon$.}
The eigenvalues of $P$ are $\mu=\pm i \sqrt{\lambda^2-\varepsilon^2}$ and the eigenvectors are
$$v=\begin{pmatrix} \varepsilon -i\sqrt{\lambda^2-\varepsilon^2}\\ -\lambda \end{pmatrix},\quad \overline{v}=\begin{pmatrix} \varepsilon +i\sqrt{\lambda^2-\varepsilon^2}\\ -\lambda \end{pmatrix}.$$

In order for $ f(0)= f(1)$, we require $\sqrt{\lambda^2-\varepsilon^2}=2\pi n$.  Hence $\lambda=\pm\sqrt{(2\pi n)^2+\varepsilon^2}$. \cb Here $n>0$ since we are assuming that $|\lambda|>\varepsilon$. \cb Then $f$ is the real or imaginary part of $e^{i\sqrt{\lambda^2-\varepsilon^2}t} v=e^{i2\pi n t}v$.  Then for $n>0$ we obtain
\begin{align}\label{eqn: eigenfunction1}
 f_{\pm 2n}(t) &= \begin{pmatrix} \varepsilon \cos (2\pi nt)+2\pi n \sin (2\pi nt)\\ \mp\sqrt{(2\pi n)^2+\varepsilon^2}\cos (2\pi nt) \end{pmatrix},\\
\label{eqn: eigenfunction2}
 f_{\pm (2n+1)}(t) &= \begin{pmatrix} -2\pi n\cos (2\pi nt)+\varepsilon \sin (2\pi nt)\\ \mp\sqrt{(2\pi n)^2+\varepsilon^2}\sin (2\pi nt) \end{pmatrix},
\end{align}
$$\lambda_{\pm 2n}= \lambda_{\pm (2n+1)}= \cb \pm \sqrt{(2\pi n)^2+\varepsilon^2}. \cb $$
In particular, we have multiple eigenvalues.

\s\n
{\bf Case  $|\lambda|\leq \varepsilon$.} The eigenvalues are real, and in order for $f(0)=f(1)$ we require $\lambda=\pm \varepsilon$.  We obtain
$$ f_{-1}(t)=\begin{pmatrix} 1\\1  \end{pmatrix}, \quad  f_1(t)=\begin{pmatrix} 1\\ -1 \end{pmatrix},$$
$$\lambda_{-1}=-\varepsilon,\quad \lambda_1=\varepsilon.$$

\subsubsection{Negative hyperbolic case} \label{subsubsection: negative hyperbolic case}

We identify a neighborhood of a negative hyperbolic orbit with $((\R/\Z)\times \R^2)/\sim$, where $(t,x)$ are coordinates on $[0,1]\times \R^2$ and $(1,x)\sim (0,-x)$.  With respect to these coordinates we write $A=-j_0{\bdry\over \bdry t}-S(t)$ with $S(t)=\begin{pmatrix} 0 & \varepsilon \\ \varepsilon & 0 \end{pmatrix}$. The calculations are similar to the positive hyperbolic case and for \cb $n\geq 1$ \cb we obtain
\begin{align}\label{eqn: eigenfunction3}
 f_{\pm (2n-1)}(t) &= \begin{pmatrix} \varepsilon \cos ((2n-1)\pi t)+(2n-1)\pi \sin ((2n-1) \pi t)\\ \mp\sqrt{((2n-1)\pi)^2+\varepsilon^2}\cos ((2n-1)\pi t) \end{pmatrix},\\
\label{eqn: eigenfunction4}
 f_{\pm 2n}(t) &= \begin{pmatrix} -(2n-1)\pi\cos ((2n-1) \pi t)+\varepsilon \sin ((2n-1) \pi t)\\ \mp\sqrt{((2n-1)\pi)^2+\varepsilon^2}\sin ((2n-1)\pi t) \end{pmatrix},
\end{align}
$$\lambda_{\pm (2n-1)}= \lambda_{\pm 2n}= \cb \pm\sqrt{((2n-1)\pi)^2+\varepsilon^2}. \cb $$

\subsection{Automatic transversality}

In this subsection we summarize the parts of the proof of the automatic transversality theorem of Wendl~\cite{We}  which are used later. Automatic transversality was originally due to Gromov~\cite{Gr} and worked out carefully by Hofer-Lizan-Sikorav \cite{HLS} for closed $J$-holomorphic curves.  In this section $J\in \mathcal{J}(\alpha)$.

A Reeb orbit is {\em even} (resp.\ {\em odd}) if its Conley-Zehnder index is even (resp.\ odd) with respect to any trivialization $\tau$. \cb Recall that the parity of the Conley-Zehnder index is independent of the choice of trivialization $\tau$. Note that a Reeb orbit is even if and only if it is positive hyperbolic,  and an even multiple cover of a negative hyperbolic Reeb orbit is even. \cb  Given a $J$-holomorphic curve $u$, let $\#\Gamma_0(u)$ be the number of ends that limit to an even orbit and $\#\Gamma_1(u)$ be the number of ends that limit to an odd orbit.

\cb The following is a special case of \cite[Theorem 1]{We}: \cb

\begin{thm} [Automatic transversality]
\label{thm: automatic transversality}
Let $u:\dot F\to \R\times M$ be an element of $\mathcal{M}_J(\boldsymbol\gamma,\boldsymbol\gamma')$. If $u$ is an immersion, then $u$ is regular if
\begin{equation}
\op{ind}(u) > 2g(F)-2 + \# \Gamma_0(u).
\end{equation}
\end{thm}

\begin{proof}
Let $u\in \mathcal{M}_J(\boldsymbol\gamma,\boldsymbol\gamma')$ be an immersion and let $N\to \dot F$ be a normal bundle to $u(\dot F)\subset \R\times M$ such that $N=TD^2_{\delta'}$ on the ends of $u$ for $\delta'>0$ small.  Here each Reeb orbit has a neighborhood of the form $(\R/\Z)\times D^2_{\delta_0}$ as described in Theorem~\ref{thm: elimination}.
Let
$$D_u^N: W^{k+1,p}(\dot F,N) \to W^{k,p}(\dot F, \Lambda^{0,1}T^* \dot F \otimes_J N)$$
be the linearized normal $\overline\bdry$-operator for $u$ and $(D_u^N)^*$ be its adjoint. The Fredholm index of $D_u^N$ is equal to the Fredholm index $\op{ind}(u)$ of the full linearized $\overline{\bdry}$-operator.  Since $\op{coker} D_u^N\simeq \op{ker} (D_u^N)^*$, it suffices to show $\op{ker} (D_u^N)^*=0$ in order to prove Theorem~\ref{thm: automatic transversality}.

Let $\tau$ be a trivialization of $\xi$ along $\boldsymbol \gamma$ and $\boldsymbol\gamma'$.  Given a nontrivial element $\eta\in \op{ker} D_u^N$, we define the {\em winding number $\op{wind}_\tau(\eta,p)$ of $\eta$ near a puncture $p$} of $F$ to be the winding number of the leading asymptotic eigenfunction of $\eta$ (projected to $TD^2_{\delta'}$) with respect to $\tau$, viewed as a trivialization of $TD^2_{\delta'}$.  We also define the {\em total winding number}
\begin{equation}
\op{wind}_\tau(\eta)=\sum_{p_+\in\mathbf{p_+}} \op{wind}_\tau(\eta,p_+) -\sum_{p_-\in\mathbf{p_-}} \op{wind}_\tau(\eta,p_-).
\end{equation}
For each positive puncture $p$ that limits to $\gamma$, $2\op{wind}_\tau(\eta,p)\leq \mu_\tau(\gamma)$ and equality is possible if and only if $\gamma$ is even; similarly, for each negative puncture that limits to $\gamma$, $2\op{wind}_\tau(\eta,p)\geq \mu_\tau(\gamma)$ and equality is possible if and only if $\gamma$ is even.  Hence
\begin{align}
2\op{wind}_\tau(\eta)& \leq \mu_\tau(\boldsymbol\gamma)-\mu_\tau(\boldsymbol \gamma')-\#\Gamma_1(u)\\
&= \mu_\tau(\boldsymbol\gamma)-\mu_\tau(\boldsymbol \gamma')+\#\Gamma_0(u)-k, \nonumber
\end{align}
where $k$ is the total number of punctures.

The intersection number of $\eta$ with the zero section of $N$ is given by $\op{wind}_\tau(\eta) +c_1(N,\tau)$ and the positivity of intersections in dimension $4$ implies that:

\begin{claim} \label{claim1a}
If $\eta\in \op{ker} D_u^N$ is nonzero, then $\op{wind}_\tau(\eta) +c_1(N,\tau)\geq 0$.
\end{claim}

\begin{claim}\label{claim1b}
If
\begin{equation}
2c_1(N,\tau)+ \mu_\tau(\boldsymbol\gamma)-\mu_\tau(\boldsymbol \gamma')+\#\Gamma_0(u)-k<0,
\end{equation}
then $\op{ker} D_u^N=0$.
\end{claim}

Similarly, if we replace $N$ by $\Lambda^{0,1}T^* \dot F\otimes N$ and note that the leading term of $(D_u^N)^*$ is the $\bdry$-operator, we obtain:

\begin{claim}\label{claim2a}
If $\zeta\in\op{ker} (D_u^N)^*$ is nonzero, then
\begin{equation} \label{eqn: 2a1}
2\op{wind}_\tau(\zeta)\geq \mu_\tau(\boldsymbol\gamma)-\mu_\tau(\boldsymbol \gamma')+\#\Gamma_1(u),
\end{equation}
\begin{equation} \label{eqn: 2a2}
\op{wind}_\tau(\zeta)+c_1(\Lambda^{0,1}T^* \dot F\otimes N,(ds-idt)\otimes \tau)\leq 0,
\end{equation}
where $ds-idt$ is the trivialization of $\Lambda^{0,1}T^*\dot F$ on the ends of $u$ induced by $ds-idt$.  Here $(s,t)$ are the first two coordinates of $\R\times (\R/\Z)\times D^2_{\delta_0}$.
\end{claim}



\begin{claim}\label{claim2b}
If
\begin{equation} \label{equation: condition for vanishing2}
2c_1(\Lambda^{0,1}T^*\dot F\otimes N,(ds-idt)\otimes\tau)+\mu_\tau(\boldsymbol\gamma)-\mu_\tau(\boldsymbol \gamma')+k-\#\Gamma_0(u)>0,
\end{equation}
then $\op{ker} (D_u^N)^*=0$.
\end{claim}

Since
$$\chi(\dot F)=2-2g(\dot F) -k=c_1(T\dot F,ds-idt)= c_1(\Lambda^{0,1}T^*\dot F,ds-idt)$$
and
$$c_1(\Lambda^{0,1}T^*\dot F\otimes N,(ds-idt)\otimes\tau)= c_1(\xi,\tau),$$
Inequality~\eqref{equation: condition for vanishing2} can be rephrased as:
\begin{align*}
\op{ind}(u) &= -\chi (\dot F) + 2c_1(\xi,\tau)+ \mu_\tau(\boldsymbol\gamma)-\mu_\tau(\boldsymbol \gamma')\\
&= -\chi(\dot F) +2 c_1(\Lambda^{0,1}T^*\dot F\otimes N,(ds-idt)\otimes\tau) + \mu_\tau(\boldsymbol\gamma)-\mu_\tau(\boldsymbol \gamma')\\
&> -\chi(\dot F) + \#\Gamma_0(u)-k=2g(\dot F) -2 +\#\Gamma_0(u),
\end{align*}
which implies the theorem.
\end{proof}

The following theorem, due to Hutchings-Taubes~\cite[Theorem 4.1]{HT2} will also be used frequently in conjunction with Theorem~\ref{thm: automatic transversality}.

\begin{thm}\label{thm: immersion}
If $J$ is generic, then for each $\boldsymbol\gamma,\boldsymbol\gamma'$, the set of non-immersed $u\in \mathcal{M}_J^s(\boldsymbol\gamma,\boldsymbol\gamma')$ has real codimension at least two. \cb  More precisely, it is the image of a smooth map $\phi:\mathcal{Z}\to  \mathcal{M}_J^s(\boldsymbol\gamma,\boldsymbol\gamma')$, where  $\mathcal{Z}$ is a manifold of dimension $\dim \mathcal{M}_J^s(\boldsymbol\gamma,\boldsymbol\gamma')-2$. \cb
\end{thm}

\cb Intuitively, $du(p)=0$ with $p\in \dot F$ fixed is a real codimension four condition for a $J$-holomorphic map $u$ and $\dot F$ is real two-dimensional. \cb

\begin{rmk}
\cite[Theorem 4.1]{HT2} states a less general theorem, but its proof \cb (in particular the discussion of the first paragraph of \cite[p.\ 31]{HT2}) implies \cb Theorem~\ref{thm: immersion}.
\end{rmk}

\subsection{Proof of Lemma~\ref{lemma: good basis for coker}} \label{subsection: proof of lemma}

As an application of the automatic transversality technique, we prove the following lemmas, which are slight strengthenings of Lemma~\ref{lemma: good basis for coker} and which will be used in the proof of the chain homotopy.

We are assuming \hyperref[Condition1]{(C1)}--(C4). As before, at the positive end $\gamma''$ of $v_1$, let $f_i(t)$, $i\in \Z-\{0\}$, be the eigenfunctions of the asymptotic operator such that
\begin{enumerate}
\item[(D1)] the corresponding eigenvalues $\lambda_i$ satisfy
$$\dots \leq \lambda_{-2}\leq \lambda_{-1} < 0 < \lambda_1\leq \lambda_2\leq \dots,$$
\item[(D2)] $\{f_i(t)\}_{i\in\Z-\{0\}}$ forms an orthonormal basis of $L^2(\R/\Z;\R^2)$.
\end{enumerate}
At the negative end $\gamma'$ of $v_0$, let $g_i(t)$, $i\in\Z-\{0\}$, be the eigenfunctions of the asymptotic operator such that
\begin{enumerate}
\item[(E1)] the corresponding eigenvalues $\lambda'_i$ satisfy
$$\dots \leq \lambda'_{-2}\leq \lambda'_{-1}< 0 < \lambda'_{1}\leq \lambda'_2\leq \dots,$$
\item[(E2)] $\{g_i(t)\}_{i\in \Z-\{0\}}$ forms an orthonormal basis of $L^2(\R/\Z;\R^2)$.
\end{enumerate}

\begin{lemma}\label{lemma: better basis for coker}
Suppose \hyperref[Condition1]{(C1)}--(C4) hold. Then there exists a basis $\{\sigma_1,\dots,\sigma_k\}$ for $\ker (D_{v_0}^N)^*$ such that the following hold:
\begin{enumerate}
\item The positive ends of $\sigma_i$, $i=1,\dots, k$, are of the form
\begin{equation*}
\sigma_i(s,t)= e^{-\lambda_i s} f_i(t) \quad \mbox{modulo $f_{k+1}, f_{k+2},\cdots$}
\end{equation*}
\item The negative ends of $\sigma_i$, $i=1,\dots, k$, span $\R\langle e^{-\lambda'_{j} s} g_j(t) \rangle_{j=-k}^{-1}$ modulo $g_{-k-1}, g_{-k-2},\cdots$.
\item The negative end of $\sigma_i$, $i=1,\dots, k$ projects nontrivially to $$\R\langle e^{-\lambda'_{j} s} g_j(t) \rangle_{j=-k+i-2}^{-1}$$ if $\lambda_{-k+i-2}'=\lambda'_{-k+i-1}$ and to $$\R\langle e^{-\lambda'_{j} s} g_j(t) \rangle_{j=-k+i-1}^{-1}$$ if $\lambda_{-k+i-2}'<\lambda'_{-k+i-1}.$
\item The negative end of $\sigma_k$ is of the form
$$\sigma_k(s,t)=d_{k,-1} e^{-\lambda'_{-1} s} g_{-1}(t) \quad \mbox{modulo $g_{-k-1}, g_{-k-2},\cdots$},$$
where $d_{k,-1}$ is nonzero.
\item $\sigma_k=Y$, up to a nonzero constant multiple.
\end{enumerate}
\end{lemma}

\begin{convention}
\cb From now on, we may suppress the upper index ``$N$" from $D_u^N$ for simplicity. \cb
\end{convention}

\begin{proof}
(1) We consider the linearized normal $\overline\bdry$-operators
$$D_{v_0}^{\delta}=D_{v_0}^{N,\delta}: W^{k+1,p}_\delta(\dot F,N)\to W^{k+1,p}_\delta(\dot F, \Lambda^{0,1}T^*\dot F\otimes N),$$
where $N\to \dot F$ is the normal bundle to $v_0:\dot F\to \widehat{X}^{\tau_0}$ and we are using weights $\chi_\delta(s)$, i.e., $\zeta\in W^{k+1,p}_\delta$ if and only if $\chi_\delta(s) \zeta\in W^{k+1,p}$. Here $\chi_\delta:\R\to \R^{> 0}$ is a smooth function such that $\chi_\delta(s)=1$ for $s\leq 0$ and $\chi_\delta(s)= e^{-\delta s}$ for $s\gg 0$.

Suppose $k$ is odd; the case of $k$ even is similar (but slightly harder) and will be omitted.  By the asymptotic eigenfunction calculations from Section \ref{subsubsection: negative hyperbolic case},
$$\op{ind}(D_{v_0}^{0})=\op{ind}(D_{v_0}^{\delta_0})=-k, \quad \op{ind}(D_{v_0}^{\delta_1})=-k+2, \quad \cdots \quad, \quad \op{ind}(D_{v_0}^{\delta_{\lfloor {k/2}\rfloor}})=-1,$$
where $\delta_j= 2\pi j$. By Claim~\ref{claim1a},
$$\op{ker}(D_{v_0}^{\delta_0})=\dots =\op{ker}(D_{v_0}^{\delta_{\lfloor {k/2}\rfloor}})=0.$$
Hence we obtain
$$\op{dim} \op{ker}((D_{v_0}^{\delta_0})^*)=k,~~ \op{dim}\op{ker}((D_{v_0}^{\delta_1})^*)=k-2,~~ \cdots~~ ,~~ \op{dim} \op{ker}((D_{v_0}^{\delta_{\lfloor {k/2}\rfloor}})^*)=1.$$

Now consider $\sigma_1,\sigma_2\in \op{ker}((D_{v_0}^{\delta_0})^*)$, whose projections span
$$\op{ker}((D_{v_0}^{\delta_0})^*)/ \op{ker}((D_{v_0}^{\delta_1})^*).$$
For $i=1,2$, let us write
$$\sigma_i=c_1(\sigma_i) e^{-\lambda_1s}f_1(t) + c_2(\sigma_i)e^{-\lambda_2s} f_2(t)\quad \mbox{modulo $f_3,f_4,\cdots$}$$
at the positive end. Then $(c_1(\sigma_1),c_2(\sigma_1))$ and $(c_1(\sigma_2),c_2(\sigma_2))$ are linearly independent, since otherwise a nontrivial linear combination of $\sigma_1$ and $\sigma_2$ will be in $\op{ker}((D_{v_0}^{\delta_1})^*)$ and this contradicts the fact that $\op{ker}((D_{v_0}^{\delta_0})^*)/ \op{ker}((D_{v_0}^{\delta_1})^*)$ is $2$-dimensional.  By induction, row reduction, and possibly renaming the $\sigma_i$, we eventually obtain the basis  $\{\sigma_1,\dots,\sigma_k\}$ which satisfies Equation~\eqref{eqn: good basis}.

(2) follows from the argument of (1).

(3) follows from Claim~\ref{claim2a}.

(4), (5) By the regularity of $u_0$ in a $1$-parameter family $\overline{J}^\tau$,  $\op{ker}( (D_{u_0}^N)^*)$ is $1$-dimensional and is generated by $Y_0$ which comes from the variation of $\overline{J}^\tau$. Choose a trivialization $\tilde\tau$ along $\gamma''$ and $\gamma'$ so that $\mu_{\tilde\tau}(\gamma'')=-1$ and $\mu_{\tilde \tau}(\gamma')=0$.  By Claim~\ref{claim2a}, $Y_0$ has $\op{wind}_{\tilde\tau}=0$ at both ends. Hence its pullback $Y=\pi^* Y_0$ to $\op{ker} ((D_{v_0}^N)^*)$ also has $\op{wind}_{\tilde\tau}=0$ at both ends and is a multiple of $\sigma_k$. This proves (5).  Moreover, since $Y$ is a pullback, it does not have any terms $d_{i,j} e^{-\lambda'_{j} s} g_j(t)$ at the negative end, where $j\not= -1$ modulo $k$. This proves (4).
\end{proof}

\begin{lemma}\label{lemma: better basis for coker 2}
Suppose (C1)--(C4) hold. After modifying $\{g_i(t)\}_{i\in \Z-\{0\}}$ subject to (E1) and (E2), there exists a basis $\{\sigma'_1,\dots,\sigma'_{k-1},\sigma_k'=\sigma_k\}$ for $\ker (D_{v_0}^N)^*$ such that the following hold:
\begin{enumerate}
\item The positive ends of $\sigma'_i$, $i=1,\dots, k$, are of the form
\begin{equation*}
\sigma'_i(s,t)= e^{-\lambda_i s} f_i(t) +\sum_{j>i} c_{i,j} e^{-\lambda_j s} f_j(t).
\end{equation*}
\item The negative ends of $\sigma'_i$, $i=1,\dots, k-1$, are of the form
\begin{equation*}
\sigma'_i(s,t)= d_{i,-k+i-1} e^{-\lambda'_{-k+i-1} s} g_{-k+i-1}(t) + \sum_{j< -k+i-1} d_{i,j} e^{-\lambda'_j s} g_j(t),
\end{equation*}
where $d_{i,-k+i-1}$ is nonzero.
\item The negative end of $\sigma'_k$ is of the form
$$\sigma'_k(s,t)=d_{k,-1} e^{-\lambda'_{-1} s} g_{-1}(t) \quad \mbox{modulo $g_{-k-1}, g_{-k-2},\cdots$},$$
where $d_{k,-1}$ is nonzero.
\item $\sigma'_k=Y$, up to a nonzero constant multiple.
\end{enumerate}
\end{lemma}

\begin{proof}
Let $\{\sigma_1,\dots,\sigma_k\}$ be as in Lemma~\ref{lemma: better basis for coker}.
We construct $\sigma_i'$ by induction, starting with $\sigma_k'=\sigma_k$, which satisfies (3) and (4) by Lemma~\ref{lemma: better basis for coker}. Suppose $\sigma_{i+1}'$ satisfies (1) and (2) and is of the form $\sigma_{i+1}+\sum_{j>i+1} k_j \sigma_j'$. We take $\sigma_i'$ of the form $\sigma_i+\sum_{j>i} k_j \sigma'_j$ so that (1) and (2) are satisfied, with the possible exception of $d_{i,-k+i-1}$ being nonzero.  If $\lambda_{-k+i-2}'<\lambda'_{-k+i-1}$, then $d_{i,-k+i-1}\not=0$ is a consequence of Claim~\ref{claim2a}. On the other hand, if $\lambda_{-k+i-2}'=\lambda'_{-k+i-1}$, then $(d_{i,-k+i-1},d_{i,-k+i-2})\not =0$ by Claim~\ref{claim2a}, and modifying $g_{-k+i-1}$ and $g_{-k+i-2}$ subject to (E1) and (E2) gives $d_{i,-k+i-1}\not=0$.
\end{proof}

\section{The definition of $HC(\mathcal{D})$} \label{section: def of HC}

\subsection{The differential}

\cb
Let $(M,\alpha)$ be $L$-supersimple and $L$-monotone. Recall that an orbit $\gamma$ of $R_\alpha$ is {\em bad} if it is an even multiple cover of a negative hyperbolic orbit and is {\em good} if it is not bad.  Let $\mathcal P_{\alpha, \op{bad}}^L$ (resp.\ $\mathcal P_{\alpha, \op{good}}^L$) be the set of bad (resp.\ good) Reeb orbits of $R_\alpha$ of action $<L$. We write $\mathcal{P}_\alpha^L:= \mathcal P_{\alpha, \op{good}}^L \sqcup \mathcal P_{\alpha, \op{bad}}^L$.  Given $J\in \mathcal{J}^{<L,\tiny \mbox{reg}}_\star(\alpha)$, let $CC^{L}(M,\alpha,J)$ be the  $\Q$-vector space generated by $\mathcal{P}_{\alpha,\op{good}}^L$. \cb

If $\gamma$ is an $m$-fold cover of a simple orbit, then we define the {\em multiplicity $m(\gamma)$ of $\gamma$} to be $m$.
The differential $\bdry$ of $CC^{L}(M,\alpha,J)$ is given by
\begin{equation} \label{differential definition}
\bdry \gamma= \sum_{\gamma'\in \mathcal{P}_{\alpha,\op{good}}^L} \#(\mathcal{M}^{\op{ind}=1,\op{cyl}}_{J}(\gamma,\gamma')/\R) {1\over m(\gamma')} \cdot \gamma',
\end{equation}
where $\#$ refers to the signed count using the coherent orientation system from \cite{BM}; see Section~\ref{subsection: orientation d squared} for more details on orientations.

\begin{rmk} \label{rmk: integer coefficients}
Although there is a denominator ${1\over m(\gamma')}$ in the definition of $\bdry$, the coefficient of $\gamma'$ is always an integer; the same holds for coefficients in  chain maps. For example, when $(u,{\bf r})\in \mathcal{M}^{\op{ind}=1,\op{cyl}}_{J}(\gamma,\gamma')$ and $u$ is not multiply-covered \cb (which is automatic by Lemma~\ref{lemma: multiplicative}), \cb the contribution of all the $(u,{\bf r})$ with the same image towards
$$\#(\mathcal{M}^{\op{ind}=1,\op{cyl}}_{J}(\gamma,\gamma')/\R) {1\over m(\gamma')}$$
as we range over ${\bf r}$ is $\pm m(\gamma)$.
\end{rmk}

\begin{rmk}
\cb One can also define the differential by
$$\bdry' \gamma= \sum_{\gamma'\in \mathcal{P}_{\alpha,\op{good}}^L} \#(\mathcal{M}^{\op{ind}=1,\op{cyl}}_{J}(\gamma,\gamma')/\R) {1\over m(\gamma)} \cdot \gamma'.$$
The homologies with respect to $\bdry$ and $\bdry'$ are isomorphic only over $\Q$. \cb
\end{rmk}

\begin{rmk}
In view of Remark~\ref{rmk: integer coefficients}, $CC^L(M,\alpha,J)$ can be defined over $\Z$.  Similarly, the chain maps $\Phi_{X,\alpha,\overline{J}}$ can be defined over $\Z$.  {\em What is not defined over $\Z$ are the chain homotopy maps $K_\pm$.}
\end{rmk}

\begin{thm} \label{thm: differential}
If $J$ is generic, $\alpha$ is $L$-supersimple and $L$-monotone, and $(\alpha,J)$ is an $L$-supersimple pair, then $\partial$ is defined (i.e., the moduli space $\mathcal{M}^{\op{ind}=1,\op{cyl}}_{J}(\gamma,\gamma')/\R$ in Equation~\eqref{differential definition} is transversely cut out, compact, and oriented) and $\bdry^2=0$.
\end{thm}

\begin{proof}

\cb
We give a proof modulo orientation considerations that are postponed until Section~\ref{subsection: orientation d squared}.  We first state a couple of preliminary lemmas.

\begin{lemma}\label{lemma: no negative levels}
The moduli space $\mathcal{M}^i:=\mathcal{M}^{\op{ind}=i}_J(\gamma,\bs\gamma')$, where $\bs\gamma'=(\gamma'_1,\dots,\gamma_\ell')$ and $\gamma,\gamma'_1,\dots,\gamma_\ell'\in \mathcal{P}_\alpha^L$, is empty if $i<0$ and consists of possibly multiple covers of trivial cylinders if $i=0$.
\end{lemma}

\begin{proof}[Proof of Lemma~\ref{lemma: no negative levels}]
A simply-covered $J$-holomorphic curve is regular by the genericity of $J$.  Hence there are no simply-covered curves of $\op{ind} \leq 0$ besides trivial cylinders.  Lemma~\ref{lemma: multiplicative} then implies that $\mathcal{M}^{i<0}$ is empty and $\mathcal{M}^0$ consists of possibly multiple covers of trivial cylinders.
\end{proof}

\begin{lemma} \label{lemma: differential transversality}
For $i=1,2$, the moduli space $\mathcal{M}^i:=\mathcal{M}^{\op{ind}=i,\op{cyl}}_{J}(\gamma,\gamma')$ is transversely cut out if $\gamma,\gamma'\in\mathcal{P}_\alpha^L$.  All the curves of $\mathcal{M}^1$ are simply-covered.
\end{lemma}

\begin{proof}[Proof of Lemma~\ref{lemma: differential transversality}]
When $i=1$, any $u \in \mathcal{M}^1$ is simply-covered by Lemma~\ref{lemma: multiplicative} and hence is transversely cut out.

When $i=2$, either $u\in\mathcal{M}^2$ is simply-covered or double covers a simple $u'$ with $\op{ind}(u')=1$ and no branch points. By Theorem~\ref{thm: immersion}, $u'$ is an immersion, so $u$ is also an immersion. By Theorem~\ref{thm: automatic transversality}, $u$ is regular.
\end{proof}

By Lemmas~\ref{lemma: no negative levels} and \ref{lemma: differential transversality} and the additivity of $\op{ind}$, $\mathcal{M}^{\op{ind}=1,\op{cyl}}_{J}(\gamma,\gamma')/\R$ is compact and transversely cut out. This implies the finiteness of $\#(\mathcal{M}^{\op{ind}=1,\op{cyl}}_{J}(\gamma,\gamma')/\R)$.

For $\bdry^2=0$, we consider $\op{ind}=2$ moduli space $\mathcal{M}:=\mathcal{M}^{\op{ind}=2,\op{cyl}}_{J}(\gamma,\gamma')$, which is transversely cut out by Lemma~\ref{lemma: differential transversality}. By Lemma~\ref{lemma: no negative levels}, every element $u^\infty$ of $\bdry \mathcal{M}$ is a two-level building $u_1\cup u_2$, where $u_2$ is above $u_1$, $u_i$, $i=1,2$, $\op{ind}(u_i)=1$, and $u_1, u_2$ are a priori not necessarily cylindrical.  By the monotonicity of $\alpha$, every holomorphic plane in $u^\infty$ must have $\op{ind}\geq 3$, which is a contradiction. This implies that all the irreducible components of $u^\infty$ are cylinders.

It remains to consider the case where $u_2$ is a curve from $\gamma$ to $\gamma''$ and $u_1$ is a curve from $\gamma''$ to $\gamma'$, where $\gamma''$ is a bad Reeb orbit.  By Lemma~\ref{lemma: differential transversality}, $u_1$ and $u_2$ are simply-covered.
Since $\gamma''$ is an even multiple cover of a negative hyperbolic orbit, the gluing occurs in canceling pairs.  (Recall from \cite{BM} that the orientation is reversed under a $\Z/2$-deck transformation if $\gamma''$ is an even multiple of a negative hyperbolic orbit.)
\cb
\end{proof}

\subsection{Chain maps}  \label{subsection: chain maps}

Let $(X^4,\alpha)$ be a compact, connected, exact symplectic cobordism, where $d\alpha$ is symplectic, $\bdry X=M_+-M_-$, and $\alpha_\pm= \alpha|_{M_\pm}$ is a contact form on $M_\pm$.  Let $(\widehat X,\widehat\alpha)$ be the completion of $(X,\alpha)$, obtained by attaching the symplectization ends $[0,\infty)\times M_+$ and $(-\infty,0]\times M_-$, and let $\overline{J}$ be an almost complex structure which tames $(\widehat X,\widehat\alpha)$ and which restricts to $\alpha_\pm$-tame almost complex structures $J_\pm$ at the positive and negative ends.
Also let $\mathcal{M}_{\overline{J}}(\boldsymbol\gamma_+,\boldsymbol\gamma_-)$ be the moduli space of $\overline{J}$-holomorphic maps in $\widehat X$ from $\boldsymbol\gamma_+$ to $\boldsymbol\gamma_-$ with markings, defined in a manner analogous to that of $(\R\times M,J)$.

\cb

\begin{thm} \label{thm: chain maps}
If $\overline{J}$ is generic, $(M_+,\alpha_+)$ and $(M_-,\alpha_-)$ are $L$-supersimple and $L$-monotone, and $(\alpha_+,J_+)$ and $(\alpha_-,J_-)$ are $L$-supersimple pairs, then $(X,\alpha,\overline{J})$ induces a chain map \cb
\begin{gather*}
\Phi_{(X,\alpha,\overline{J})}: CC^L(M_+,\alpha_+,J_+)\to CC^L(M_-,\alpha_-,J_-),\\
\gamma_+\mapsto \sum_{\gamma_-\in\mathcal{P}_{\alpha_-,\op{good}}^{L}}\#(\mathcal{M}^{\op{ind}=0,\op{cyl}}_{\overline{J}}(\gamma_+,\gamma_-)){1\over m(\gamma_-)}\cdot\gamma_-.
\end{gather*}
In particular, the moduli spaces $\mathcal{M}^{\op{ind}=0,\op{cyl}}_{\overline{J}}(\gamma_+,\gamma_-)$ are compact, transversely cut out, and oriented.  \cb
\end{thm}

\begin{proof}
The proof is given in three steps.

\s\n
{\em Step 1.}
We first start with the following analogs of Lemmas~\ref{lemma: no negative levels} and \ref{lemma: differential transversality}:

\begin{lemma}\label{lemma: no negative levels 2}
The moduli space $\mathcal{M}^{\op{ind}=i}_{\overline{J}}(\gamma_+,\bs\gamma_-)$, where $\bs\gamma_- = ( \gamma_{-,1},\dots,\gamma_{-,\ell})$, $\gamma_+\in \mathcal{P}^{L_+}_{\alpha_+}$, and $\gamma_{-,1},\dots,\gamma_{-,\ell}\in \mathcal{P}^{L_-}_{\alpha_-}$, is empty if $i<0$.
\end{lemma}

The proof of Lemma~\ref{lemma: no negative levels 2} is the same as that of Lemma~\ref{lemma: no negative levels}.

\begin{lemma}\label{lemma: chain maps transversality}
For $i=0,1$, the moduli space $\mathcal{M}^i:=\mathcal{M}^{\op{ind}=i,\op{cyl}}_{\overline{J}}(\gamma_+,\gamma_-)$ is transversely cut out for $\gamma_+\in \mathcal{P}^{L_+}_{\alpha_+}$ and $\gamma_-\in \mathcal{P}^{L_-}_{\alpha_-}$, away from curves $u\in \mathcal{M}^0$ that are even multiple covers of cylinders between (odd multiples of) negative hyperbolic orbits. All the curves of $\mathcal{M}^1$ are simply-covered.
\end{lemma}

\begin{proof}[Proof of Lemma~\ref{lemma: chain maps transversality}]
When $i=1$, any $u\in \mathcal{M}^1$ is simply-covered by Lemma~\ref{lemma: multiplicative} and hence regular.

Suppose $i=0$.  Since simple curves in $\mathcal{M}^0$ are regular, it remains to prove that $v\in \mathcal{M}^0$ which is an unbranched $m$-fold cover of a simple $u$ from $\gamma$ to $\gamma'$ is regular. By Theorem~\ref{thm: immersion}, we may assume that $u$ is immersed.  Let $\dot F_1$ and $\dot F_0$ be the domains of $v$ and $u$ and let $\pi: \dot F_1\to \dot F_0$ be the covering map such that $v= u\circ\pi$.

Let $N$ be a normal bundle of $v$ as described in Section~\ref{section: automatic transversality} and let $\tau$ be a trivialization of $N$ so that either $\mu_\tau(\gamma)=\mu_\tau(\gamma')=0$ or $\mu_\tau(\gamma)=\mu_\tau(\gamma')=1$.

Suppose that $\mu_\tau(\gamma)=\mu_\tau(\gamma')=1$. If $m$ is even, then $\gamma_+$ and $\gamma_-$ are even multiples of $\gamma$ and $\gamma'$ which are odd multiples of negative hyperbolic orbits. This contradicts our assumption. Hence $m$ is odd and $\#\Gamma_0(v)=0$. Theorem~\ref{thm: automatic transversality} then implies that $v$ is regular.

Next suppose that $\mu_\tau(\gamma)=\mu_\tau(\gamma')=0$.  By assumption $u$ is regular.  If $v$ is not regular, then $0\not=\zeta\in \op{ker} (D_v^N)^*$ satisfies Equations~\eqref{eqn: 2a1} and \eqref{eqn: 2a2}, which implies the asymptotic condition
\begin{equation} \label{eqn: equal to zero}
\op{wind}_\tau(\zeta,p_+)=\op{wind}_\tau(\zeta,p_-)=0,
\end{equation}
where $p_+$ (resp.\ $p_-$) is the puncture corresponding to the positive (resp.\ negative) end.  We now consider the ``trace map''
$$\pi_*: \op{ker} (D_v^N)^* \to \op{ker} (D_u^N)^*$$
which maps $\zeta$ to $\pi_*\zeta$, where $\pi_*\zeta(x)= \sum_{y\in \pi^{-1}(x)} \zeta(y)$. By Equation~\eqref{eqn: equal to zero}, $\pi_*\zeta\not=0$, which is a contradiction. This proves the regularity of $v$.
\end{proof}

Let $(u,{\bf r})\in \mathcal{M}^{\op{ind}=0,\op{cyl}}_{\overline{J}}(\gamma_+,\gamma_-)$.  If $u$ is not multiply-covered, then the contribution of all $(u,{\bf r})$ towards $\#(\mathcal{M}^{\op{ind}=0,\op{cyl}}_{\overline{J}}(\gamma_+,\gamma_-)){1\over m(\gamma_-)}$ as we range over ${\bf r}$ is $\pm m(\gamma_+)$.  On the other hand, if $u$ is an $a$-fold cover of a simple curve, the contribution is $\pm m(\gamma_+)/a$.

\s\n {\em Step 2:  $\Phi=\Phi_{(X,\alpha,\overline{J})}$ is defined, i.e., $\mathcal{M}^0$ is transversely cut out and compact, and hence finite.}  By Lemma~\ref{lemma: chain maps transversality}, $\mathcal{M}^0$ is transversely cut out. By the SFT compactness theorem~\cite{BEHWZ}, if $\mathcal{M}^0$ is not compact, then there is a sequence $u_1,u_2,\dots \in \mathcal{M}^0$ that limits to an SFT building $u^\infty$ of $\geq 2$ levels.  By Lemma~\ref{lemma: no negative levels} all the levels of $u^\infty$ that map to $\mathbb{R} \times M_\pm$ have $\op{ind}> 0$ and by Lemma~\ref{lemma: no negative levels 2} the level of $u^\infty$ that maps to the cobordism has $\op{ind}\geq 0$. Hence $u^\infty$ is a $1$-level building in $\mathcal{M}^0$ and we get a contradiction.

\s\n {\em Step 3: $\bdry\circ \Phi=\Phi\circ \bdry$.} First, $\mathcal{M}^1$ is regular by Lemma~\ref{lemma: chain maps transversality}.  We consider the boundary $\bdry \mathcal{M}^1$ of $\mathcal{M}^1$. If $u^\infty\in \bdry \mathcal{M}^1$, then
by Lemmas~\ref{lemma: no negative levels} and \ref{lemma: no negative levels 2} we have $u^\infty=v_0\cup v_1$ or $v_{-1}\cup v_0$, where $v_0$ is a cylinder that maps to $\widehat{X}$ and $v_i$, $i\not=0$,
is a cylinder that maps to $\R\times M_\pm$; the levels are arranged from bottom to top as $i$ increases; $\op{ind}(v_i)=1$, $i\not=0$, and $\op{ind}(v_0)=0$;
and by \cb Lemma~\ref{lemma: differential transversality} \cb $v_i$, $i\not=0$, is simple and $v_0$ may be multiply-covered. (By the monotonicity of $\alpha_\pm$, every holomorphic plane in $u^\infty$ must have $\op{ind}\geq 3$.  Since $\op{ind}>0$ for every nontrivial curve in $\R\times M_\pm$ and $\op{ind}\geq 0$ for every nontrivial curve in $\widehat{X}$, there can be no planes. This implies that each irreducible component of $u^\infty$ is a cylinder.)
By Lemmas~\ref{lemma: differential transversality} and \ref{lemma: chain maps transversality}, each component of $u^\infty$ is regular since $\gamma_+$ and $\gamma_-$ are good by assumption and one of the ends of $v_0$ cannot limit to a bad orbit.

Suppose that $u^\infty=v_0\cup v_1$; the considerations for $u_\infty=v_{-1}\cup v_0$ are analogous.
Suppose that $v_0$ limits to a bad orbit at its positive end.  Then, as above, $v_0$ is not an even multiple cover of a cylinder between (odd multiples of) negative hyperbolic orbits since otherwise $\gamma_-$ is bad.  By \cite{BM} this is sufficient to ensure that there is an even number of ways to glue $v_0$ and $v_1$ and they result in canceling pairs. 

The orientation considerations are similar to those of Theorem~\ref{thm: differential}.
\end{proof}
\cb

\subsection{Definition of $HC(\mathcal{D})$} \label{subsection: def}

We are now in a position to define the {\em cylindrical contact homology group $HC(\mathcal{D})$.}
Given $(M,\xi)$, let
$$\mathcal{D}=(\alpha,\{L_i\},\{\vp_i\},\{J_i\},\{\overline{J}_i\} )$$
be a tuple consisting of a nondegenerate $\alpha$ for $(M,\xi)$ with no contractible Reeb orbits, sequences $\{L_i\}$ and $\{\vp_i\}_{i=1}^{\infty}$ given at the end of Section~\ref{section: elimination of elliptic orbits}, a sequence $\{(\vp_i\alpha,J_i)\}$ of $L_i$-supersimple pairs, and a sequence $\{\overline{J}_i\}$ of generic almost complex structures that are tamed by exact symplectic cobordisms on $\R\times M$ from $\vp_i\alpha$ to $\vp_{i+1}\alpha$ and that agree with $J_i$ and $J_{i+1}$ at the positive and negative ends.

The symplectic cobordisms together with $\overline{J}_i$ induce chain maps
$$\Phi_i: CC^{L_i}(\vp_i\alpha,J_i)\to CC^{L_{i+1}}(\vp_{i+1}\alpha,J_{i+1}),$$
and we define $HC(\mathcal{D})$ as the direct limit of the induced maps $(\Phi_i)_*$ on homology.

\section{The evaluation map} \label{section: evaluation map}

The goal of this section is to introduce and discuss the properties of the evaluation map. \cb We consider an $L$-supersimple pair $(\alpha, J)$ as defined in Section~\ref{subsection: L-supersimple pairs}. \cb  Suppose that $\gamma$ is a positive hyperbolic orbit of action $<L$; the situation of $\gamma$ negative hyperbolic is similar.

Let $\gamma$ be an $m(\gamma)$-fold cover of a simple orbit $\gamma_0$, let $\gamma_0\times D^2_{\delta_0/3}$ be a neighborhood of $\gamma_0$, where $\delta_0$ is as in Section~\ref{subsection: L-supersimple pairs}, and let $(\R/\Z)\times D^2_{\delta_0/3}\to \gamma_0\times D^2_{\delta_0/3}$ be its $m(\gamma)$-fold cover with coordinates $(t,z=x+iy)$ such that $\{z=0\}$ corresponds to $\gamma$.  Also let $\R\times (\R/\Z)\times D^2_{\delta_0/3}$ be the cylinder over $(\R/\Z)\times D^2$ with coordinates $(s,t,z=x+iy)$.

Let
\begin{equation} \label{eqn: ordering eigenvalues}
\dots \leq \lambda_{-2}\leq \lambda_{-1} < 0<\lambda_1\leq \lambda_2\leq \dots
\end{equation}
be the eigenvalues of $A=-j_0{\bdry\over \bdry t}-S$ with $S= \begin{pmatrix}  0 & \varepsilon \\ \varepsilon & 0 \end{pmatrix}$ and let $f_i(t)$ be an eigenfunction corresponding to $\lambda_i$ with $L^2$-norm $1$ so that $\{f_i(t)\}_{i\not=0}$ is an orthonormal basis for $L^2(S^1,\R^2)$.


\begin{fact}
If $J\in \mathcal{J}_\star^{<L,\tiny\mbox{reg}}$ and
$$u:(-\infty,R]\times S^1\to \R\times M$$
is a $J$-holomorphic half-cylinder which is negatively asymptotic to $\gamma$ for some $R$, then there exists $R'\ll 0$ such that $u(s,t)$, after reparametrization of the domain and lifting to the $m(\gamma)$-fold cover $\R\times (\R/\Z)\times D^2_{\delta_0/3}$, can be written on $\{s\leq R'\}$ as:
\begin{equation} \label{eqn: Fourier}
\widetilde u(s,t)=\left(s,t,\sum_{i=1}^\infty c_i e^{\lambda_i s} f_i(t)\right)\in \R\times (\R/\Z)\times D^2_{\delta_0/3}.
\end{equation}

\end{fact}

By abuse of notation we will usually not distinguish between $u$ and $\widetilde u$.
We refer to Equation~\eqref{eqn: Fourier} as the ``Fourier series'' for $u$. The real constants $c_i$ will be referred to as the ``Fourier coefficients''.  We also define the {\em order $k$ Fourier polynomial} $P_k(u)$ of $u$ as:
$$P_k(u)= \sum_{i=1}^k c_i e^{\lambda_i s} f_i(t).$$

Let $J\in \mathcal{J}_\star^{<L,\tiny\mbox{reg}}$ and let $\mathcal{A}_{\alpha_+}(\gamma_+),\mathcal{A}_{\alpha_+}(\gamma'_+) < L$.  We define the {\em order $k$ evaluation map at the negative end}
$$ev^k_-(\gamma_+,\gamma_+',J):\mathcal{M}^{\op{cyl}}_{J}(\gamma_+,\gamma_+')\to \R^k$$
$$ (u,r_+,r_-)\mapsto (c_1,\dots,c_k),$$
where $u$ agrees with a $J$-holomorphic half-cylinder $(-\infty,R]\times S^1\to \R\times M$ which is negatively asymptotic to $\gamma_+'$ for some $R$ and has Fourier coefficients $c_1,\dots,c_k$. Here we parametrize $\R/\Z$ such that the asymptotic marker $r_-$ at the negative end corresponds to $t=0$.

\begin{fact}
The map $ev^k_-(\gamma_+,\gamma_+',J)$ is smooth.
\end{fact}

The moduli space $\mathcal{M}^{\op{cyl}}_{J}(\gamma_+,\gamma_+')$ admits the usual $\R$-translation which corresponds to the $\R^+$-action on $\R^k$ given by:
$$(c_1,\dots,c_k)\mapsto (c_1 e^{\lambda_1 s},\dots, c_k e^{\lambda_k s}).$$
{\em Provided there is no $u$ with $(c_1,\dots,c_k)=0$,}  $ev^k_-(\gamma_+,\gamma_+',J)$ descends to the quotient
\begin{equation}
\widetilde{ev}^k_-(\gamma_+,\gamma_+',J):\mathcal{M}^{\op{cyl}}_{J}(\gamma_+,\gamma_+')/\R \to (\R^k-\{0\})/\R^+\simeq S^{k-1}.
\end{equation}

\begin{rmk} \label{identification}
We will make more precise the identification $(\R^k-\{0\})/\R^+\simeq S^{k-1}$: Each path
$$\mathfrak{c}: \R\to \R^k, \quad s\mapsto (c_1 e^{\lambda_1 s},\dots, c_k e^{\lambda_k s})$$
\cb with $(c_1,\dots,c_k)\not =0$ \cb is transverse to the spheres $S^{k-1}_r=\{ |x|\in \R^k ~|~ |x|=r\}$ since $\langle \mathfrak{c}'(s), \mathfrak{c}(s)\rangle >0$ for all $s\in \R$. We will take the representative of $\mathfrak{c}$ to be its intersection with $S^{k-1}_r$, where $r$ is taken to be $1$ unless specified otherwise.
\end{rmk}

The maps
$$ev^k_+(\gamma_-',\gamma_-,J):\mathcal{M}^{\op{cyl}}_{J}(\gamma_-',\gamma_-)\to \R^k,$$
$$\widetilde{ev}^k_+(\gamma_-',\gamma_-,J): \mathcal{M}^{\op{cyl}}_{J}(\gamma_-',\gamma_-)/\R\to S^{k-1}$$
are defined similarly.

The main results of this subsection concern the transversality properties of the order $k$ evaluation map, which generalizes \cite[Proposition~3.2]{HT2} in a special case.
Let
$$\pi_\R: \mathcal{M}(\gamma_+,\gamma_+')\to \mathcal{M}(\gamma_+,\gamma_+')/\R$$
be the quotient map by translations in the $s$-direction.

\begin{thm}\label{thm: transversality of evaluation map}
Given $J\in \mathcal{J}^{<L,\tiny\mbox{reg},\delta_0}_\star$, a compact domain \cb $K\subset \mathcal{M}_{J}^{\op{ind}=\ell,\op{cyl},s}(\gamma_+,\gamma_+')/\R$, and a submanifold $\widetilde Z\subset S^{k-1}$, \cb there exist an arbitrarily close $J'\in \mathcal{J}^{<L,\tiny\mbox{reg},\delta_0'}_\star$ with $\delta'_0<\delta_0$ and a compact domain $K'\subset \mathcal{M}_{J'}^{\op{ind}=\ell,\op{cyl},s}(\gamma_+,\gamma_+')/\R$ such that:
\begin{enumerate}
\item $\pi^{-1}_\R(K')$ contains all the elements of $\mathcal{M}_{J'}^{\op{ind}=\ell,\op{cyl},s}(\gamma_+,\gamma_+')$ that are sufficiently close to $\pi^{-1}_\R(K)$; and
\item \cb if $k\geq \ell$, then the restriction of the evaluation map $ev^k_-(\gamma_+,\gamma_+',J')$ to $\pi^{-1}_\R(K')$ descends to $\widetilde{ev}^k_-(\gamma_+,\gamma_+',J')$, which is transverse to $\widetilde Z$. \cb
\end{enumerate}
\end{thm}

\begin{rmk}
The proof is modeled on but is substantially easier than that of \cb \cite[Proposition~3.2]{HT2}. \cb  This is due to the fact that the $J$-holomorphic equation is linear near each $\R\times \gamma$.  This allows us to dispense with the quadratic estimates.
\end{rmk}

\begin{proof}
\cb
Let $\mathcal{J}_{J}$ be the subset of $\mathcal{J}_\star^{<L,reg,\delta_0}$ consisting of $J'$ that are $C^1$-close to $J$. Also let $\pi:\widetilde{\mathcal{M}}\to \mathcal{J}_{J}$ be the bundle with fiber $\pi^{-1}(J')=\mathcal{M}^{\op{cyl}}_{J'}(\gamma_+,\gamma_+')/\R$, $\widetilde \pi: \mathcal M \to \widetilde{\mathcal{M}}$ be the bundle with fiber $\widetilde \pi ^{-1} ([u], J') = \pi_{\R, J'}^{-1}([u])$, where
$$\pi_{\R, J'}: \mathcal{M}^{\op{cyl}}_{J'}(\gamma_+,\gamma_+')\to \mathcal{M}^{\op{cyl}}_{J'}(\gamma_+,\gamma_+')/\R$$
is the quotient map by translations in the $s$-direction, and let $\mathcal{M}^s\subset \mathcal{M}$ and $\widetilde{\mathcal{M}}^s\subset \widetilde{\mathcal{M}}$ be the subsets of simple curves. Define the map
$$Ev^k_-(\gamma_+,\gamma_+'): \mathcal M^s \to \R^k,$$
$$(u,J')\mapsto ev^k_-(\gamma_+,\gamma_+',J')(u).$$
Choose a smooth section $\mathfrak s$ of $\widetilde \pi: \mathcal M \to \widetilde{\mathcal{M}}$ and denote $\mathcal N^s := \mathfrak s(\widetilde{\mathcal{M}}^s)$.  

We will show that the restriction $Ev^k_-(\gamma_+,\gamma_+')|_{\mathcal N^s}$ is regular at all points over $K$; hence $Ev^k_-(\gamma_+,\gamma_+')|_{\mathcal N^s}$ is transverse to any submanifold $Z\subset \R^k$.  Theorem~\ref{thm: transversality of evaluation map} then follows from Sard's theorem: To show that $ev^k_-(\gamma_+,\gamma_+',J')$ avoids the origin ${\bf 0}\in \R^k$, let $Z = \{{\bf 0}\}$. Then $Ev^k_-(\gamma_+,\gamma_+')|_{\mathcal N^s}^{-1}({\bf 0})$ has codimension $k$. Since each fiber of $\pi\circ \widetilde \pi$ has dimension $\ell-1\leq k-1$, for a generic $J' \in \mathcal J_J$,
$$(\pi\circ \widetilde \pi)^{-1}(J') \cap Ev^k_-(\gamma_+,\gamma_+')|_{\mathcal N^s}^{-1}({\bf 0}) = \emptyset,$$
and $ev^k_-(\gamma_+,\gamma_+',J')$ descends to $\widetilde{ev}^k_-(\gamma_+,\gamma_+',J')$.
To show that $\widetilde{ev}^k_-(\gamma_+,\gamma_+',J')$ is transverse to $\widetilde Z$, we choose $ Z$ to be the preimage of $\widetilde Z$ under the projection $\R^k \setminus \{{\bf 0}\} \to S^{k-1}$.
\cb

Let $u\in K$ and $(\xi,Y)\in T_{(u,J)}\mathcal{M}^s$. With respect to the usual coordinates $(t,x,y)$ on $(\R/\Z)\times D^2_{\delta_0}$, the almost complex structure $J$ maps ${\bdry\over \bdry s}\mapsto {\bdry\over \bdry t}+X_H$ and ${\bdry\over \bdry x}\mapsto {\bdry\over \bdry y}$ on $D^2_{\delta_0/3}$.  We assume that $Y$ corresponds to the path $J+\tau Y$, $\tau\in[-\varepsilon,\varepsilon]$, which maps ${\bdry\over \bdry s}\mapsto {\bdry\over \bdry t}+X_H$ and leaves $TD^2_{\delta_0/3}$ invariant.  We write $j_0+\tau Y_0$ for the restriction of $J+\tau Y$ to $TD^2_{\delta_0/3}$.

If $\overline\bdry_{J+\tau Y}(u+\tau\xi)=0$ for $\tau\in[-\varepsilon,\varepsilon]$, then we claim that $(\xi,Y)$ satisfies
\begin{equation}\label{eqn: perturbation}
{\bdry \xi\over \bdry s} +j_0 {\bdry \xi\over \bdry t} -j_0\left(X_H(\xi)+Y_0(s,t,u(s,t))\left({\bdry u\over \bdry s}\right)\right)=0.
\end{equation}
Since we are only concerned with the negative end of $u$, we assume without loss of generality that $u$ is graphical over $(-\infty,0]\times (\R/\Z)$, has image in $(-\infty,0]\times (\R/\Z)\times D^2_{\delta_0/3}$, and is written as $(s,t,u(s,t))$.  Similarly we write $\xi$ as $(s,t,\xi(s,t))$.
The almost complex structure $J+\tau Y$ maps
$$\left(1,0,{\bdry (u+\tau \xi)\over \bdry s}\right)\mapsto \left(0,1, j_0{\bdry (u+\tau \xi) \over \bdry s}+ X_H(u+\tau\xi) +\tau Y_0 \left({\bdry (u+\tau \xi)\over \bdry s}\right)\right),$$
which must equal $(0,1,{\bdry (u+\tau \xi)\over \bdry t})$. Equating the third coordinates, using $j_0{\bdry u\over \bdry s}+X_H(u)= {\bdry u\over \bdry t}$, and differentiating with respect to $\tau$, we obtain Equation~\eqref{eqn: perturbation}.

\s\n
{\bf Case 1.}
First suppose that $u$ does not intersect $\R\times \gamma'_+$.

We solve for $(\xi,Y)\in T_{(u,J)}\mathcal{M}^s$ in Equation~\eqref{eqn: perturbation}, i.e., on the negative end, where
\begin{equation} \label{kiku}
j_0  Y_0 (s,t,u(s,t))\left({\bdry u\over \bdry s}\right)= \mu(s) f_i(t)
\end{equation}
and $\xi(s,t)=\rho(s) f_i(t)$ for $i\geq 1$, $\mu$ and $\rho$ are smooth in $s$, $\mu(s)\geq 0$ is a function with support on $[R',R]$ for some $R',R\in\R$ and total integral $1$, and $\rho(s)=0$ for $s>R$.

Equation~\eqref{eqn: perturbation} then becomes
$$\rho'(s) f_i(t) - \rho(s) \lambda_i f_i(t) -\mu(s) f_i(t)=0,$$
$$\rho'(s) -\lambda_i\rho(s)=\mu(s).$$
We then pick the solution
\begin{equation} \label{heatwave}
\rho(s)= e^{\lambda_i s} \int_s^R e^{-\lambda_i \sigma}\mu(\sigma)d\sigma.
\end{equation}
Then $\xi(s,t)=\rho(s) f_i(t)$ is equal to $0$ on $s=R$ and can be written as
\begin{equation} \label{heatwave2}
\xi(s,t)= c e^{\lambda_i s} f_i(t),\quad c\geq e^{-\lambda_i R}/\lambda_i
\end{equation}
on $s\leq R'$. (Note that $R$ is a large negative number.)

We then consider the extension of $(\xi,Y)$. We extend $\xi$ to all of the domain of $u$ by setting $\xi=0$. Let $u_T$ be the curve obtained from $u$ by translating in the $s$-direction by $T$ units. By the following claim, for sufficiently negative $R'$ and $R$, $Y$ can be extended from the annulus $\op{Im} u|_{R'\leq s\leq R}$ to all of $\R\times M$ so that $Y$ is $s$-invariant and has support on
$$\coprod_{T\in \R} \op{Im} u_T|_{R'\leq s\leq R}\subset \R\times (\R/\Z)\times (D^2_{\delta_0/3}-D^2_{\delta_0'/3})$$
for some $0<\delta_0'<\delta_0$.

\begin{claim}\label{claim: injective}
For $s_0\ll 0$,  $u|_{s=s_0}(t)= \sum_{i=1}^\infty c_i e^{\lambda_i s_0} f_i(t)$ is an embedding.
\end{claim}

\begin{proof}[Proof of Claim~\ref{claim: injective}]
The reason the claim is not trivial is that $\gamma_+'$ may be an $m$-fold cover of a simple orbit $\gamma_0$.  Since $u$ is simple by assumption, the negative end of $u$ is not multiply-covered.  Hence there exists an integer $j>0$ such that:
\begin{itemize}
\item for each $c_i\not=0$ with $i< j$, $f_i(t)$ is an $a$-fold cover of an asymptotic eigenfunction for $\gamma_0^{m/a}$, where $a>1$, and
\item $c_j\not=0$ and $f_j(t)$ is a $b$-fold cover of an asymptotic eigenfunction for $\gamma_0^{m/b}$, where $(a,b)=1$.
\end{itemize}
By an explicit calculation of asymptotic eigenfunctions (i.e., Equations~\eqref{eqn: eigenfunction1}--\eqref{eqn: eigenfunction4}), it follows that $u|_{s=s_0}$ is an embedding for $s_0\ll 0$.
\end{proof}

Since $u$ does not intersect $\R\times \gamma'_+$, the support of $Y$ intersects $u$ only near the negative end of $u$. \cb  Hence the pair $(\xi,Y)$, described above and satisfying Equation~\eqref{heatwave2}, is indeed an element of $T_{(u,J)} \mathcal{M}^s$.  This implies that $Ev^k_-(\gamma_+,\gamma_+')$ is regular at $(u,J)$.  Since all the perturbations $(\xi,Y)$ constructed above have the form that $\xi$ is supported near the negative end, it follows that such $\xi$'s are independent from the tangent vector $\zeta$ at $u$ corresponding to the translation in the $s$-direction.  Hence $Ev^k_-(\gamma_+,\gamma_+')|_{\mathcal N^s}$ is regular at $(u,J)$  and the theorem holds in this case. \cb

\s\n
{\bf Case 2.} Suppose that $u:\dot F\to \R\times M$ nontrivially intersects $\R\times \gamma'_+$. Let $Y$ be as in Case 1 and let $u^* Y$ be the pullback of $Y$ to $u^*\op{End}(T(\R\times M))$.  Then $u^* Y$ can be written as $Y'+ Y''$, where $Y'$ is supported on the negative end of $\dot F$ and $Y''$ is supported on a neighborhood of $\Theta\subset \dot F$. Here $\Theta$ is the preimage of $u(\dot F)\cap (\R\times \gamma'_+)$.

Let $\xi'$ be the solution to $\overline\bdry_{J+\tau Y'}(u+\tau\xi')=0$ up to first order in $\tau$, as constructed in Case 1 (the notation in Case 1 is $(\xi,Y)$) and whose negative end satisfies Equation~\eqref{heatwave2}.

Let $\xi''$ be the solution to $\overline\bdry_{J+\tau Y''}(u+\tau\xi'')=0$ up to first order in $\tau$, which is $L^2$-orthogonal to the kernel of the linearized $\overline\bdry$-operator $D_{u,J}$.  We now estimate that the solution $\xi''$ corresponding to $Y''$ is much smaller than $\xi'$ along $s=R'$ if $R'\ll R\ll 0$.  We will use Morrey space norms which are defined in Section~\ref{subsection: banach spaces} and ideas that are used in Section~\ref{subsection: linearized section s_0}. The constant $c>0$ may change from line to line. We first observe that $\xi''= D_{u,J}^{-1}(\zeta)$, where $\|\zeta\|\leq c\|Y\|$.  Since $D_{u,J}^{-1}$ is bounded, $\|\xi''\|_*\leq c \|Y\|$. Then, by Lemma~\ref{lemma: sobolev embedding}, $|\xi''|_{C^0}\leq c \|Y\|$. Hence $|\xi''|_{s=R}$ has the same order of magnitude as $|\xi'|_{C^0}$.  However, since $\xi''$ decays exponentially as $s\to -\infty$,
\begin{equation} \label{rat zapper}
|\xi''|_{s=R'}\leq c e^{-\lambda(R-R')} |\xi''|_{s=R},
\end{equation}
where $\lambda=\op{min}(\lambda_1,|\lambda_{-1}|)$. This implies that $|\xi''|_{s=R'}\ll |\xi'|_{s=R'}$.

Since $\xi'+\xi''$ is a solution corresponding to $Y'+Y''$, the theorem follows.
\end{proof}

Let us abbreviate \cb $\mathcal{M}^k = \mathcal{M}_{J}^{\op{ind}=k,\op{cyl}}(\gamma_+,\gamma_+')$ \cb and $\widetilde{ev}^k_-=\widetilde{ev}^k_-(\gamma_+,\gamma_+',J)$.  We also use the superscript ``$\mbox{sing}$'' to denote the subset of $\mathcal{M}^k$ consisting of curves with singularities.

\begin{thm} \label{thm: niceness of ev}
For a generic $J\in \mathcal{J}^{<L,\tiny\mbox{reg}}_\star$, $\widetilde{ev}^k_-:\mathcal{M}^k/\R\to S^{k-1}$, \cb $k\geq 1$, \cb satisfies the following:
\begin{enumerate}
\item The restriction of $\widetilde{ev}^k_-$ to $(\mathcal{M}^k -\mathcal{M}^{k,\op{sing}})/\R$ is an immersion.
\item $\widetilde{ev}^k_-(\mathcal{M}^{k,\op{sing}}/\R)$ has codimension at least $2$ and is disjoint from $(0,\dots, 0,\pm 1)$.
\item In particular, $(0,\dots,0,\pm 1)$ are regular values of $\widetilde{ev}^k_-$.
\end{enumerate}
\end{thm}

Note that $\widetilde{ev}^k_-$ exists by Theorem~\ref{thm: transversality of evaluation map}(2), since $\ell=k$ in our case.

\begin{proof}
(1) Let $u\in \mathcal{M}^k-\mathcal{M}^{k,\op{sing}}$. \cb Since $u$ is immersed, if $\op{ind}(u)=k\geq 2$, then $u$ is regular by Theorem~\ref{thm: automatic transversality}. If $k=1$, then $u$ is regular by Lemma~\ref{lemma: differential transversality}. \cb By the argument of Lemma~\ref{lemma: good basis for coker}, there exists a basis $\{e_1,\dots,e_k\}$ for $\ker D_u$, such that the negative ends of $e_i$ are of the form
$$e_i(s,t)=e^{\lambda_i s} f_i(t) \quad \mbox{modulo } f_{k+1}, f_{k+2},\dots.$$
This implies the surjectivity of $(\widetilde{ev}^k_-)_*:T_{[u]}(\mathcal{M}^k/\R)\to T_{\widetilde{ev}^k_-([u])} S^{k-1}$.

(2) $\mathcal{M}^{k, s,\op{sing}}\subset \mathcal{M}^{k,s}$ has (real) codimension at least two by Theorem~\ref{thm: immersion}. Hence $\widetilde{ev}^k_-(\mathcal{M}^{k,s,\op{sing}}/\R)$ has codimension at least two and is disjoint from $\{(0,\dots, 0,\pm 1)\}$ by Theorem~\ref{thm: transversality of evaluation map}. (We will be a little sloppy: The issue here is to find a fixed radius $\delta_0'$ which works for all of $\overline{\mathcal{M}^k/\R}$. Strictly speaking, Theorem~\ref{thm: transversality of evaluation map} holds for a large compact subset $K\subset \mathcal{M}^{k,s,\op{sing}}/\R$.  We then apply Theorem~\ref{thm: transversality of evaluation map} to the strata of the boundary $\overline{\mathcal{M}^k/\R}$ to obtain Theorem~\ref{thm: niceness of compactification} below.  This in turn allows us to use $\mathcal{M}^{k, s,\op{sing}}/\R$ instead of $K$.)

It remains to consider the multiply-covered curves in $\mathcal{M}^k$.  Writing $\gamma_+=\gamma_{+,0}^b$ and $\gamma_+'=(\gamma_{+,0}')^b$, let $\mathcal{S}'=\mathcal{M}^{\op{ind}=a,\op{cyl}}_J(\gamma_{+,0},\gamma_{+,0}'),$
where $ab=k$, $b>1$, and let $\mathcal{S}\subset\mathcal{M}^k$ be the set of $b$-fold covers of curves in $\mathcal{S}'$.  By induction, suppose that
$$\widetilde{ev}^a_-(\gamma_{+,0},\gamma_{+,0}',J):\mathcal{S}'/\R\to S^{a-1}$$
satisfies (2) with $k$ replaced by $a$. (As the initial step of the induction, observe that if $k=2$ then $\mathcal{M}^k$ has no singular curves by Theorem~\ref{thm: immersion} and (2) holds.) Then $\widetilde{ev}^a_-$ can be ``lifted'' to $\widetilde{ev}^k_-:\mathcal{S}/\R \to S^{ab-1}$ as follows: Given $v\in \mathcal{S}$, suppose it is the $b$-fold cover of $u\in \mathcal{S}'$. Then $\widetilde{ev}^k_-(v)= i\circ\widetilde{ev}^a_-(u)$, where the inclusion $i:S^{a-1}\to S^{ab-1}$ is induced by \cb the map
\begin{equation} \label{eqn: inclusion}
\R^a\to \R^{ab}, \quad (x_1,\dots,x_a)\mapsto (0,\dots, 0,x_1,0,\dots,0,x_2,0,\dots,0,x_3,0,\dots).
\end{equation}
Here,
\begin{itemize}
\item the zeros are inserted in the same positions for all $(x_1,\dots,x_a)$ and
\item $0,\dots,0$ stands for zero or more $0$'s.
\end{itemize}
\cb While it is possible for $\widetilde{ev}^k_-(\mathcal{S})$ to pass through $(0,\dots,0,\pm 1)\in S^{k-1}$ \cb (this may happen if the last entry of the $ab$-tuple is $x_a$ in Equation~\eqref{eqn: inclusion}), \cb $\widetilde{ev}^k_-(\mathcal{S}^{\op{sing}})$ will not pass through $(0,\dots,0,\pm 1)$ by the induction hypothesis.  This proves (2).

(1) and (2) then imply (3).
\end{proof}

Finally we consider the extension $\overline{ev}^k_-:\overline{\mathcal{M}^k/\R}\to S^{k-1}$ of $\widetilde{ev}^k_-$ to the compactification of $\mathcal{M}^k/\R$.

\begin{thm} \label{thm: niceness of compactification}
For a generic $J\in \mathcal{J}^{<L,\tiny\mbox{reg}}_\star$ and $k\geq 2$, $\overline{ev}^k_-:\overline{\mathcal{M}^k/\R}\to S^{k-1}$ satisfies the following:
\begin{enumerate}
\item $\overline{ev}^k_-(\bdry(\mathcal{M}^k/\R))$ is disjoint from $(0,\dots,0,\pm 1)$.
\item If $\mathcal{S}$ is a (nonempty) stratum of $\bdry(\mathcal{M}^k/\R)$ consisting of $l$-level buildings $u_1\cup\dots\cup u_l$ with $\op{ind}(u_i)=a_i$, then \cb $0< a_i < k$, \cb $\op{dim}(\overline{ev}^k_-(\mathcal{S}))= a_1-1$, and $\mathcal{S}^{\op{sing}}$ has codimension at least $2$ in $\mathcal{S}$.
\end{enumerate}
\end{thm}

As usual, $u_1$ is the lowest level and $u_l$ is the highest level.

\begin{proof}\cb
By Lemma~\ref{lemma: no negative levels},  we have $0 < a_i  < k$. \cb
We apply Theorem~\ref{thm: niceness of ev} to the $\op{ind}=a_1$ moduli space $\mathcal{N}$ containing $u_1$ as given in the statement of (2) and the evaluation map $\widetilde{ev}^{a_1}_-: \mathcal{N}/\R\to S^{a_1-1}$.  We then define the map $\widetilde{ev}^k_-: \mathcal{N}/\R\to S^{k-1}$ \cb by composing $\widetilde{ev}^{a_1}_-$ with $i: S^{a_1-1}\to S^{k-1}$ of the form given by Equation~\eqref{eqn: inclusion} \cb to obtain (1) and (2). The details are left to the reader.
\end{proof}

\section{Chain homotopy} \label{section: chain homotopy}

\subsection{Chain homotopy} \label{subsection: chain homotopy}

Let $X^4$ be a compact connected $4$-manifold such that $\bdry X=M_+-M_-$, let $\alpha|_{M_\pm}$ be a contact form on $M_\pm$, and let $\{ \alpha^\tau\} _{0\leq \tau \leq 1}$ be a $1$-parameter family of $1$-forms on $X$ such that $d\alpha^\tau$ is symplectic and $\alpha^\tau |_{M_{\pm}}=\alpha_{\pm}$, for all $\tau \in [0,1].$ Also let $\{(\widehat {X}^\tau=\widehat{X}, \widehat {\alpha}^\tau)\}_{0\leq \tau \leq 1}$ be a $1$-parameter family of completions of $\{(X^\tau,  \alpha^\tau)\}_{0\leq \tau \leq 1}$ and let $\overline{J}^\tau$ be an almost complex structure which tames $(\widehat {X}^\tau,\widehat {\alpha} ^\tau)$ and which restricts to $\alpha_\pm$-tame almost complex structures $J_\pm$ near the positive and negative ends. \emph{From now on we further require that all contact forms are $L$-hypertight.}

\begin{thm} \label{thm: chain homotopy}
Suppose that $(M_-,\alpha_-)$ and $(M_+,\alpha_+)$ are $L$-supersimple and $L$-hypertight, $\{\overline{J}^\tau\}_{0\leq \tau \leq 1}$ is generic as a family, and $(\alpha_+,J_+)$ and $(\alpha_-,J_-)$ are $L$-supersimple pairs. Then the chain maps $\Phi_{(X^1,\alpha^1,\overline{J}^1)}$ and $\Phi_{(X^0,\alpha^0,\overline{J}^0)}$ defined in Theorem \ref{thm: chain maps} induce the same map on homology.
\end{thm}

To prove this theorem, we show there exist linear maps $$K_\pm: CC^L(M_+,\alpha_+,J_+)\to CC^L(M_-,\alpha_-,J_-)$$ such that for any $\gamma_+\in \mathcal{P}^L_{\alpha_+,\op{good}},$ one has \cb
\begin{equation} \label{equation: chain homotopy}
\Phi_{(X^1,\alpha^1,\overline{J}^1)}(\gamma_+)-\Phi_{(X^0,\alpha^0,\overline{J}^0)}(\gamma_+)=\partial_- K_-(\gamma_+)-K_+\partial_+(\gamma_+).
\end{equation}
Let us write
$$\mathcal{M}^i:= \coprod_{0\leq \tau \leq 1} \mathcal{M}^{\op{ind}=i,\op{cyl}}_{\overline{J}^\tau}(\gamma_+,\gamma_-).$$
We consider the $1$-dimensional moduli space $\mathcal{M}^0$ and explain how each component of $\partial \mathcal{M}^0$ contributes to Equation~\eqref{equation: chain homotopy}.

We first establish some notation.  An SFT building $u^\infty$ will be written as:
$$u^\infty=v_{-b}\cup \dots \cup v_{-1}\cup v_0\cup v_1\cup\dots \cup v_a$$
where the levels are arranged from bottom to top as we go from left to right; $v_j$, $j<0$, maps to $(\R\times M,J_-)$; $v_0$ maps to $\widehat{X}^\tau$ for some $\tau$; $v_j$, $j>0$, maps to $(\R\times M, J_+)$; the $v_j$'s are all holomorphic cylinders; and the pregluing of all the $v_j$ yields a cylinder of $\op{ind}=0$.  The only possible level with negative Fredholm index is $v_0$ \cb by Lemma \ref{lemma: no negative levels}.

\cb
\begin{lemma} \label{lemma: regularity for chain homotopy}
Let $\gamma_+\in \mathcal{P}^{L_+}_{\alpha_+}$ and $\gamma_-\in \mathcal{P}^{L_-}_{\alpha_-}$.
\be
\item $\mathcal{M}^{-1}$ is transversely cut out as a family and all the curves of $\mathcal{M}^{-1}$ are simply-covered.
\item If at least one of $\gamma_+$ and $\gamma_-$ is good, then $\mathcal{M}^{0}$ is transversely cut out as a family and all the curves of each component of $\mathcal{M}^{0}$ have the same multiplicity.
\ee
\end{lemma}
\cb

\begin{proof}
(1) There are no simply-covered curves of $\op{ind}<-1$ by the regularity of $J$. Hence, by Lemma~\ref{lemma: multiplicative}, each $u\in \mathcal{M}^{-1}$ must be simply-covered and regular \cb in a family. \cb

(2) Since simply-covered curves are regular in a generic family, it remains to consider the case where $v\in \mathcal{M}^0$ is an unbranched $m$-fold cover of a simple $u$ with $\op{ind}=0$. (Since we are assuming that $v$ is a cylinder, $v$ cannot be a branched cover of a simple $u$ with nonempty branch locus.)  The proof of the regularity of $v$ (in a family) is similar to that of Lemma~\ref{lemma: chain maps transversality}: Using the notation there, the only case to consider in more detail is when $\mu_\tau(\gamma)=\mu_\tau(\gamma')=0$ and $\ker(D^N_u)^*$ is generated by $Y_0$ which comes from the variation of $\overline{J}^\tau$. Equation~\eqref{eqn: equal to zero} implies that $\ker (D^N_v)^*$ is at most $1$-dimensional; this must be generated by the pullback of $Y_0$ under the covering map, implying the regularity of $v$ in a family.  The second assertion of (2) is immediate.
\end{proof}

\begin{rmk}
While we have shown that $\mathcal{M}^0$ is regular, the same cannot be said about the levels of $\bdry\mathcal{M}^0$.  In fact, a careful analysis of $\bdry\mathcal{M}^0$ is the key part of the proof of the chain homotopy.
\end{rmk}

\cb

\s\n
{\bf Simple $\mathcal{M}^0$ case.} Suppose first that no curve of $\mathcal{M}^0$ is a multiple cover.

The boundary $\partial \mathcal{M}^0=\overline{\mathcal{M}^0}-\mathcal{M}^0$ admits a decomposition
$$\bdry_1\mathcal{M}^0\coprod \bdry_2 \mathcal{M}^0\coprod \bdry_3\mathcal{M}^0,$$
each of which will be discussed below.

\s\n
{\em Type 1.} $\partial_1 \mathcal{M}^0$ corresponds to the case when $\tau=0$ or $1$, i.e.,
$$\partial_1 \mathcal{M}^0=\mathcal{M}^{\op{ind}=0,\op{cyl}}_{\overline{J}^0}(\gamma_+,\gamma_-)\coprod\mathcal{M}^{\op{ind}=0,\op{cyl}}_{\overline{J}^1}(\gamma_+,\gamma_-).$$
They contribute to $$ m(\gamma_-)\cdot \left<  \Phi_{(X^1,\alpha^1,\overline{J}^1)}(\gamma_+)-\Phi_{(X^0,\alpha^0,\overline{J}^0)}(\gamma_+), \gamma_-\right>.$$

\s\n
{\em Type 2.}
$\partial_2 \mathcal{M}^0$ corresponds to the case when curves in $\mathcal{M}^0$ converge to two-level holomorphic buildings:   $v_0 \cup v_1$ or $v_{-1} \cup v_0$, where $v_{-1}$, $v_0$ and $v_1$ are all immersions and $v_0$ is a $k$-fold cover of a simple holomorphic cylinder $u_0$ of index $-1$. We denote the sets of these two types of two-level buildings by $\partial_2^+ \mathcal{M}^0$ and $\partial_2^- \mathcal{M}^0$ respectively, so that $\partial_2 \mathcal{M}^0=\partial_2^+\mathcal{M}^0\coprod \partial_2^- \mathcal{M}^0$.

\s\n
{\em Type 3.} $\bdry_3\mathcal{M}^0$ consists of higher codimension strata.

{\em Case I}. $v_{-1} \cup v_0 $ or $v_0\cup v_1$, where $v_{\pm 1}$ is a singular curve in $\mathbb{R}\times M_\pm$. 

{\em Case II}. $v_{-l} \cup \cdots  \cup v_0$ or $v_0 \cup \cdots \cup v_{l},$ where $l>1$.
%

{\em Case III}. $v_{-l} \cup \cdots \cup  v_0 \cup \cdots \cup v_{l'}$, where $l,l'>0$.

\begin{lemma} \label{lemma: boundary three}
$\bdry_3\mathcal{M}^0=\varnothing$.
\end{lemma}

\begin{proof}
Cases I and II cannot occur by Proposition \ref{prop: can avoid boundary and singular points}. 
Case III is eliminated in Section \ref{subsection: other cases}.
\end{proof}

Now we go back to analyze the contributions from the Type 2 boundary. We will focus on $\partial_2^+\mathcal{M}^0$ and show that it corresponds to the term $\langle K_+ \partial_+(\gamma_+),\gamma_-\rangle$. $\partial_2^-\mathcal{M}^0$ can be dealt with similarly and corresponds to the term $\langle \partial_-K_-(\gamma_+),\gamma_-\rangle$.

We first consider the moduli spaces of the form $\mathcal{M}_{\overline{J}^{\tau_l}}^{\op{ind}=-k,\op{cyl}}(\gamma'_+,\gamma_-)$ for $k\geq 1$.  There are only finitely many $\tau_l \in[0,1]$ such that $\mathcal{M}_{\overline{J}^{\tau_l}}^{\op{ind}=-k,\op{cyl}}(\gamma'_+,\gamma_-)$ is non-empty since $\{\overline{J}^\tau\}$ is generic; and, for each such $\tau_l$, if $(v,{\bf r})\in \mathcal{M}_{\overline{J}^{\tau_l}}^{\op{ind}=-k,\op{cyl}}(\gamma'_+,\gamma_-)$, then $v$ is a $k$-fold cover of a simple curve $u$, where $(u,{\bf r'})\in \mathcal{M}_{\overline{J}^{\tau_l}}^{\op{ind}=-1,\op{cyl}}((\gamma'_+)^{1/k},\gamma^{1/k}_-)$, $\gamma'_+$ is the $k$-fold cover of $(\gamma'_+)^{1/k}$, and $\gamma_-$ is the $k$-fold cover of $\gamma_-^{1/k}$.

\begin{rmk} \label{count}
For each $\overline{J}^{\tau_l}$-holomorphic cylinder from $(\gamma'_+)^{1/k}$ to $\gamma_-^{1/k}$ of index $-1$, without markers but after quotienting by automorphisms, there are $m((\gamma'_+)^{1/k})\cdot m(\gamma_-^{1/k})$ elements of $\mathcal{M}_{\overline{J}^{\tau_l}}^{\op{ind}=-1,\op{cyl}}((\gamma'_+)^{1/k},\gamma_-^{1/k})$ and $k\cdot m((\gamma'_+)^{1/k})\cdot m(\gamma_-^{1/k})$ elements of $\mathcal{M}_{\overline{J}^{\tau_l}}^{\op{ind}=-k,\op{cyl}}(\gamma'_+,\gamma_-)$.
\end{rmk}

\begin{lemma} \label{lemma: boundary in chain homotopy}
For each integer $k\geq 1$ and $R\gg 0$, there exists a $(k-1)$-dimensional vector bundle
$$ \mathcal{O}_{+,k} \to  [R,\infty)\times  \coprod_{\gamma'_+\in \mathcal{P}^L_{\alpha_+}}
\left(\coprod_{0\leq \tau \leq 1} \mathcal{M}_{\overline{J}^\tau}^{\op{ind}=-k,\op{cyl}}(\gamma'_+,\gamma_-) \times \mathcal{M}_{J_+}^{\op{ind}=k,\op{cyl}}(\gamma_+,\gamma'_+)/\mathbb{R}  \right),$$
called the {\em obstruction bundle}, and an {\em obstruction section} $\mathfrak{s}_{+,k}$ of $\mathcal{O}_{+,k},$ for which there exists a neighborhood $\mathcal{N}\subset \overline{\mathcal{M}^0}$ of $\bdry_2^+\mathcal{M}^0$ such that
$$\mathcal{N}-\partial_2^+ \mathcal{M}^0=\coprod_{k\geq 1}  (\mathfrak{s}_{+,k})^{-1}(0).$$
\end{lemma}

\begin{rmk}
When $k=1$, we have the usual gluing of $v_0\cup v_1$, where $\op{ind}(v_0)=-1$, $\op{ind}(v_1)=1$, $v_0$ maps to some $\widehat{X}^{\tau_0}$,  $v_1$ maps to $\R\times M$,  and $v_0,v_1$ are simple.
\end{rmk}

The explicit definitions of $\mathcal{O}_{+,k}$ and $\mathfrak{s}_{+,k}$ are given in Section~\ref{subsection: obstruction bundle} ---  there they are written as $\mathcal{O}$ and $\mathfrak{s}$ with the understanding that $k$ and $(v_0,{\bf r}_0)$ are fixed.  The proof of Lemma \ref{lemma: boundary in chain homotopy} follows from Theorem \ref{thm: bijectivity of gluing map}.

Recall that we are viewing $((v_0,{\bf r}_0),(v_1,{\bf r}_1))\sim((v'_0,{\bf r}'_0),(v'_1,{\bf r}'_1))$ if there exist automorphisms $\pi_1$ and $\pi_2$ of the domains $\dot F_1=\R\times S^1$ and $\dot F_2=\R\times S^1$ such that $v_1=v_1'\circ \pi_1$, $v_2=v_2'\circ \pi_2$, and $\pi_1$ and $\pi_2$ take positive (resp.\ negative) punctures to positive (resp.\ negative) punctures \cb and take markers to markers.

\begin{rmk}[\em Source of $\Q$-coefficients.]
It is tempting to instead identify
\begin{equation} \label{eqn: equivalence reln}
((v_0,{\bf r}_0),(v_1,{\bf r}_1))\sim'((v'_0,{\bf r}'_0),(v'_1,{\bf r}'_1))
\end{equation}
if we do not require $\pi_1$ and $\pi_2$ map markers to markers, but only satisfy
\begin{equation} \label{eqn: equivalence reln 2}
(r_{0+}-r_{0-})+(r_{1+}-r_{1-})  = (r'_{0+}-r'_{0-})+ (r'_{1+}-r'_{1-}).
\end{equation}
Here the asymptotic markers are viewed as elements of $S^1\in \R/\Z$.  In the $k=1$ case, the pairs identified by $\sim'$ represent the same boundary point of $\mathcal{M}^0$.  However, for $k>1$, unless $((v_0,{\bf r}_0),(v_1,{\bf r}_1))=((v'_0,{\bf r}'_0),(v'_1,{\bf r}'_1))$, the upper level of the $2$-level buildings identified by $\sim'$ do not continue to the same component of $(\overline{ev}^k_-)^{-1}(\nu)$; cf.\ the proof of Theorem~\ref{thm: chain homotopy}.  This is the source of ``branching'', which in turn forces us to use $\Q$-coefficients.
\end{rmk}
\cb

\s
Define $n^k_{+,\tau_l}(\gamma_+,\gamma'_+;\gamma_-)$ as the signed count
$$\#\left( \mathfrak{s}_{+,k}^{-1}(0) \cap \left(\{T\}\times \mathcal{M}_{\overline{J}^{\tau_l}}^{\op{ind}=-k,\op{cyl}}(\gamma'_+,\gamma_-)\times  \mathcal{M}_{J_+}^{\op{ind}=k,\op{cyl}}(\gamma_+,\gamma'_+)/\mathbb{R} \right)\right) {1\over m(\gamma'_+)}$$
for generic $T\in [R,\infty)$.  Then, in view of Lemma~\ref{lemma: boundary three},
\begin{align} \label{equation: Phi1-Phi0(gamma+),gamma-}
\left\langle \left(\Phi_{(X^1,\alpha^1,\overline{J}^1)}-\Phi_{(X^0,\alpha^0,\overline{J}^0)}\right)(\gamma_+),\gamma_- \right>
= & \sum_{k=1}^{\infty}\sum_{\gamma'_+\in \mathcal{P}^L_{\alpha_+}} \sum_l n^k_{+,\tau_l}(\gamma_+,\gamma'_+;\gamma_-)\frac{1}{m(\gamma_-)} \\
\nonumber  & + ~~~ \textrm{ terms coming from } \partial_2^- \mathcal{M}^0.
\end{align}

Recall the evaluation map
\[\overline{ev}^k_-=\overline{ev}^k_-(\gamma_+,\gamma'_+;J_+):\overline{\mathcal{M}_{J_+}^{\op{ind}=k,\op{cyl}}(\gamma_+,\gamma'_+)/\R} \to  S^{k-1}.\]
from Section \ref{section: evaluation map}.   Also recall that $\overline{ev}^k_-(v_1,r_{1+},r_{1-})$ depends on the marker $r_{1-}$ at the negative end.  Finally, the asymptotic eigenfunctions $f_1,\dots,f_k$ and the cokernel element $Y$ depend on the parametrization given by $(v_0,r_{0+},r_{0-})$.

\begin{lemma} \label{lemma: zeros of s= zeros of ev}
For any positive integer $k$,
\begin{enumerate}
\item if $(\gamma'_+)^{1/k}$ is negative hyperbolic, then $n^k_{+,\tau_l}(\gamma_+,\gamma'_+;\gamma_-)$ is given by 
\begin{equation} \label{equation: count}
 \#\mathcal{M}^{\op{ind}=-k,\op{cyl}}_{\overline{J}^{\tau_l}}({\gamma'}_+,{\gamma}_-)\cdot
 \# \left( \left(\widetilde{ev}^k_-(\gamma_+,\gamma'_+;J_+)\right)^{-1} (\{(0,\dots,0,\pm 1)\})\right){1\over {m(\gamma'_+)}};
\end{equation}
\item if $(\gamma'_+)^{1/k}$ is positive hyperbolic, then we replace $(0,\dots,0,\pm 1)$ by $(\pm 1, 0,\dots,0)$.
\end{enumerate}
\end{lemma}

\begin{proof}
The proof of Lemma~\ref{lemma: zeros of s= zeros of ev} follows from \cb Lemma \ref{lemma: s_0}, \cb Proposition \ref{prop: zeros do not escape}, and Section~\ref{subsection: other cases} A.
%
\end{proof}

Now we are in a position to finish the proof of Theorem \ref{thm: chain homotopy} in the simple $\mathcal{M}^0$ case.

\begin{proof}[ Proof of Theorem~\ref{thm: chain homotopy} in the simple $\mathcal{M}^0$ case.]
Let us abbreviate
$$
\mathcal{M}=\mathcal{M}_{J_+}^{\op{ind}=k,\op{cyl}}(\gamma_+,\gamma'_+)$$
and denote the restriction of $\overline{ev}^k_-$ to the boundary $\bdry (\mathcal{M}/\R)$ by $\bdry \overline{ev}^k_-$. We assume $(\gamma'_+)^{1/k}$ is negative hyperbolic; the case of $(\gamma'_+)^{1/k}$ positive hyperbolic is similar. By Theorems \ref{thm: niceness of ev} and \ref{thm: niceness of compactification}, there exists a generic embedded  \cb oriented path $\nu:[0,1]\to S^{k-1}$ from $(0,\dots,0,-1)$ to $(0,\dots,0,+1)$ such that $\nu \pitchfork \overline{ev}^k_-$ and $\nu\pitchfork \bdry\overline{ev}^k_-$. In particular, $\bdry \overline{ev}^k_-\cap \nu$ consists of finitely many points.
Here the path $\nu$ is oriented with the standard orientation from $[0,1]$, and $\bdry \nu = \{(0,\dots,0,\pm1)\}$ is oriented as the boundary. We postpone further details about orientation to Section~\ref{subsection: orientations chain homotopy}.
 \cb

Let
$$\mathcal{M}_+^{k-1}(\zeta_+,\gamma'_+)
=\bigg \{  w \in \mathcal{M}^{\op{ind}=k-1,\op{cyl}}_{J_+}(\zeta_+,\gamma'_+) ~~\Big | ~~ \widetilde{ev}^k_-(\zeta_+,\gamma'_+;J_+)(w)\in \bdry \overline{ev}^k_- \cap \nu \bigg \}.$$
For $k=1$ we define 
$$K_+^{k-1}(\zeta_+,\gamma_-)=\sum_l \#\mathcal{M}^{\op{ind}=-1,\op{cyl}}_{\overline{J}^{\tau_l}}(\zeta_+,\gamma_-)$$
and for $k \geq 2$ we define
$$K_+^{k-1}(\zeta_+,\gamma_-)=\sum_l \sum_{\gamma'_+\in \mathcal{P}^L_{\alpha_+}}
\#\mathcal{M}^{\op{ind}=-k,\op{cyl}}_{\overline{J}^{\tau_l}}({\gamma'}_+,{\gamma}_-)
\# (\mathcal{M}_{+}^{k-1}(\zeta_+,\gamma'_+)/\R) {1\over m({\gamma'}_+)}.$$
We then define
\begin{equation}
K_+(\zeta_+)=\sum_{k=1}^\infty \sum_{\gamma_-\in \mathcal{P}^L_{\alpha_-}} K_+^{k-1}(\zeta_+,\gamma_-)\frac{1}{m(\gamma_-)}\cdot \gamma_-.
\end{equation}

We claim that
\begin{equation} \label{rats}
(\bdry \overline{ev}^k_-)^{-1}(\nu)=\left.\left(\coprod_{\zeta_+} \mathcal{M}_+^{k-1}(\zeta_+,\gamma'_+)/\R \times
(\mathcal{M}_{J_+}^{\op{ind}=1,\op{cyl}}(\gamma_+,\zeta_+)/\mathbb{R})  \right)\right/ \sim',
\end{equation}
where $\sim'$ is defined as in Equations~\eqref{eqn: equivalence reln} and \eqref{eqn: equivalence reln 2} with all the subscripts increased by $1$.
Indeed, ``$\supseteq$" follows from the fact that the evaluation map $\overline{ev}_-^k$ only depends on the behavior of holomorphic curves near the negative end; and ``$\subseteq$" follows from the fact that generically there is no other boundary component by Theorem \ref{thm: niceness of compactification}. Observe that $ \{(0,\dots,0,\pm 1)\} \cap \bdry \overline{ev}^k_-=\emptyset$.

By examining the boundary of the $1$-dimensional manifold $(\overline{ev}^k_-)^{-1}(\nu)$,
\begin{align} \label{equation: ev-=}
& \# \left( (\overline{ev}^k_-)^{-1}(\{(0,\dots,0,\pm 1)\})\right)  \\
&\hspace{1cm}=\sum_{\zeta_+\in \mathcal{P}^L_{\alpha_+}} \#(\mathcal{M}_+^{k-1}(\zeta_+,\gamma'_+)/\R) \#(\mathcal{M}_{J_+}^{\op{ind}=1,\op{cyl}}(\gamma_+,\zeta_+)/\mathbb{R}){1\over m(\zeta_+)}, \nonumber
\end{align}
where the right-hand side comes from Equation~\eqref{rats}. 
Then Lemma \ref{lemma: zeros of s= zeros of ev} and Equation~\eqref{equation: ev-=} imply that the first term on the right-hand side of Equation~\eqref{equation: Phi1-Phi0(gamma+),gamma-} is equal to $\left< K_+ \partial_+ (\gamma_+),\gamma_- \right>$.  This finishes the proof of Theorem \ref{thm: chain homotopy}, modulo the discussion of orientations from Section~\ref{subsection: orientations chain homotopy}.
\end{proof}

\s\n
{\bf Non-simple $\mathcal{M}^0$ case.} We explain the necessary modifications when some curve $u\in\mathcal{M}^0$ is multiply-covered.

\cb By Lemma~\ref{lemma: chain maps transversality}, $u\in \mathcal{M}^0$ is part of a regular $1$-dimensional family and all the curves in the connected component of $\mathcal{M}^0$ containing $u$ are $m$-fold covers of a regular $1$-dimensional family $\coprod_{0\leq \tau \leq 1} \mathcal{M}^{\op{ind}=0,\op{cyl}}_{\overline{J}^\tau}(\gamma,\gamma')$. \cb

Let $\mathcal{M}=\mathcal{M}_{J_+}^{\op{ind}=k,\op{cyl}}(\gamma_+,\gamma'_+)$ as before and let $\mathcal{M}'\subset \mathcal{M}$ be the subset of curves which are $b$-fold covers, where $b|k$. We then use the inclusion given by Equation~\eqref{eqn: inclusion} to argue that for generic $J$ the evaluation map $(ev_-')^k: \mathcal{M}'\to \R^k$ does not pass through zero and hence descends to $(\widetilde{ev}'_-)^k: \mathcal{M}'/\R\to S^{k-1}$.

The rest of the argument is the same.  Suppose the family $\mathcal{M}^0$ has boundary which consists of a two-level building $v_0\cup v_1$, where $\op{ind}(v_0)=-k$, $\op{ind}(v_1)=-k$, $v_0$ is a $k$-fold cover, and $v_1$ is a $b$-fold cover with $b|k$.  As before, the only pairs $v_0\cup v_1$ that glue satisfy $(\widetilde{ev}'_-)^k(v_1)=(0,\dots,0,\pm 1)$ and we continue the family along some generic $\nu\subset S^{k-1}$ from $(0,\dots,0,1)$ to $(0,\dots,0,-1)$, where the evaluation map we are taking is $\widetilde{ev}_-^k:\mathcal{M}/\R \to S^{k-1}$.

\subsection{$HC(\mathcal{D}) $ is independent of $\mathcal{D}$}

To show that $HC(\mathcal{D})$ is independent of $\mathcal{D}$, we study the composition of chain maps induced by a composition of symplectic cobordisms.

For $j=1,2$, let $(X^j,\alpha^j)$ be an exact symplectic cobordism as in Section \ref{subsection: chain maps}, with $M_-^1=M_+^2$ and $\alpha_-^1= \alpha_+^2$. We define $X^{12}$ to be $X^1\cup X^2$ identified along $M_-^1=M_+^2$ and define the $1$-form $\alpha^{12}$ on $X^{12}$ by $\alpha^{12}|_{X^1}=\alpha^1$ and $\alpha^{12}|_{X^2}=\alpha^2$.

Let $(\widehat{X}^j,\widehat{\alpha}^j)$ be the completion of $(X^j,\alpha^j)$ and let $\overline{J}^j$ be an almost complex structure which tames $(\widehat X^j,\widehat\alpha^j)$ and restricts to $\alpha^j_\pm$-tame almost complex structures $J^j_\pm$ at the positive and negative ends. Similarly, let $(\widehat{X}^{12}, \widehat{\alpha}^{12})$ be the completion of $(X^{12},\alpha^{12})$ and let $\overline{J}^{12}$ be a generic almost complex structure on $\widehat{X}^{12}$ which tames $(\widehat{X}^{12}, \widehat{\alpha}^{12})$ and coincides with $\overline{J}^1_+$ and $\overline{J}^2_-$ at the positive and negative ends.

\begin{thm} \label{thm: composition of chain maps}
Suppose that $(\widehat{X}^j, \widehat{\alpha}^{j}, \overline{J}^{j})$ and $(\widehat{X}^{12}, \widehat{\alpha}^{12}, \overline{J}^{12})$ satisfy the assumptions of Theorem \ref{thm: chain maps}. Then $\Phi_{(\widehat{X}^{12}, \widehat{\alpha}^{12}, \overline{J}^{12}) }$ and $\Phi_{(\widehat{X}^2, \widehat{\alpha}^{2}, \overline{J}^{2}) } \circ \Phi_{(\widehat{X}^1, \widehat{\alpha}^{1}, \overline{J}^{1}) }$ induce the same map on homology.
\end{thm}

\begin{proof}
The proof of this is essentially the same as the proof of Theorem \ref{thm: chain homotopy}, as one can construct a $1$-parameter family of completed symplectic cobordisms between $\widehat{X}^{12}$ and $\widehat{X}^1 \cup \widehat{X}^2$.
\end{proof}

Theorem \ref{thm: chain homotopy} and Theorem \ref{thm: composition of chain maps} imply:

\begin{cor}
$HC(\mathcal{D})$ is independent of the auxiliary data $\mathcal{D}$.
\end{cor}

\section{Gluing}\label{section: gluing}

In this section we construct the gluing map that is used in Section \ref{section: chain homotopy}. In Section \ref{section: chain homotopy}, we see that a sequence of curves in $\coprod_{0\leq \tau \leq 1}\mathcal{M}^{\op{ind}=0,\op{cyl}}_{\overline{J}^{\tau}}(\gamma'_+,\gamma_-)$ can degenerate into a holomorphic building
$$([v_-],[v_+])\in \left(\mathcal{M}^{\op{ind}=-k,\op{cyl}}_{\overline{J}^{\tau_0}}(\gamma'_+,\gamma_-)\times (\mathcal{M}_{J_+}^{\op{ind}=k,\op{cyl}}(\gamma_+,\gamma'_+)/\mathbb{R})\right) /\sim$$
for some $\tau_0$.
For convenience we write $\mathcal{M}=\mathcal{M}_{J_+}^{\op{ind}=k,\op{cyl}}(\gamma_+,\gamma'_+)$.  For most of this section we are assuming Conditions \hyperref[Condition1]{(C1)}--(C4) of the prototypical gluing problem, where $u_0,v_0,v_1,\gamma,\gamma'',\gamma'$ are now called $u_-,v_-,v_+,\gamma_+,\gamma'_+,\gamma_-$.  At the end we will treat similar, but slightly different, cases in Section~\ref{subsection: other cases}.

Certainly not every $[v_+]\in \mathcal{M}/\mathbb{R}$ can be glued with $[v_-]$ to give a holomorphic map. In this section we closely follow \cite{HT2} (with the appropriate modifications) and define an obstruction bundle
$$\mathcal{O} \to [R,\infty)\times \mathcal{M}/\mathbb{R}, \quad R\gg 0,$$
and a section $\mathfrak{s}$ of $\mathcal{O}$ such that $\mathfrak{s}^{-1}(0)\cap (\{T\}\times \mathcal{M}/\R)$ are exactly the curves $[v_+]$ that glue with $[v_-]$ for $T\gg 0$.  The expressions will be substantially simpler since $(\alpha_+,J_+)$ is $L$-supersimple.  The proofs of the results that carry over with minimal changes will be omitted.

{\em Assume for the moment that all the curves in $\mathcal{M}$ are immersed.} In Section~\ref{subsection: the general case} we will explain how to modify the argument in the general case.

From now on we implicitly choose a smooth section of $\mathcal{M}\to\mathcal{M}/\mathbb{R}$ and representatives of $([v_-],[v_+])$, and write $(v_-,v_+)$ instead of $([v_-],[v_+])$.

Also, in this section the constant $c>0$ may change from line to line when we are making estimates.

\subsection{Pregluing}\label{subsection:pregluing}

Fix a constant $T_0\gg 0$.  Also let $T\gg 2T_0$, which is allowed to vary.
Let $v_{+,T}(s,t):=v_+(s-2T,t)$.

Recall the coordinates $(s,t,x,y)$ on the neighborhood $\R\times\R/\Z \times D^2_{\delta_0/3}$ of $\R\times \gamma_+'$ on which $J$ satisfies (J1) and (J2). For sufficiently large $T_0$, $v_-(s,t)$ can be written in terms of these coordinates as $(s,t,\eta_-(s,t))$ on $s\geq T_0$ and $v_{+,T}(s,t)$ can be written as $(s,t,\eta_{+,T}(s,t))$ on $s\leq T$.

Fix constants $0<h<1$ and $r\gg h^{-1}$.  We take $T_0> 5r$.
Choose a cutoff function $\beta: \mathbb{R} \to [0,1]$ such that $\beta(s)=0$ for $s\leq 0$ and $\beta(s)=1$ for $s\geq 1$.  Let $\beta_{-,T}(s)=\beta({T-s\over hr})$ and $\beta_{+,T_0}(s)=\beta({s-T_0\over hr})$.

\begin{figure}[ht]
\begin{center}
\psfragscanon
\psfrag{s}{\tiny $s$}
\psfrag{A}{\tiny $T_0$}
\psfrag{B}{\tiny $T_0+hr$}
\psfrag{C}{\tiny $T-hr$}
\psfrag{D}{\tiny $T$}
\psfrag{E}{\tiny $\beta_{+,T_0}$}
\psfrag{F}{\tiny $\beta_{-,T}$}
\includegraphics[width=7cm]{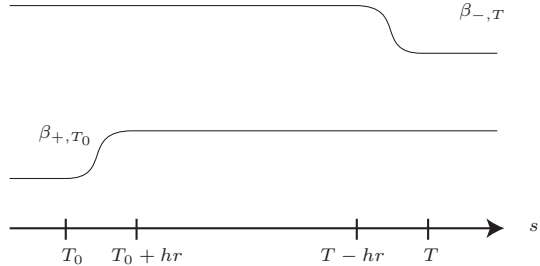}
\end{center}
\caption{Cutoff functions.} \label{fig: cutoff functions}
\end{figure}

We define the {\em pregluing of $v_+$ and $v_-$} by
$$v_*(s,t)=\left\{\begin{array}{ll} v_{+,T}(s,t) & \op{for } s\geq T, \\
(s,t, \beta_{+,T_0}(s)\eta_{+,T}(s,t)+\beta_{-,T}(s)\eta_-(s,t) )&   \op{for } T_0\leq s \leq T, \\
v_{-}(s,t) & \op{for } s\leq T_0. \end{array}\right.$$
We often suppress the $T$ or $T_0$ in $\beta_{+,T_0}$, $\beta_{-,T}$, and $\eta_{+,T}$.

\subsection{Gluing}\label{subsection: gluing}

Let $\psi_{+}$ be a section of the normal bundle of $v_{+,T}$ and $\psi_{-}$ be a section of the normal bundle of $v_{-}$. We deform $v_*$ to
\begin{equation}\label{equation: deform v*}
v=\exp_{v_*} (\beta_+\psi_++\beta_-\psi_-),
\end{equation}
where $\exp_{v_*}$ is an exponential map which identifies the normal bundle to $v_*$ with a tubular neighborhood of $v_*$. The exponential maps $\exp_{v_{+,T}}$ and $\exp_{v_-}$ can be chosen such that
$$\exp_{v_{+,T}}\psi_+= (s,t,\eta_+ +\psi_+) \quad \mbox{and} \quad \exp_{v_-}\psi_-=(s,t,\eta_- +\psi_-)$$
for $s\leq T$ and $s\geq T_0$, respectively. Similarly, we assume that
$$\exp_{v_*} (\beta_+\psi_++\beta_-\psi_-)=(s,t,\eta_*+\beta_+\psi_++\beta_-\psi_-)$$
for $T_0\leq s \leq T$.
Here $\eta_*=\beta_+\eta_+ +\beta_-\eta_-$.

Let $\tau\in[0,1]$ be close to $\tau_0$. The equation $\overline{\partial}_{\overline{J}^\tau}v=0$ is equivalent to:
\begin{equation} \label{eqn: d-bar rewritten}
\beta_- \left(D_- \psi_-  +(\tau-\tau_0)Y' +\frac{\partial \beta_+}{\partial s}\eta_+ +\mathcal{R}_-\right)+\beta_+\left(D_+\psi_++\frac{\partial \beta_-}{\partial s}\eta_- + \mathcal{R}_+\right)=0,
\end{equation}
where $D_-$ is the linearization of $\overline{\partial}_{\overline{J}^{\tau_0}}$ for the normal bundle of $v_-$,  $D_+$ is the linearization of $\overline{\partial}_{J_+}$ for the normal bundle of $v_{+,T}$, $Y'$ comes from the variation of the almost complex structure from $\overline{J}^\tau$ to $\overline{J}^{\tau_0}$, and by abuse of notation we write $\frac{\partial \beta_+}{\partial s}\eta_+$ instead of $\frac{\partial \beta_+}{\partial s}\eta_+\otimes (ds-idt)$.

We will now describe the terms $\mathcal{R}_-$ and $\mathcal{R}_+$. Recall from \cite[Definition~5.1]{HT2} that $F(\psi)$ is {\em type $1$ quadratic} if it can be written as
$$F(\psi)=P(\psi) + Q(\psi)\cdot \nabla\psi,$$
where $|P(\psi)|\leq c|\psi|^2$ and $|Q(\psi)|\leq c|\psi|$ for some constant $c>0$.
On $s\leq T_0$,
$$\mathcal{R}_-=F_-(\psi_-,\tau-\tau_0),$$
where $F_-(\psi_-,\tau-\tau_0)$ is type $1$ quadratic with respect to $\psi=(\psi_-,\tau-\tau_0)$.  On $s\geq T$,
$$\mathcal{R}_+= F_+(\psi_+),$$
where $F_+(\psi_+)$ is type $1$ quadratic. [In local coordinates,
\begin{align*}
{\bdry(u+\xi)\over \bdry s} + J_\tau(u+\xi) {\bdry (u+\xi)\over \bdry t} & =  \left({\bdry u\over\bdry s} + J_{\tau_0}(u){\bdry u\over \bdry t}\right) +
   \left({\bdry \xi\over \bdry s} + J_{\tau_0}(u){\bdry \xi\over \bdry t}\right)\\
  & \hspace{-1cm}+  (J_\tau(u+\xi)-J_{\tau_0}(u)){\bdry u\over \bdry t} +(J_\tau(u+\xi)-J_{\tau_0}(u)){\bdry \xi\over \bdry t}.
\end{align*}
The first term on the right is zero; the second is part of the linearization $D_u \xi$; the third contributes $(\tau-\tau_0)Y'$ and a term $\nabla J_{\tau_0}(\xi){\bdry u\over \bdry t}$ towards $D_u\xi$, and the remainder is quadratic and higher in $(\xi, \tau-\tau_0)$; the fourth is bounded by $|(\xi, \tau-\tau_0)|\cdot |\nabla \xi|$.]

\begin{claim} \label{claim: R plus and R minus on neck}
If $(\alpha_\pm,J_\pm)$ is $L$-supersimple and $T_0\gg 0$, then $Y'=0$ on $s\geq T_0$ and we can take $\mathcal{R}_\pm$ to be
\begin{equation} \label{eqn: expression fo R pm between T_0 and T}
\mathcal{R}_-=\frac{\partial \beta_+}{\partial s}\psi_+, \quad \mathcal{R}_+=\frac{\partial \beta_-}{\partial s}\psi_-
\end{equation}
for $s\geq T_0$ and $s\leq T$, respectively.
\end{claim}

\begin{proof}
Since $\overline{J}^\tau=J_+$ on $s\geq T_0$ and $T_0\gg 0$, we have $Y'=0$.

On $T_0\leq s\leq T$, $\overline{\bdry}_{\overline{J}^\tau}v =0$ can be written as:
\begin{align*}
& \beta_- D_-\psi_- +\beta_+ D_+\psi_+ + {\bdry \beta_-\over \bdry s}(\eta_-+\psi_-) +{\bdry\beta_+\over \bdry s}(\eta_++\psi_+) \\
& = \beta_-\left( D_-\psi_- + {\bdry\beta_+\over \bdry s} (\eta_++\psi_+)\right) +\beta_+\left(D_+\psi_+ +{\bdry \beta_-\over \bdry s}(\eta_-+\psi_-)\right) =0,
\end{align*}
where $D_-=D_+={\bdry\over \bdry s}-A$. (Note that $D_-\eta_-=0$ and $D_+\eta_+=0$.)  This decomposition is consistent with Equation~\eqref{eqn: d-bar rewritten} and hence we can define $\mathcal{R}_\pm$ by Equation~\eqref{eqn: expression fo R pm between T_0 and T} for $T_0\leq s\leq T$.

We also define $\mathcal{R}_-=0$ for $s\geq T$ and $\mathcal{R}_+=0$ for $s\leq T_0$.
\end{proof}

We then rewrite Equation~\eqref{eqn: d-bar rewritten} as
$$\beta_-\Theta_-(\psi_-,\psi_+,\tau)+\beta_+\Theta_+(\psi_-,\psi_+)=0,$$
by setting
\begin{equation}
\label{Theta -}
\Theta_-(\psi_-,\psi_+,\tau)= D_- \psi_-  +(\tau-\tau_0)Y' +\frac{\partial \beta_+}{\partial s}\eta_+ +\mathcal{R}_-,
\end{equation}
\begin{equation}
\label{Theta +}
\Theta_+(\psi_-,\psi_+) = D_+\psi_++\frac{\partial \beta_-}{\partial s}\eta_- + \mathcal{R}_+.
\end{equation}

We want to solve the equations $\Theta_-(\psi_-,\psi_+,\tau)=0$ and $\Theta_+(\psi_-,\psi_+)=0$, subject to $\psi_\pm$ being in  $(\op{ker}(D_\pm))^\perp$ the $L^2$-orthogonal complement of $\op{ker}(D_\pm)$. For any sufficiently small $\psi_-$, one can solve for $\psi_+=\psi_+(\psi_-)\in (\ker D_+)^{\perp}$ in $\Theta_+(\psi_-,\psi_+)=0$, since $D_+$ is surjective; see Lemma~\ref{azalea}.

Next we solve for $(\psi_-,\tau)$ in $\Theta_-(\psi_-,\psi_+(\psi_-),\tau)=0$. This is equivalent to solving the following triple of equations
\begin{equation}\label{equation: D+1-Pi}
D_- \psi_-+(1-\Pi)\left((\tau-\tau_0)Y'+\frac{\partial \beta_+}{\partial s}\eta_+ +\mathcal{R}_-\right)=0,
\end{equation}
\begin{equation}\label{equation: one more}
\Pi_Y\left((\tau-\tau_0)Y'+\frac{\partial \beta_+}{\partial s}\eta_+ +\mathcal{R}_-\right)=0,
\end{equation}
\begin{equation}\label{equation: Pi=0}
\Pi'\left(\frac{\partial \beta_+}{\partial s}\eta_+ +\mathcal{R}_-\right)=0,
\end{equation}
where $\Pi$ is the orthogonal projection to $\ker D_-^*$, $\Pi_Y$ is the orthogonal projection to $\R\langle Y\rangle$, and $\Pi'$ is the orthogonal projection to $\ker D_-^*\cap Y^\perp$.  Here $Y=\Pi Y'$ and $Y^\perp$ is the orthogonal complement of $Y$.  $Y$ is nonzero because $\{ \overline{J}^\tau\}_{0\leq \tau\leq 1}$ is generic.

Given $(T,v_+)$, one can always solve for $\psi_+$, $\psi_-$ and $\tau$ in Equations~\eqref{Theta +}, \eqref{equation: D+1-Pi} and \eqref{equation: one more}, subject to $\psi_+\in (\ker D_+)^{\perp}$; see Lemma~\ref{estimate for psi and tau}. We then set
$$v(T,v_+):=\exp_{v_*} (\beta_+\psi_++\beta_-\psi_-)$$
where the right-hand side of the equation (implicitly) depends on $T$ and $v_+$.

\subsection{Banach spaces} \label{subsection: banach spaces}

The function spaces that we use are {\em Morrey spaces}, following \cite[Section 5.5]{HT2}.\footnote{This is rather nonstandard and we chose to adopt it to avoid redoing the work in \cite{HT2} for $W^{k,p}$-spaces.} Let $u:\dot F\to \R\times M$ be an immersed finite energy holomorphic curve and $N\to \dot F$ be a normal bundle. On $\dot F$ we choose a Riemannian metric so that the ends are cylindrical and on $N$ we use the metric induced from an $\R$-invariant Riemannian metric on $\R\times M$.

The {\em Morrey space} $\mathcal{H}_0(\dot F, \Lambda^{0,1} T^*\dot F\otimes N)$ is the Banach space which is the completion of the compactly supported sections of $\Lambda^{0,1} T^*\dot F\otimes N$ with respect to the norm
$$\| \xi\|= \left( \int_{\dot F} |\xi|^2  \right)^{1/2} + \left( \sup_{x\in \dot F} \sup_{\rho\in(0,1]} \rho^{-1/2} \int_{B_\rho(x)} |\xi|^2  \right)^{1/2},$$
where $B_\rho(x)\subset \dot F$ is the ball of radius $\rho$ about $x$.  Similarly, $\mathcal{H}_1(\dot F, N)$ is the completion of the compactly supported sections of $N$ with respect to $$\| \xi \|_*= \|\nabla \xi\| + \|\xi\|.$$

The analog of the usual Sobolev embedding theorem is the following:

\begin{lemma} \label{lemma: sobolev embedding}
There is a bounded linear map
$$\mathcal{H}_1(\dot F, N)\to C^{0,1/4}(\dot F, N), \quad \xi\mapsto \xi,$$
where $C^{0,1/4}$ denotes the space of H\"older continuous functions with exponent ${1\over 4}$.
\end{lemma}

\subsection{Some estimates} \label{subsection: estimates}

Let $\lambda= \min \{\lambda_1, |\lambda_{-1}|\}$.  Let $\mathcal{H}_{+}$ be the $L^2$-orthogonal complement of $\ker D_+$ in $\mathcal{H}_1(\dot F,N)$ corresponding to $v_+$, $\mathcal{H}_{-}$ be $\mathcal{H}_1(\dot F,N)$ corresponding to $v_-$, and $\mathcal{B}_{\pm}$ be the closed ball of radius $\varepsilon$ in $\mathcal{H}_{\pm}$ centered at $0$.

The following lemma closely follows \cite[Proposition 5.6]{HT2}, but the estimates are slightly different.

\begin{lemma} \label{azalea}
There exist $r\gg 0$ and $c,\varepsilon>0$ such that for $T\gg 0$ the following holds:
\begin{enumerate}
\item There is a map $P:\mathcal{B}_-\to \mathcal{H}_+$ such that  Equation~\eqref{Theta +} holds with $\psi_+=P(\psi_-).$
\item $\| P(\psi_-)\|_*\leq cr^{-1}( e^{-\lambda T}+\|\psi_-\|_*)$.
\end{enumerate}
\end{lemma}

\begin{proof}
(1) We are trying to solve for
$$D_+\psi_+ + {\bdry \beta_-\over \bdry s}(\eta_-+\psi_-) +\mathcal{R}'_+(\psi_+)=0,$$
where $\mathcal{R}'_+(\psi_+)$ is type 1 quadratic.  Writing
\begin{equation} \label{def of I}
\mathcal{I}(\psi_+)= -D_+^{-1}\left({\bdry \beta_-\over \bdry s}(\eta_-+\psi_-) +\mathcal{R}'_+(\psi_+)\right),
\end{equation}
where $D_+^{-1}$ is the bounded inverse of $D_+|_{\mathcal{H}_+}$, we define $P(\psi_-)$ as the unique fixed point $\mathcal{I}(\psi_+)=\psi_+$.

The unique fixed point is guaranteed by the contraction mapping theorem, which relies on two estimates.  Our first estimate is
\begin{align} \label{estimate for I}
\|\mathcal{I}(\psi_+)\|_* & \leq cr^{-1}(e^{-\lambda T}+ \|\psi_-\|_*) +c\|\psi_+\|^2_*.
\end{align}
Indeed, since $|\tfrac{\bdry\beta_-}{\bdry s}|< cr^{-1}$ for some $c>0$,
$$\| \tfrac{\bdry \beta_-}{\bdry s}(\eta_-+\psi_-)  \|\leq cr^{-1}(e^{-\lambda T}+ \|\psi_-\|_*),$$
and, since $\mathcal{R}'_+$ is type 1 quadratic,
$$\| \mathcal{R}'_+(\psi_+)\| \leq c\|\psi_+\|^2_*.$$
The estimate follows from the boundedness of $D_+^{-1}$; observe that the domain of $D_+^{-1}$ uses the norm $\|\cdot \|$ and the range of $D_+^{-1}$ uses $\|\cdot \|_*$.  If $r,T\gg 0$ and $\varepsilon>0$ is small, then the right-hand side of Equation \ref{estimate for I} is $<\varepsilon$ whenever $\|\psi_-\|_*, \|\psi_+\|_*<\varepsilon$. Hence $\mathcal{I}$ maps a radius $\varepsilon$ ball into itself.


Our second estimate is
\begin{align} \label{estimate for I two}
\|\mathcal{I}(\psi_+)-\mathcal{I}(\psi_+')\|_* & = \| D_+^{-1}\left( \mathcal{R}'_+(\psi_+)-\mathcal{R}'_+(\psi'_+)  \right)\|_*\\
\nonumber &\leq c (\|\psi_+\|_* +\|\psi_+'\|_*)\cdot \| \psi_+-\psi'_+\|_*.
\end{align}
This follows from observing that $D_+^{-1}$ is bounded and $\mathcal{R}'$ is type 1 quadratic. When $\|\psi_+\|_*, \|\psi_+'\|_*$ are sufficiently small, $\mathcal{I}$ gives a contraction mapping.  This proves (1).

(2) $\psi_+=P(\psi_-)$ satisfies $\mathcal{I}(\psi_+)=\psi_+$. Hence Equation~\eqref{estimate for I} gives:
$$\|\psi_+\|_* \leq cr^{-1}(e^{-\lambda T}+ \|\psi_-\|_*) +c\|\psi_+\|^2_*.$$
This implies (2) since $c\|\psi_+\|^2_*\ll \|\psi_+\|_*$ for $\varepsilon>0$ small.
\end{proof}

The following lemma closely follows \cite[Proposition 5.7]{HT2}.

\begin{lemma}\label{suzuran}
There exist $r\gg 0$, $c,\varepsilon>0$, and $0<\varepsilon_0<1$ such that for $T\gg 0$ the following holds:
\begin{enumerate}
\item There is a map $P':\mathcal{B}_+\times [-\varepsilon_0,\varepsilon_0]\to \mathcal{H}_-$ such that Equation~\eqref{equation: D+1-Pi} holds with $\psi_-=P'(\psi_+,\tau-\tau_0)$.
\item $\| P'(\psi_+,\tau-\tau_0)\|_*\leq cr^{-1}( e^{-\lambda T}+\|\psi_+\|_*)+c|\tau-\tau_0|$.
\end{enumerate}
Here $\tau-\tau_0$ is the coordinate for $[-\varepsilon_0,\varepsilon_0]$.
\end{lemma}

\begin{proof}
(1)  We are trying to solve for
$$ D_-\psi_- +(1-\Pi)\left((\tau-\tau_0)Y' + {\bdry\beta_+\over \bdry s} (\eta_++\psi_+) +\mathcal{R}'_-(\psi_-,\tau-\tau_0)\right)=0,$$
where $\mathcal{R}'_-$ is type 1 quadratic. Let us write
\begin{align} \label{revisit}
\mathcal{I}_{\psi_+,\tau-\tau_0}(\psi_-)= -D_-^{-1}(1-\Pi)\left((\tau-\tau_0)Y' + {\bdry\beta_+\over \bdry s} (\eta_++\psi_+) +\mathcal{R}'_-(\psi_-,\tau-\tau_0)\right).
\end{align}

We first estimate
\begin{align} \label{estimate for I three}
\| \mathcal{I}_{\psi_+,\tau-\tau_0}(\psi_-) \|_* & \leq c|\tau-\tau_0| + cr^{-1}(e^{-\lambda T}+\|\psi_+\|_*)+ c(\|\psi_-\|_*^2 + |\tau-\tau_0|^2 ).
\end{align}
Provided $r,T\gg 0$, $\varepsilon>0$ is small, and $|\tau-\tau_0|$ is small, the right-hand side is $<\varepsilon$ whenever $\|\psi_-\|_*, \|\psi_+\|_*<\varepsilon$. Hence $\mathcal{I}_{\psi_+,\tau-\tau_0}$ maps a radius $\varepsilon$ ball into itself.

Next we estimate
\begin{align} \label{estimate for I four}
\|\mathcal{I}_{\psi_+,\tau-\tau_0}(\psi_-)-\mathcal{I}_{\psi_+,\tau-\tau_0}(\psi_-')\|_* & = c  \|\mathcal{R}'_-(\psi_-,\tau-\tau_0)-\mathcal{R}'_-(\psi'_-,\tau-\tau_0) \|\\
\nonumber &\leq  C\| \psi_--\psi'_-\|_*, \quad 0< C< 1,
\end{align}
provided $|\tau-\tau_0|$, $\|\psi_-\|_*$, and $\|\psi_-'\|_*$ are small.  The above estimates provide a contraction mapping.

(2) Follows from Equation~\eqref{estimate for I three} and $\mathcal{I}_{\psi_+,\tau-\tau_0}(\psi_-)=\psi_-$.
\end{proof}

We also have:

\begin{lemma}\label{estimate for tau}
There exist  $r\gg 0$ and $\varepsilon>0$ such that for $T\gg 0$ the following holds:
\begin{enumerate}
\item There is a map $P'':\mathcal{B}_- \times \mathcal{B}_+ \to [-\varepsilon_0,\varepsilon_0]$ such that Equation~\eqref{equation: one more} holds with $\tau-\tau_0=P''(\psi_-,\psi_+)$.
\item $| P''(\psi_-,\psi_+)| \leq cr^{-1}( e^{-\lambda T}+\|\psi_+\|_*)+c\|\psi_-\|_*^2$.
\end{enumerate}
\end{lemma}

\begin{proof}
This is proved using the contraction mapping theorem as in the proofs of Lemmas~\ref{azalea} and \ref{suzuran} and is left to the reader.
\end{proof}

Putting Lemmas~\ref{azalea}, \ref{suzuran}, and \ref{estimate for tau} together we obtain:

\begin{lemma}\label{estimate for psi and tau}
There exist $r\gg 0$, $\varepsilon>0$, and $\varepsilon_0>0$ such that for $T\gg 0$ there is a unique solution $(\psi_-,\psi_+,\tau-\tau_0)\in \mathcal{B}_-\times \mathcal{B}_+\times [-\varepsilon_0,\varepsilon_0]$ to the equation $\Theta_+(\psi_-,\psi_+)=0$ and Equations~\eqref{equation: D+1-Pi} and ~\eqref{equation: one more}. Moreover,
$$\|\psi_-\|_*,\|\psi_+\|_*,|\tau-\tau_0| \leq cr^{-1}e^{-\lambda T}.$$
\end{lemma}

\begin{proof}
We solve Equation~\eqref{Theta +} using Lemma~\ref{azalea} to obtain $\psi_+=P(\psi_-)$. We then plug $\psi_+=P(\psi_-)$ into Equation~\eqref{equation: one more} and apply Lemma~\ref{estimate for tau} to obtain $\tau=P''(\psi_-,P(\psi_-))+\tau_0$. Finally, we plug $\psi_+=P(\psi_-)$ and $\tau=P''(\psi_-,P(\psi_-))+\tau_0$ into Equation~\eqref{equation: D+1-Pi} and solve for $\psi_-$.

To this end, we define a map $\mathcal I':\mathcal B_- \to \mathcal H_-$ which takes $\psi_-$ to $\mathcal I_{\psi_+,\tau-\tau_0}(\psi_-)$ as in Equation~\eqref{revisit} with $\psi_+=P(\psi_-)$ and $\tau=P''(\psi_-,P(\psi_-))+\tau_0$. Now we show that $\mathcal I'$ is a contraction mapping. By Estimate~\eqref{estimate for I three} and Lemmas~\ref{azalea}(2) and \ref{estimate for tau}(2) we get
\begin{align*}
\|\mathcal I'(\psi_-)\|_* & \leq c|\tau-\tau_0| + cr^{-1}(e^{-\lambda T}+\|\psi_+\|_*)+ c\|\psi_-\|_*^2\\
& \leq  c\left(cr^{-1}(e^{-\lambda T}+\|\psi_+\|_*)+c\|\psi_-\|_*^2\right)\\
&\hspace{1.5cm}+ cr^{-1}(e^{-\lambda T}+\|\psi_+\|_*)+ c\|\psi_-\|_*^2\\
& \leq cr^{-1}(e^{-\lambda T}+\|\psi_-\|_*)+c\|\psi_-\|_*^2,
\end{align*}
where the constant $c>0$ changes from line to line. Hence $\mathcal I'$ maps an $\varepsilon$-ball into itself.

Next we estimate
\begin{align}
\label{thurs}
&\|\mathcal I'(\psi_-)-\mathcal I'(\psi_-')\|_* \\
\nonumber \leq & c|\tau-\tau'|+c\| \tfrac{\partial \beta_+}{\partial s}(\psi_+-\psi_+')\|_* +c\|\mathcal R_-'(\psi_-,\tau-\tau_0)-\mathcal R_-'(\psi_-',\tau'-\tau_0)\|_* \\
\nonumber \leq & c|\tau-\tau'|+cr^{-1}\|\psi_+-\psi_+'\|_* +C\cdot(\|\psi_--\psi_-'\|_*+|\tau-\tau'|),
\end{align}
where $\psi_+'=\psi_+(\psi_-')$, $\tau'=\tau(\psi_+',\psi_-')$, and $C\to 0$ as
$$\|\psi_-\|^2+\|\psi_-'\|^2 +|\tau-\tau_0|+|\tau'-\tau_0| \to 0.$$
We now estimate the terms $|\tau-\tau'|$ and $\|\psi_+-\psi_+'\|_*$ in Estimate~\eqref{thurs}. Starting from Equation~\eqref{equation: one more} we estimate
\begin{equation*}
|\tau-\tau'|\leq cr^{-1}\|\psi_+-\psi_+'\|_* + C\|\psi_--\psi_-'\|_*.
\end{equation*}
Along the lines of Estimate~\eqref{estimate for I two} we get
\begin{equation*}
\|\psi_+-\psi_+'\|_* \leq C'\|\psi_--\psi_-'\|_*,
\end{equation*}
where $C'\to 0$ as $\|\psi_-\|_*+\|\psi_-'\|_*\to 0$. Combining these we obtain
\begin{equation}
\|\mathcal I'(\psi_-)-\mathcal I'(\psi_-')\|_*\leq C''\cdot \|\psi_--\psi_-\|_*,
\end{equation}
where $0<C''<1$. Hence $\mathcal I'$ is a contraction mapping, and it has a unique fixed point $\psi_-=\mathcal I'(\psi_-)$.

Finally, the claimed estimates follow from Lemmas~\ref{azalea}, \ref{suzuran} and \ref{estimate for tau}.
\end{proof}

\subsection{Obstruction bundle and obstruction section} \label{subsection: obstruction bundle}

Let $\mathcal{O}$ be the {\em obstruction bundle}
$$\mathcal{O} \to [R,\infty)\times \mathcal{M}/\mathbb{R}, \quad R\gg 0, $$
whose fiber over $(T,v_+)$ is
$$\mathcal{O}_{T,v_+}=\op{Hom}(\op{ker}D^*_-/\mathbb{R}\langle Y \rangle,\mathbb{R}).$$

We define the section $\mathfrak{s}$ of $\mathcal{O}$ as follows:
\begin{equation} \label{equation for s}
\mathfrak{s}(T,v_+)(\sigma)=\left\langle \sigma,\frac{\partial \beta_+}{\partial s}\eta_+ +\mathcal{R}_-\right\rangle,
\end{equation}
where $\sigma\in \op{ker} D^*_-/\R\langle Y \rangle$ and $\langle\cdot,\cdot\rangle$ denotes the $L^2$-inner product. More precisely, given $(T,v_+)$, we solve for $\psi_+$, $\psi_-$, $\tau$ in $\Theta_+(\psi_-,\psi_+)=0$ and in Equations~\eqref{equation: D+1-Pi} and \eqref{equation: one more} and evaluate the right-hand side of Equation~\eqref{equation for s} using these $\psi_+$, $\psi_-$, and $\tau$.

\begin{thm}[{\cite[Lemma~6.3]{HT2}}]
$\mathfrak{s}$ is smooth for $R\gg 0$.
\end{thm}

As in Lemma~\ref{lemma: good basis for coker},
$$\dim \op{ker} (D_-) =0, \quad \dim \ker (D^*_-)=\dim \op{coker}(D_-)=k.$$
Therefore, we expect $\mathfrak{s}^{-1} (0)$ to intersect $\{T\} \times \mathcal{M}/\mathbb{R}$ at finitely many points, for generic $T\geq 10r$.

To any $(T,v_+)\in [R,\infty)\times \mathcal{M}/\mathbb{R}$ we can associate $v(T,v_+)$ and $\tau(T,v_+)$, as described at the end of Section~\ref{subsection: gluing}.  When $(T,v_+)\in \mathfrak{s}^{-1} (0)$, we can view $(v(T,v_+),\tau(T,v_+))$ as an element of $\coprod_{0\leq \tau \leq 1}\mathcal{M}^{\op{ind}=0,\op{cyl}}_{\overline{J}^{\tau}}(\gamma'_+,\gamma_-)$.  Hence we have the gluing map
$$G: \mathfrak{s}^{-1}(0) \to \coprod_{0\leq \tau \leq 1}\mathcal{M}^{\op{ind}=0,\op{cyl}}_{\overline{J}^{\tau}}(\gamma'_+,\gamma_-),$$
$$ (T,v_+)\mapsto (v(T,v_+),\tau(T,v_+)).$$

\subsection{Bijectivity of the gluing map}\label{subsection: bijectivity}

Let $K$ be a subset of $\mathcal{M}/\R$.  Given $\delta>0$, let $\widetilde{\mathcal{G}}_\delta (v_-,K)$ be the set of (not necessarily holomorphic) maps that are $\delta$-close to some holomorphic building $v_-\cup v_+$ in the sense of \cite[Definition~7.1]{HT2}, where $v_+\in K$. Let $\mathcal{G}_\delta (v_-,K)$ be the subset of
$$\coprod_{0\leq \tau \leq 1}\mathcal{M}^{\op{ind}=0,\op{cyl}}_{\overline{J}^{\tau}}(\gamma_+,\gamma_-) $$
consisting of pairs $(v,\tau)$, where $v$ is a $\overline{J}^\tau$-holomorphic curve in $\widetilde{\mathcal{G}}_\delta (v_-,K)$.  Also let $\mathcal{U}_{\delta,K}\subset [R,\infty)\times K$ be the set of $(T,v_+)$ such that $v(T,v_+)\in \widetilde{\mathcal{G}}_\delta (v_-,K)$.

\begin{thm}[{\cite[Theorem~7.3]{HT2}}] \label{thm: bijectivity of gluing map}
Suppose $K\subset \mathcal{M}/\R$ is compact.
\begin{enumerate}
\item For $R'>0$ sufficiently large with respect to $\delta$,  $[R',\infty)\times K\subset  \mathcal{U}_{\delta,K}$.
\item For $\delta>0$ sufficiently small, there exists a compact set $K\subset K'\subset \mathcal{M}/\R$ such that the restriction
$$G|_{ \mathcal{U}_{\delta,K'}}:\mathfrak{s}^{-1}(0) \cap \mathcal{U}_{\delta,K'}\to \mathcal{G}_\delta(v_-,\mathcal{M}/\R)$$
of $G$ satisfies $\op{Im}(G|_{ \mathcal{U}_{\delta,K'}})\supset \mathcal{G}_\delta(v_-, K)$. Moreover, $G|_{ \mathcal{U}_{\delta,K'}}$ is a homeomorphism onto its image.
\end{enumerate}
\end{thm}

\subsection{The linearized section $\mathfrak{s}_0$} \label{subsection: linearized section s_0}

Recall the evaluation map $\widetilde{ev}^k_-: \mathcal{M}/\R\to S^{k-1}$ and also that $\dim\mathcal{M}/\R=k-1$.

We define a homotopy of sections $\mathfrak{s}_\zeta$, $\zeta\in[0,1]$, as follows: For $\zeta\in[{1\over 2},1]$,
\begin{equation} \label{homotopy of s}
\mathfrak{s}_\zeta(T,v_+)(\sigma)=\left\langle \sigma,\frac{\partial \beta_+}{\partial s}\eta_+ +\beta(2\zeta-1)\cdot\mathcal{R}_-\right\rangle,
\end{equation}
where $\beta$ is the cutoff function from Section~\ref{subsection:pregluing}.
More precisely, given $(T,v_+)$, we solve for $\psi_+$, $\psi_-$, $\tau$ in $\Theta_+(\psi_-,\psi_+)=0$ and in Equations~\eqref{equation: D+1-Pi} and \eqref{equation: one more} with $\mathcal{R}_-$ replaced by $\beta(2\zeta-1)\mathcal{R}_-$, and evaluate the right-hand side of Equation~\eqref{homotopy of s} using these $\psi_+$, $\psi_-$, and $\tau$.  Observe that the estimates from Section~\ref{subsection: estimates} carry over with $\mathcal{R}_-$ replaced by $\beta(2\zeta-1)\mathcal{R}_-$, allowing us to define $\mathfrak{s}_\zeta$.  For $\zeta\in[0,{1\over 2}]$,
\begin{equation} \label{homotopy of s two}
\mathfrak{s}_\zeta(T,v_+)(\sigma)=\left\langle \sigma_\zeta,\frac{\partial \beta_+}{\partial s}\eta_+ \right\rangle,
\end{equation}
where $\sigma_\zeta$ is the linear interpolation between $\sigma$ and $\tilde\sigma$, and $\tilde\sigma$ is $\sigma$ with the $f_{k+1}(t),f_{k+2}(t),\dots$ terms truncated.  Note that we are only concerned with the positive end of $\sigma$.

Here $\mathfrak{s}_1=\mathfrak{s}$ and $\mathfrak{s}_0(T,v_+)(\sigma)=\left\langle \tilde\sigma,\frac{\partial \beta_+}{\partial s}\eta_+ \right\rangle$.

\begin{lemma} \label{lemma: s_0}
With respect to the basis $\{\sigma_1,\dots,\sigma_k\}$ from Lemma~\ref{lemma: good basis for coker},
\begin{equation} \label{eqn: expression for s zero}
\mathfrak{s}_0(T,v_+)(\sigma_i)=e^{-2\lambda_i T}  c_i
\end{equation}
where $\widetilde{ev}^k_-(v_+)=(c_1,\dots,c_k)$ and $i=1,\dots,k-1$, and
\begin{equation}
\mathfrak{s}_0^{-1}(0)=[R,\infty)\times (\widetilde{ev}^k_-)^{-1}(\{(0,\dots,0,\pm 1)\}).
\end{equation}
\end{lemma}

\begin{proof}
We compute
\begin{align*}
\mathfrak{s}_0(T,v_+)(\sigma_i)& =\left\langle \tilde\sigma_i, {\bdry\beta_+\over \bdry s} \eta_+\right\rangle \\ & = \left\langle e^{-\lambda_is}f_i(t),  {\bdry\beta_+\over \bdry s}\left(\sum_j c_j e^{\lambda_j(s-2T)}f_j(t)\right)\right\rangle\\
&= e^{-2\lambda_i T}  c_i.
\end{align*}
Recall that $\{f_j\}_{j\in \Z-\{0\}}$ is an $L^2$-orthonormal basis and that $\int_{-\infty}^\infty {\bdry\beta_+\over \bdry s} ds=1$. The calculation of $\mathfrak{s}_0^{-1}(0)$ immediately follows from Equation~\eqref{eqn: expression for s zero}.
\end{proof}

The following proposition allows us to substitute $\mathfrak{s}$ for the linearized section $\mathfrak{s}_0$:

\begin{prop} \label{prop: zeros do not escape}
Let $K\subset \mathcal{M}/\R$ be a compact $(k-1)$-dimensional submanifold with boundary such that $\widetilde{ev}^k_-(\bdry K)\cap \{(0,\dots,0,\pm 1)\}=\varnothing$.  Then for $R\gg 0$, there are no zeros of $\mathfrak{s}_\zeta$, $\zeta\in[0,1]$, on $[R,\infty)\times \bdry K$.
\end{prop}

\begin{proof}
For our purposes we take a smooth section of $\mathcal{M}\to\mathcal{M}/\R$ over $K$ and a constant $C>0$ so that, for each $v_+$ representing $[v_+]\in K$,
\begin{enumerate}
\item[(K1)] $v_+|_{s\leq 0}$ has image in $\R\times (\R/\Z)\times D^2_{\delta_0/3}$ and is graphical over a subset of $\R\times \gamma_+'$;
\item[(K2)] if we write $v_+ =\left(s,t,\displaystyle \sum_{i>0} c_i e^{\lambda_i s} f_i(t)\right)$ on $s\leq 0$, then $\displaystyle\sum_{i=1}^k c_i^2={C^2\delta_0^2}$; and
\item[(K3)] $\displaystyle\sum_{i=1}^k c_i^2\gg \displaystyle\sum_{i>k} c_i^2$.
\end{enumerate}
We then take $0< \delta\ll {C\delta_0}$.

Let $R>0$ be sufficiently large with respect to $\delta$. Arguing by contradiction, suppose $(T,v_+)\in \mathfrak{s}^{-1}(0)\cap ([R,\infty)\times \bdry K)$.
Given $(T,v_+)\in [R,\infty)\times \bdry K$, consider $(v(T,v_+),\tau(T,v_+))$. Observe that $v(T,v_+)= \op{exp}_{v_*}(\psi_+)$ for $s\geq 2T$; $v(T,v_+)= \op{exp}_{v_*}(\psi_-)$ for $s\leq T_0$; and $v(T,v_+)$ can be written as
$$(s,t,\beta_+(\eta_++\psi_+)+\beta_-(\eta_-+\psi_-))$$
on $T_0\leq s\leq 2T$ by our choice of section. Recall that $v_{+,T}$ is $v_+$ translated up by $2T$.

By Lemmas~\ref{estimate for psi and tau} and \ref{lemma: sobolev embedding},
\begin{equation} \label{eqn: C zero bounds}
|\psi_-|_{C^0},~~~|\psi_+|_{C^0}, ~~~ |\tau-\tau_0| <cr^{-1}e^{-\lambda T},
\end{equation}
for some constant $c>0$. In view of the normalization of $v_+$ along $s=0$ given by (K2) and (K3), we have
\begin{equation} \label{coffee}
|\psi_+|_{C^0} \ll |\eta_+|_{C^0}
\end{equation}
along $s=2T$, provided $R\gg 0$.

We now make some explicit calculations in local coordinates for $\psi_+$ on the $s\leq 2T$ region. By Claim~\ref{claim: R plus and R minus on neck}, $\psi_+$ satisfies the equation
\begin{equation} \label{fig}
D_+\psi_+ + {\bdry \beta_-\over \bdry s} (\eta_-+\psi_-)=0.
\end{equation}
For the term $\eta_- + \psi_-$ in Equation~\eqref{fig} we have:

\begin{claim} \label{peach 1}
If $(T,v_+)\in \mathfrak{s}^{-1}(0)$, then
$$\eta_-+\psi_-=\displaystyle\sum_{i < 0} d_i e^{\lambda_i s}f_i (t)$$
for $s\geq T_0+hr$.
\end{claim}

\begin{proof}[Proof of Claim~\ref{peach 1}]
First note that $\eta_-$ is holomorphic.  On $s\geq T_0$, $\psi_-$ satisfies the equation
$$D_-\psi_-+{\bdry \beta_+\over \bdry s} (\eta_++\psi_+)=0$$
by Claim~\ref{claim: R plus and R minus on neck}.
Restricted to $s\geq T_0+hr$, this becomes $D_-\psi_-=0$.  Since both $\eta_-$ and $\psi_-$ have exponential decay as $s\to \infty$, $\eta_-+\psi_-$ can be written as $\displaystyle \sum_{i < 0} d_i e^{\lambda_i s}f_i (t)$ for $s\geq T_0+hr$.
\end{proof}

Next let us write
$$\psi_+(s,t)=\sum_{i\in\Z-\{0\}} b_i(s)e^{\lambda_i s} f_i(t).$$
Then we have the following:

\begin{claim} \label{momo}
If $(T,v_+)\in \mathfrak{s}^{-1}(0)$, then on $s<2T$:
\begin{enumerate}
\item for $i>0$, $b_i(s)$ is constant; and
\item for $i<0$, $b_i(s)=- d_i \int_{-\infty}^s {\bdry \beta_-\over \bdry s}ds.$
\end{enumerate}
\end{claim}

\begin{proof}[Proof of Claim~\ref{momo}]
Since $\psi_+$ has exponential decay as $s\to -\infty$ and ${\bdry \beta_-\over \bdry s}=0$ on $s< T-hr$, $b_i(s)$ must be constant for all $i$ and equal to zero for $i<0$ on $s<T-hr$ by Equation~\eqref{fig}.  By a similar argument which uses Claim \ref{peach 1}, $b_i(s)$ is constant on $T_0+hr<s<2T$ for $i>0$.
Therefore, $b_i(s)$ is constant on $s<2T$ for $i>0$, which gives (1). Solving for the Fourier coefficients for $i<0$ gives (2).
\end{proof}

\begin{claim} \label{claim: nonzero Fourier polynomial}
If $R\gg 0$, then $\widetilde{ev}^k_-(\eta_+ + \psi_+)$ is close to  $\widetilde{ev}^k_-(\eta_+)$ for all $[T,v_+)\in [R,\infty)\times \bdry K$.  In particular, $\widetilde{ev}^k_-(\eta_+ + \psi_+)\not=(0,\dots, 0, \pm 1)$.
\end{claim}

We will be using the identification $(\R^k-\{0\})/\R^+\simeq S^{k-1}_{C\delta_0}$; see Remark~\ref{identification}.  This is to take advantage of (K2).

\begin{proof}[Proof of Claim~\ref{claim: nonzero Fourier polynomial}]
If we write $\eta_+(s,t)=\displaystyle \sum_{i>0} c_i e^{\lambda_i (s-2T)} f_i(t)$, then the order $k$ Fourier polynomial for the negative end of $\eta_+ + \psi_+$ is
$$P_{k}(\eta_+ + \psi_+)=\displaystyle\sum_{i=1}^{k} (c_ie^{-2\lambda_i T} + b_i)e^{\lambda_i s} f_i(t).$$
By (K2), $\widetilde{ev}^k_-(\eta_+)=(c_1,\dots,c_k)$.  By using Equation~\eqref{coffee} along $s=2T$, we see that $\sum_{i=1}^k c_i^2\gg \sum_{i=1}^k (b_i e^{2\lambda_i T})^2$.
The claim then follows.
\end{proof}

\begin{claim} \label{Fourier polynomial error term}
For any $\varepsilon>0$, there exists $R\gg 0$ such that
$$|(\eta_+ + \psi_+)- P_{k}(\eta_++\psi_+)|  <\varepsilon |P_{k}(\eta_++\psi_+)|$$
pointwise on $s\leq 2T$ for all $[T,v_+)\in [R,\infty)\times \bdry K$.
\end{claim}

\begin{proof}[Proof of Claim~\ref{Fourier polynomial error term}]
By (K3),
$$|(\eta_+ + \psi_+)- P_{k}(\eta_++\psi_+)|  <\varepsilon |P_{k}(\eta_++\psi_+)|$$
holds along $s=2T$.  Since the error $(\eta_+ + \psi_+)- P_{k}(\eta_++\psi_+)$ has faster exponential decay than $P_k(\eta_++\psi_+)$ as $s$ decreases, the claim follows.
\end{proof}

Let us now consider $\Theta_-(\psi_-,\psi_+)=0$, which can be written as
\begin{equation} \label{asagao}
\overline{\bdry}_{\overline{J}^\tau} \op{exp}_{v_-}(\psi_-)   =-\frac{\partial \beta_+}{\partial s}(\eta_+ +\psi_+).
\end{equation}
Consider the finite-dimensional vector space $W = \R \langle {\bdry\beta_+\over \bdry s} e^{\lambda_i s} f_i(t) \rangle_{i=1}^{k}$. We write $Z_k=\R \langle {\bdry\beta_+\over \bdry s} e^{\lambda_k s} f_k(t) \rangle$.

\begin{lemma} \label{no soln}
If $B\subset W$ is a sufficiently small ball centered at $0$, then for all $g\in B- Z_k$ there is no solution to the equation
\begin{equation} \label{watermelon}
\overline{\bdry}_{\overline{J}^\tau} \op{exp}_{v_-}(\psi_-) =g.
\end{equation}
The same is also true for all $g$ of the form $g'+g''$, where $g'\in B-Z_k$, $g''\in \mathcal{H}_0(\dot F, \Lambda^{0,1} T^*F\otimes N)$, $\|g'\|< \varepsilon \|g''\|$, and $\varepsilon>0$ is a fixed sufficiently small number.
\end{lemma}

\begin{proof}[Proof of Lemma~\ref{no soln}]
First we claim that if $g\in B\cap Z_k$, then there exists a pair $(\psi_+,\tau)$ which solves Equation~\eqref{watermelon}:  Let us write $v_-=u_-\circ \pi$. Then $g\in B\cap Z_k$ can be written as $g=\pi^* g_0$ for some $g_0$. There is a solution $(\psi_{-,0},\tau)$ of
$$\overline{\bdry}_{\overline{J}^\tau} \op{exp}_{u_-}(\psi_{-,0})=g_0,$$
since $\op{ker}( D_{u_-}^*)$ is $1$-dimensional and is generated by $Y_0$ which comes from the variation of $\overline{J}^\tau$. If we write $\psi_-=\pi^* \psi_{-,0}$, then the pair $(\psi_-,\tau)$ solves Equation~\eqref{watermelon}.

The claim, interpreted suitably, implies the lemma: Consider the map
$$\overline{\bdry}^\mathcal{P}:\mathcal{H}_1(\dot F,N)\oplus \R\to \mathcal{H}_0(\dot F,\Lambda^{0,1}T^*\dot F\otimes N),$$
$$(\psi_-,\tau-\tau_0)\mapsto \mathcal{P}( \overline{\bdry}_{\overline{J}^\tau}\op{exp}_{v_-}(\psi_-)),$$
where $\mathcal{P}$ refers to a suitable parallel transport. For a sufficiently small ball $\mathcal{B}$ of $\mathcal{H}_1(\dot F,N)\oplus \R$ centered at $(0,0)$, $\overline{\bdry}^\mathcal{P}(\mathcal{B})$ is a submanifold of codimension $k-1$ in $\mathcal{H}_0(\dot F,\Lambda^{0,1}T^*\dot F\otimes N)$: this is because the linearization $L_{(\psi_-,\tau-\tau_0)}$ of $\overline{\bdry}^\mathcal{P}$ at any point $(\psi_-,\tau-\tau_0)\in\mathcal{B}$ has Fredholm index $-k+1$, $L_{(0,0)}$ has maximal rank, i.e., codimension $k-1$, and hence $L_{(\psi_-,\tau-\tau_0)}$ also has codimension $k-1$. If $(\psi_-,\tau-\tau_0)\in \mathcal{B}$, then $L_{(\psi_-,\tau-\tau_0)}$ has image which is close to that $L_{(0,0)}$.  Since
$$L_{(0,0)}=T_{(0,0)} \overline{\bdry}^\mathcal{P}(\mathcal{B})=\op{Im}(D_-)+\R\langle \sigma_k\rangle,$$
the tangent space $T_{(\psi_-,\tau-\tau_0)} \overline{\bdry}^{\mathcal{P}}(\mathcal{B})$ is close to $\op{Im}(D_-)+\R\langle \sigma_k\rangle$ when $(\psi_-,\tau-\tau_0)$ is small. If we write $g= g^\flat+g^\sharp\in B-Z_k$ with $g^\flat\in Z_k$ and $0\not = g^\sharp\in \R\langle {\bdry\beta_+\over \bdry s} e^{\lambda_i s} f_i(t) \rangle_{i=1}^{k-1}$ and solve for $(\psi_-,\tau-\tau_0)$ in $\overline{\bdry}_{\overline{J}^\tau} \op{exp}_{v_-}(\psi_-) =g^\flat$, then the absolute value of the angle between $g^\sharp$ and $T_{(\psi_-,\tau-\tau_0)} \overline{\bdry}^{\mathcal{P}}(\mathcal{B})$ is bounded below by a positive constant. Hence Equation~\eqref{watermelon} does not admit a solution for $g\in B-Z_k$.
\end{proof}

By Claims~\ref{claim: nonzero Fourier polynomial} and  \ref{Fourier polynomial error term}, the right-hand side of Equation~\eqref{asagao} is an element of $W$, plus a small error.  Hence Equation~\eqref{asagao} has no solution by Lemma~\ref{no soln}.  Unwinding the definition of $\mathfrak{s}$, $\mathfrak{s}(T,v_+)\not=0$ and we have a contradiction.

The argument for $\mathfrak{s}_\zeta$, $\zeta\in[{1\over 2},1]$, is done in exactly the same way.  The estimates from Lemmas~\ref{azalea}(2) and \ref{suzuran}(2) carry over and the nonexistence of the solution to Equation~\eqref{asagao} also holds for $\Theta_-(\psi_-,\psi_+)=0$ with $\mathcal{R}_-$ replaced by $\beta(2\zeta-1)\mathcal{R}_-$.

Finally, $\mathfrak{s}_\zeta$, $\zeta\in[0,{1\over 2}]$, is close to $\mathfrak{s}_0$ for $R\gg 0$.  Hence there are no zeros of $\mathfrak{s}_\zeta$, $\zeta\in[0,{1\over 2}]$, on $[R,\infty)\times \bdry K$. This completes the proof of Proposition~\ref{prop: zeros do not escape}.
\end{proof}

\subsection{The general prototypical gluing} \label{subsection: the general case}

So far we have been assuming that all the curves in $\mathcal{M}$ are immersed.   In general we have the following proposition, which allows us to reduce to the immersed case.

\begin{prop} \label{prop: can avoid boundary and singular points}
For $\delta>0$ sufficiently small, there are no holomorphic curves in $\coprod_{0\leq \tau \leq 1}\mathcal{M}^{\op{ind}=0,\op{cyl}}_{\overline{J}^{\tau}}(\gamma_+,\gamma_-) $ that are:
\begin{enumerate}
\item $\delta$-close to $(v_-, v_1,\dots, v_l)$, where $v_1\cup\dots \cup v_l$ is an $l$-level building of $\bdry (\mathcal{M}/\R)$ and $l>1$ or
\item $\delta$-close to $(v_-,v_+)$, where $v_+\in \mathcal{M}^{\op{sing}}$.
\end{enumerate}
\end{prop}

\begin{rmk}
The proof of Proposition~\ref{prop: can avoid boundary and singular points}(1) does not require any $v_i$ to be regular.
\end{rmk}

\begin{proof}
(1) Suppose there is a curve $u$ satisfying (1).  We will apply the proof of Proposition~\ref{prop: zeros do not escape} to obtain a contradiction. Let $v_{+}$ be a pregluing of $v_1\cup\dots \cup v_l$ with gluing parameters $T_1,\dots,T_{l-1}$ and let $v_*$ be a pregluing of $v_-\cup v_{+}$ with gluing parameter $T$. We focus on the neck portion $T_0\leq s\leq 2T$ between $v_-$ and $v_{+,T}$, which we take to be preglued in the same way as in Sections~\ref{subsection:pregluing} and \ref{subsection: gluing}.

Let $\psi_-$ be a section of the normal bundle of $v_-$, defined on $s\leq T_0$, so that $u= \op{exp}_{v_-}(\psi_-)$.  Similarly, let $\psi_+$ be a section of the normal bundle of $v_{+,T}$, defined on $s\geq 2T$, so that $u=\op{exp}_{v_{+,T}}(\psi_+)$.  (Assume without loss of generality that $v_{+}$ is immersed, since we only care about the negative end of $v_+$.)  If $\delta>0$ is sufficiently small, then the sections $\psi_-$ and $\psi_+$ exist and $u$ is graphical over $[T_0,2T]\times \R/\Z$ on the neighborhood $[T_0,2T]\times \R/\Z\times D^2_{\delta_0/3}$ with coordinates $(s,t,x,y)$.  Also let $\eta_+$, $\eta_-$ be as defined before.

\begin{claim}
There exist extensions of $\psi_+$ to $s\leq 2T$ and $\psi_-$ to $s\geq T_0$, viewed as functions to $D^2_{\delta_0/3}$, such that
$$\Theta_-(\psi_-,\psi_+)=D_-(\eta_-+\psi_-)+{\bdry \beta_+\over \bdry s}(\eta_++\psi_+)=0,$$
$$\Theta_+(\psi_-,\psi_+)=D_+(\eta_++\psi_+)+{\bdry\beta_-\over \bdry s}(\eta_-+\psi_-)=0,$$
on $s\geq T_0$ and $s\leq 2T$, respectively.
\end{claim}

\begin{proof}
This is proved in Lemmas 7.6 and 7.7 of \cite{HT2}.

More directly, if
$$u(s,t)=\sum_{i\in\Z-\{0\}} c_i e^{\lambda_i s} f_i(t)$$
on $T_0+hr\leq s\leq T-hr$, then we set
$$\eta_++\psi_+=\sum_{i>0} c_i e^{\lambda_i s} f_i(t),\quad s\leq T-hr,$$
$$\eta_-+\psi_-= \sum_{i<0} c_i e^{\lambda_i s} f_i(t),\quad s\geq T_0+hr.$$
We solve $\Theta_+(\psi_-,\psi_+)=0$ on $T-hr \leq s\leq 2T$ by decomposing into eigenmodes: Writing
$$\eta_+ + \psi_+=\sum_{i\in\Z-\{0\}} c_i(s) e^{\lambda_i s} f_i(t),$$
$c_i(s)=c_i$ for $i>0$. On the other hand, for $i<0$, we have $c_i'(s)+{\bdry \beta_-\over \bdry s} c_i=0$ and the solution $c_i(s)=-c_i \int_{-\infty}^s {\bdry \beta_-\over \bdry s} ds$ has the right boundary conditions $c_i(T-hr)=0$ and $c_i(2T)=c_i$. The equation $\Theta_-(\psi_-,\psi_+)$ is solved similarly.
\end{proof}

Returning to the proof of Proposition~\ref{prop: can avoid boundary and singular points}, now that we have sections $\psi_-$ and $\psi_+$ {\em irrespective of the regularity of $v_i$}, we reconsider Equation~\eqref{asagao}, which is satisfied by the pair $(\psi_-,\psi_+)$. By Theorem \ref{thm: niceness of compactification}, $\overline{ev}^k_-(\bdry (\mathcal{M}/\R))$ does not pass through $(0,\dots,0,\pm 1)$.  Hence Claims~\ref{claim: nonzero Fourier polynomial} and \ref{Fourier polynomial error term} and Lemma~\ref{no soln} imply that there is no solution of Equation~\eqref{asagao}, a contradiction. Notice that in this argument we only need $v_+$ to be holomorphic near the negative end, so we only need to preglue $v_1 \cup \cdots \cup v_l$ instead of glue.

(2) is similar since $\overline{ev}^k_-(\mathcal{M}^{\op{sing}}/\R)$ does not intersect $(0,\dots,0,\pm 1)$ by Theorem~\ref{thm: niceness of ev}. Note that when a curve has a singular point, there is no normal bundle $N$; however, the entire argument carries over to the situation of $u^*T\widehat X^\tau$ with minimal change.
\end{proof}

Let $K_1\subset K_2\subset \dots \subset (\mathcal{M}-\mathcal{M}^{\op{sing}})/\R$ be an exhaustion of $(\mathcal{M}-\mathcal{M}^{\op{sing}})/\R$ by compact submanifolds of dimension $k-1$ with boundary.

\begin{cor} \label{cor: gluing}
For sufficiently small $\delta>0$, sufficiently large $j>0$, and sufficiently large $R>0$,
$$G|_{ \mathcal{U}_{\delta,K_j}}:\mathfrak{s}^{-1}(0) \cap \mathcal{U}_{\delta,K_j}\stackrel\sim\longrightarrow \mathcal{G}_\delta(v_-,\overline{\mathcal{M}/\R})$$
is a homeomorphism.
\end{cor}

\begin{proof}
Immediate from Proposition~\ref{prop: can avoid boundary and singular points} and Theorem~\ref{thm: bijectivity of gluing map}.
\end{proof}

\subsection{Other cases} \label{subsection: other cases}

In this subsection we treat two other cases:

\s\n
A.  Suppose  \hyperref[Condition1]{(C1)}--(C4) holds with (C1) modified to (C1') so that $\gamma''=\gamma_+'$ is positive hyperbolic and $\gamma'=\gamma_-$ is negative hyperbolic.  Then Lemma~\ref{lemma: good basis for coker} is modified so that $\sigma_1$ is a nonzero constant multiple of $Y$ modulo $f_{k+1},f_{k+2},\dots$; \cb this modification also carries over to Lemma~\ref{lemma: better basis for coker}. \cb The rest of the discussion carries over with $(0,\dots,0,\pm 1)\in S^{k-1}$ replaced by $(\pm 1, 0,\dots, 0)\in S^{k-1}$.

\s\n
B. Consider the $3$-level SFT building $v_{-1}\cup v_0\cup v_1$ where:
\begin{enumerate}
\item[(C1'')] $v_1$ is a cylinder from $\gamma_+$ to $\gamma'_+$, $v_0$ is a cylinder from $\gamma'_+$ to $\gamma'_-$, and $v_{-1}$ is a cylinder from $\gamma'_-$ to $\gamma_-$; we assume that $\gamma'_+$ is negative hyperbolic and $\gamma'_-$ is positive hyperbolic;
\item[(C2'')] $v_0$ maps to a cobordism $(\widehat{X}^{\tau_0},\widehat{\alpha}^{\tau_0},\overline{J}^{\tau_0})$ for some $\tau_0\in (0,1)$ and $v_{-1}$ and $v_1$ map to symplectizations;
\item[(C3'')] $\op{ind}(v_{-1})=b$, $\op{ind}(v_0)=-k$, $\op{ind}(v_1)=a$, where $a,b>0$ and $a+b=k>1$;
\item[(C4)] $v_0$ is a $k$-fold unbranched cover of a transversely cut out (in a $1$-parameter family) cylinder $u_0$ with $\op{ind}(u_0)=-1$ and $v_1$ is regular; and we write $v_0=u_0\circ \pi$, where $\pi$ is the covering map.
\end{enumerate}

The gluing setup is similar to the previous situation:  Let
$$\mathcal{M}_1=\mathcal{M}^{\op{ind}=a,\op{cyl}}_{J_+}(\gamma_+,\gamma'_+),\quad \mathcal{M}_{-1}=\mathcal{M}^{\op{ind}=b,\op{cyl}}_{J_-}(\gamma_-',\gamma_-).$$
We assume that all the curves in $\mathcal{M}_1$ and $\mathcal{M}_{-1}$ are immersed; modifications can be made using the analog of Proposition~\ref{prop: can avoid boundary and singular points}.  The obstruction bundle is
$$\mathcal{O}\to [R,\infty)^{\times 2} \times (\mathcal{M}_{-1}/\R)\times (\mathcal{M}_1/\R),$$
whose fiber over $(T_-,T_+,v_{-1},v_1)$ is $\op{Hom}(\ker D_{v_0}^*/\R\langle Y\rangle,\R)$. The pregluing is done in exactly the same way with gluing parameters $T_+$ and $T_-$, such that the pregluing between $v_1$ and $v_0$ is the same as that of Section~\ref{subsection:pregluing} with parameter $T_+$ instead of $T$ and the pregluing between $v_{-1}$ and $v_0$ is done symmetrically about $s=0$ with parameter $T_-$.  Let us write $\eta_1$ for the end corresponding to $v_1(s-2T_+,t)$ and $\eta_{-1}$ for the end corresponding to $v_{-1}(s+2T_-,t)$. Define the cut-off functions $\beta_1(s)=\beta \left(\frac{s-T_0}{hr} \right)$, $\beta_{-1}(s)=\beta \left( \frac{-T_0-s}{hr} \right)$, and $\beta_0(s)=\beta \left(\frac{s+T_{-1}}{hr}\right) \beta \left(\frac{T_1-s}{hr}\right)$.
We then define the pregluing of $v_{-1} \cup v_0 \cup v_{1}$ by
$$v_*(s,t)=\left\{\begin{array}{ll}
v_{1,T_1}(s,t) & \op{for } T_1\leq s, \\
(s,t, \beta_{1}(s)\eta_{1}(s,t)+\beta_{0}(s)\eta_0(s,t) )& \op{for } T_0\leq s < T, \\
v_{0}(s,t) & \op{for }  -T_0\leq s < T_0, \\
(s,t,\beta_{-1}(s)\eta_{-1}(s,t)+\beta_0(s)\eta_0(s,t)) & \op{for } -T_{-1}\leq s <-T_0,\\
v_{-1,T_{-1}}(s,t) & \op{for } s<-T_{-1}.
\end{array} \right.$$
Assume for convenience $v_{-1},v_0,$ and $v_1$ are all immersed. Otherwise, one can go through the proof of Proposition \ref{prop: can avoid boundary and singular points} to deal with the rest \cb of the \cb cases. Let $\psi_1,\psi_0,$ and $\psi_{-1}$ be sections of the normal bundles of $\eta_{-1},\eta_0$, and $\eta_1$, respectively. We deform the $v_*$ as before by $$v=\exp_{v_*}(\beta_1\psi_1+\beta_0\psi_0+\beta_{-1}\psi_{-1}).$$

The obstruction section $\mathfrak{s}$ is defined in a similar manner as in Section~\ref{subsection: obstruction bundle}.

The linearized section is
\begin{equation} \label{s zero part two}
\mathfrak{s}_0(T_-,T_+,v_{-1},v_1)(\sigma)= \left\langle \tilde\sigma, {\bdry\beta_{1}\over \bdry s}\eta_1  -{\bdry\beta_{-1}\over \bdry s} \eta_{-1}\right\rangle,
\end{equation}
where $\tilde\sigma$ is $\sigma$ with the $f_{k+1}(t),f_{k+2}(t),\dots$ terms truncated at the positive end and $g_{-k-1}(t),g_{-k-2}(t),\dots$ truncated at the negative end.  We use the basis
$$\{\sigma_1',\dots,\sigma_{k-1}',\sigma_k\}$$
from Lemma~\ref{lemma: better basis for coker 2} to analyze $\mathfrak{s}_0^{-1}(0)$. By Theorems~\ref{thm: transversality of evaluation map} and \ref{thm: niceness of ev} and our genericity assumption, there exist evaluation maps
$$\widetilde{ev}^a_-(\gamma_+,\gamma'_+):\mathcal{M}_1/\R\to S^{a-1},$$
$$\widetilde{ev}^b_+(\gamma'_-,\gamma_-):\mathcal{M}_{-1}/\R\to S^{b-1}.$$
The evaluation maps will be abbreviated $\widetilde{ev}(\eta_1)$ and $\widetilde{ev}(\eta_{-1})$ when the ends are understood.

\begin{lemma} \label{close to the end}
Let $K\subset (\mathcal{M}_{-1}/\R)\times (\mathcal{M}_1/\R)$ be a compact $(k-2)$-dimensional submanifold with boundary. Then for $R\gg 0$, there are no zeros of $\mathfrak{s}_0$ on $[R,\infty)^{\times 2}\times \bdry K$.
\end{lemma}

\begin{proof}
Suppose $R\gg 0$. Let $(v_{-1},v_1)\in K$ and let us denote their translated ends by
$$\eta_1(s,t)= \sum_{i=1}^\infty c_i e^{\lambda_i(s-2T_+)} f_i(t), \quad \eta_{-1}(s,t)=\sum_{i=-\infty}^{-1} d_i e^{\lambda'_i(s+2T_-)}g_i(t).$$

Suppose that $\lambda_a<\lambda_{a+1}$.
Let us write
\begin{align*}
\tilde\sigma'_i(s,t)&=\left\{
\begin{array}{ll}
\sum_{i\leq j\leq k} c_{i,j}e^{-\lambda_j s}f_j(t), & c_{i,i}=1,\\
\sum_{-k \leq j\leq -k+i-1} d_{i,j}e^{-\lambda'_j s}g_j(t),&  d_{i,-k+i-1}\not=0,
\end{array} \right.
\end{align*}
at the positive and negative ends. By a computation similar to that of Lemma~\ref{lemma: s_0},
\begin{align}
\label{udon}
\left\langle \tilde\sigma'_i, {\bdry\beta_{1}\over \bdry s} \eta_1\right\rangle & = \sum_{i\leq j\leq k} c_{i,j} c_j e^{-2\lambda_j T_+},\\
\label{soba}
\left\langle \tilde \sigma'_i, {\bdry\beta_{-1}\over \bdry s}  \eta_{-1}\right\rangle & = \sum_{-k\leq j\leq -k+i-1}d_{i,j} d_j e^{2\lambda_j' T_-}.
\end{align}

Now let us write
$$ c_{i,*}=(c_{i,1},\dots, c_{i,k}), \quad d_{i,*}=(d_{i,-k},\dots, d_{i,-1}),$$
where we are setting $c_{i,1}=\dots= c_{i,i-1}=0$ and $d_{i,-k+i}=\dots = d_{i,-1}=0$, and
$$ C_1=\left(c_1 e^{-2\lambda_1T_+},\dots, c_a e^{-2\lambda_a T_+}\right), \quad C_2=\left (c_{a+1}e^{-2\lambda_{a+1} T_+},\dots, c_k e^{-2\lambda_k T_+}\right),$$
$$ D_{-2}=\left (d_{-k} e^{2\lambda'_{-k} T_-},\dots, d_{-k+a-1}e^{2\lambda'_{-k+a-1}T_-}\right),$$
$$ D_{-1}=\left(d_{-b}e^{2\lambda'_{-b}T_-},\dots,d_{-1}e^{2\lambda'_{-1}T_-}\right).$$
Recall that $a+b=k$. Since $\lambda_a<\lambda_{a+1}$ and $(c_1,\dots,c_a)$ and $(d_{-b},\dots,d_{-1})$ are nonzero, it follows that $|C_2|\ll |C_1|$ and $|D_{-2}|\ll |D_{-1}|$ for $R\gg 0$.

Suppose $i=j_0$ maximizes \eqref{udon}, where we are ranging over $i=1,\dots, a$. We claim that if $\mathfrak{s}_0(\sigma'_{j_0})=0$, then $|C_1|\leq c |D_{-2}|$ for some constant $c>0$ which does not depend on $(T_-,T_+,v_{-1},v_1)$, provided $R\gg 0$. First observe that $\{c_{i,*}\}_{i=1}^k$ and $\{d_{i,*}\}_{i=1}^k$ form bases of $\R^k$ and that \eqref{udon} and \eqref{soba} can be written as:
$$\langle c_{i,*}, (C_1,C_2) \rangle \quad \mbox{and} \quad  \langle d_{i,*},(D_{-2},D_{-1})\rangle,$$
where we are using the standard inner product on $\R^k$.  It is not hard to see that there exist a constant $c>0$, which only depends on $\{c_{i,*}\}_{i=1}^k$, and an integer $i=j_0'$, which depends on $(C_1,C_2)$, such that
$$|\langle c_{j_0',*}, (C_1,C_2) \rangle|\geq c \cdot |c_{j_0',*}|\cdot |(C_1,C_2)|.$$
(Roughly speaking, some vector $c_{j_0',*}$ makes an angle with $(C_1,C_2)$ which is bounded away from ${\pi\over 2}$.)  If $i=j_0$ maximizes \eqref{udon} and $\mathfrak{s}_0(\sigma'_{j_0})=0$, then
$$c \cdot |c_{j_0',*}|\cdot |(C_1,C_2)| \leq |\langle c_{j_0,*}, (C_1,C_2) \rangle| \leq  |\langle d_{j_0,*},(D_{-2},D_{-1})\rangle|.$$
Hence, by Lemma~\ref{lemma: better basis for coker 2} and the fact that $|D_{-2}|\ll |D_{-1}|$,
$$ c\cdot |c_{j_0',*}| \cdot |C_1| \leq |d_{j_0,*}| \cdot |D_{-2}|,$$
which implies the claim.

Combining the inequalities  $|C_2|\ll |C_1|$, $|C_1|\leq c |D_{-2}|$, and $|D_{-2}|\ll |D_{-1}|$, we obtain $|C_2|\ll |D_{-1}|$. On the other hand, applying the above argument to $\tilde\sigma'_{j_1}$, $a+1\leq j_1< k$, such that $\mathfrak{s}_0(\sigma'_{j_1})=0$, we obtain $|D_{-1}| \leq c|C_2|$, a contradiction.
This proves the lemma when $\lambda_a<\lambda_{a+1}$.

When $\lambda_a=\lambda_{a+1}$, the only case that is not treated by the above argument is when
\begin{align*}
\eta_1(s,t)&=(c_a  f_a(t)+c_{a+1} f_{a+1}(t))e^{\lambda_a(s-2T_+)},\\
\eta_{-1}(s,t)&=(d_{-b-1}g_{-b-1}(t)+d_{-b} g_{-b}(t)) e^{\lambda'_{-b-1}(s+2T_-)},
\end{align*}
up to a small error.  By the transversality of the evaluation maps, there is only a finite number of possible values for $(c_a,c_{a+1})$ and $(d_{-b-1},d_{-b})$; moreover their genericity implies that, for each $(T_-,T_+,v_{-1},v_1)$, there exists $(r_a,r_{a+1})$ such that $\mathfrak{s}_0(T_-,T_+,v_{-1},v_1)(r_a\sigma'_a+ r_{a+1}\sigma'_{a+1})\not=0$.  This proves the lemma.
\end{proof}

The existence of a homotopy $\mathfrak{s}_\zeta$ from $\mathfrak{s}_1=\mathfrak{s}$ to $\mathfrak{s}_0$ without zeros on $[R,\infty)^{\times 2}\times \bdry K$ is proved as in Proposition~\ref{prop: zeros do not escape}, using Lemma~\ref{close to the end}. The higher codimension strata are eliminated using the argument of Section~\ref{subsection: the general case}.

\section{Orientations} \label{section: orientations}

In this section we discuss orientations for the moduli spaces. Section~\ref{subsection: orientation d squared} and Sections~\ref{subsubsection: a lemma}--\ref{subsubsection: k=1 case} are standard.

\subsection{$\bdry^2=0$} \label{subsection: orientation d squared}

\subsubsection{Signs in the definition of $\bdry$}

We first discuss orientations involved in the definition of $\bdry$.  Let $u\in \mathcal{M}^{\op{ind}=\ell,\op{cyl}}_{J}(\gamma,\gamma')$, where we are suppressing asymptotic markers from the notation. For simplicity let us assume that $u$ is immersed. Let $D=D_{u}$ be the linearized normal $\overline\bdry$-operator
and let
$$\det(D)=\Lambda^{\op{top}} \ker D \otimes \Lambda^{\op{top}} (\op{coker} D)^*$$
be its determinant line.  We define an equivalence relation $\sim$ on $\det(D)-\{0\}$ as follows: $\xi_1\sim\xi_2$ if $\xi_1= c \xi_2$ for $c\in \R^+$. The equivalence class of $\xi\in \det(D)-\{0\}$ is denoted by $[\xi]$.  Let $\frak{o}(D)$ be the orientation of $\det(D)$ given by \cite{BM}. An orientation of $\det(D)$ can be viewed as an equivalence class of $\det(D)-\{0\}$.

When $\op{ind}(u)=1$, we can write $\frak{o}(D)=[\op{sgn}(u)\cdot \bdry_s]$, where $\bdry_s$ refers to the section of the normal bundle corresponding to the infinitesimal translation in the $s$-direction and $\op{sgn}(u)=\pm 1$.

\s\n
{\bf Sign assignment.} In the definition of $\bdry$, we assign the sign $\op{sgn}(u)$ to $[u]\in \mathcal{M}^{\op{ind}=1,\op{cyl}}_{J}(\gamma,\gamma')/\R$.

\subsubsection{$\bdry^2=0$}

Next we discuss orientations in the proof of $\bdry^2=0$. We are gluing/pregluing $u_1\in \mathcal{M}^{\op{ind}=1,\op{cyl}}_J(\gamma'',\gamma')$ and $u_2\in \mathcal{M}^{\op{ind}=1,\op{cyl}}_J(\gamma,\gamma'')$. Assume $u_1$ and $u_2$ are regular. Let $u_1\# u_2$ be a pregluing of $u_1$ and $u_2$ and let $D_{u_1\# u_2}$ be the linearized normal $\overline\bdry$-operator for $u_1\# u_2$. We assume that the ``neck length" $T$ is sufficiently large. Then by the gluing property for coherent orientations from \cite{BM}, there is an isomorphism
$$\det(D_{u_1\# u_2})\simeq \det(D_{u_1})\otimes \det(D_{u_2})$$
which is natural up to a positive constant and
\begin{equation} \label{widetilde otimes}
\mathfrak{o}(D_{u_1\# u_2})= \mathfrak{o}(D_{u_1})\widetilde\otimes \mathfrak{o}(D_{u_2}),
\end{equation}
where the right-hand side is defined as follows: We use cutoff functions $\beta_i: \R\to [0,1]$, $i=1,2$, that are analogous to $\beta_\pm$ from Section~\ref{subsection:pregluing}. Given $\xi_i\in \ker D_{u_i}$, we translate by $\pm T$ and damp it out at the positive or negative end (as appropriate) by multiplying by $\beta_i$.  This yields $\xi_i'=\beta_i (\xi_i)_T$.  Here the subscript $T$ indicates a translation {\em which depends on $T$}. We then view $\xi'_i$ as an element of the domain of $D_{u_1\# u_2}$ and take the $L^2$-orthogonal projection to $\ker D_{u_1\# u_2}$ to obtain $\widehat \xi_i$.
Finally, if
$$\mathfrak{o}(D_{u_1})=[\op{sgn}(u_1)\partial_s^1] \quad \mbox{and} \quad \mathfrak{o}(D_{u_2})=[\op{sgn}(u_2)\partial_s^2],$$
where $\partial_s^i$, $i=1,2$, is the vector field that corresponds to translation in $s$-direction inside the moduli space corresponding to $u_i$, then
$$\mathfrak{o}(D_{u_1})\widetilde\otimes \mathfrak{o}(D_{u_2}):=[\op{sgn}(u_1)\widehat{\partial}_s^1\wedge \op{sgn}(u_2)\widehat{\partial}_s^2].$$

Now let $u_1\cup u_2$ and $u^\flat_1\cup u^\flat_2$ be the two boundary points of a component $\mathcal{N}$ of $\overline{\mathcal{M}/\R}$.
Let $a$ be a nonsingular vector field of $\mathcal{N}$ that points away from $u_1\cup u_2$ and towards $u_1^\flat\cup u_2^\flat$, and let $\bdry_s^{12}$ be the vector field of $\mathcal{M}$ that corresponds to translation in the $s$-direction.  If we denote $\mathfrak{o}(D_{u_1\#u_2})=[\op{sgn}(a)\partial_s^{12}\wedge a]$, then $\widehat{\partial}_s^1+\widehat{\partial}_s^2$ corresponds to $\partial_s^{12}$ and $\widehat{\partial}_s^1-\widehat{\partial}_s^2$ corresponds to $a$. Therefore,
$$[\op{sgn}(u_1)\widehat{\partial}_s^1\wedge \op{sgn}(u_2)\widehat{\partial}_s^2]=[\op{sgn}(a)(\widehat{\partial}_s^1+\widehat{\partial}_s^2)\wedge(\widehat{\partial}_s^1-\widehat{\partial}_s^2)],$$
and $\op{sgn}(u_1)\op{sgn}(u_2)=-\op{sgn}(a)$.  The analogous computation for $u_1^\flat\cup u^\flat_2$ implies
\begin{equation}
\op{sgn}(u^\flat_1) \op{sgn}(u^\flat_2)=\op{sgn}(a)=-\op{sgn}(u_1)\op{sgn}(u_2).
\end{equation}
This is the desired relation for $\bdry^2=0$.

\subsection{Chain homotopy} \label{subsection: orientations chain homotopy}

\subsubsection{A lemma} \label{subsubsection: a lemma}

In this subsection we will make frequent use of the following standard lemma (cf.\ \cite[p.\ 676]{FO3}, for example):

\begin{lemma} \label{lemma: comparison}
If $\phi: V\to W$ is a Fredholm map and $E$ is a finite-dimensional subspace of $W$ such that $W= \op{Im} \phi + E$, then there is an isomorphism
$$\Phi_E:\det\phi\stackrel\sim\longrightarrow\det \phi^{-1}(E)\otimes \det E^*,$$
which is natural up to a positive constant.
\end{lemma}

More explicitly, if $\phi^{-1}(E)= \ker \phi \oplus F$ and $E= \op{coker}\phi \oplus \phi(F)$, and $\ker\phi$, $F$, $\op{coker}\phi$, and $\phi(F)$ have bases $\{v_1,\dots,v_m\}$, $\{f_1,\dots,f_\ell\}$, $\{w_1,\dots,w_n\}$, and $\{\phi(f_1),\dots,\phi(f_\ell)\}$, then the isomorphism $\Phi_E$ is given by:
\begin{align*}
[v_1\wedge\dots\wedge v_m\otimes w_n^*\wedge\dots\wedge w_1^*]\mapsto & [v_1\wedge\dots\wedge v_m\wedge f_1\wedge\dots\wedge f_\ell\otimes\\
 & \quad \phi(f_\ell)^*\wedge\dots\wedge \phi(f_1)^*\wedge w_n^*\wedge\dots\wedge w_1^*],
\end{align*}
where $\{w_1^*,\dots,w_n^*\}$ is the dual basis to $\{w_1,\dots,w_n\}$ and $\{\phi(f_1)^*,\dots,\phi(f_\ell)^*\}$ is the dual basis to $\{\phi(f_1),\dots,\phi(f_\ell)\}$.

\subsubsection{The $k=1$ case} \label{subsubsection: k=1 case}

We mostly use the notation from Section~\ref{subsection: chain homotopy}.  Let
$$\mathcal{M}^0=\coprod_{0\leq \tau\leq 1}\mathcal{M}^{\op{ind}=0,\op{cyl}}_{\overline{J}^\tau}(\gamma_+,\gamma_-)$$
and let $\pi: \mathcal{M}^0\to [0,1]$ be the projection to $\tau\in[0,1]$.

Consider $$v_0 \cup v_1 \in \left. \left(\mathcal{M}_{\overline{J}^{\tau_l}}^{\op{ind}=-k,\op{cyl}}(\gamma'_+,\gamma_-)\times \mathcal{M}_{J_+}^{\op{ind}=k,\op{cyl}}(\gamma_+,\gamma'_+)/\mathbb{R}\right) \right/\sim.$$
Let
$$D_{v_0}: \mathcal{H}_1(\dot F_0, N_{v_0})\to \mathcal{H}_0(\dot F_0, \Lambda^{0,1} T^*\dot F_0\otimes N_{v_0}),$$
$$D_{v_1}: \mathcal{H}_1(\dot F_1, N_{v_1})\to \mathcal{H}_0(\dot F_1, \Lambda^{0,1} T^*\dot F_1\otimes N_{v_1}),$$
be the linearized normal $\overline\bdry$-operators for $v_0:\dot F_0\to \widehat{X}^{\tau_0}$ and $v_1: \dot F_1\to \R\times M_+$, where $\mathcal{H}_0$ and $\mathcal{H}_1$ are the Morrey spaces described in Section~\ref{subsection: banach spaces}.  We then write
$$\widetilde D_{v_0}: \mathcal{H}_1(\dot F_0, N_{v_0})\oplus \R\to \mathcal{H}_0(\dot F_0, \Lambda^{0,1} T^*\dot F_0\otimes N_{v_0}),$$
$$(\xi,c)\mapsto D_{v_0}\xi + cY',$$
where the generator of the summand $\R$ is denoted by $a$.

We first consider the case $k=1$. We define $\op{sgn}(v_0)$ and $\op{sgn}(v_1)$ by
$$\mathfrak{o}(D_{v_0})=[\op{sgn}(v_0) (Y')^*],\quad \mathfrak{o}(D_{v_1})=[\op{sgn}(v_1)\partial_s].$$
Then, by Lemma~\ref{lemma: comparison},
$$\mathfrak{o}(\widetilde D_{v_0})=[\op{sgn}(v_0) a\otimes (Y')^*]=\op{sgn}(v_0)[1],$$
since $\widetilde D_{v_0}$ maps $a\mapsto Y'$.

Let $\widetilde D=\widetilde D_{v_0\# v_1}$ be the operator obtained from pregluing $\widetilde D_{v_0}$ and $D_{v_1}$:
$$\widetilde D: \mathcal{H}_1(\dot F_0\# \dot F_1, N_{v_0\# v_1})\oplus \R\to \mathcal{H}_0(\dot F_0\# \dot F_1, \Lambda^{0,1} T^*(\dot F_0\# \dot F_1)\otimes N_{v_0\# v_1}),$$
$$(\xi,c)\mapsto D_{v_0\# v_1}\xi + c\beta_0 Y',$$
where $\beta_0,\beta_1$ are the same as the cutoff functions $\beta_-,\beta_+$ from Section~\ref{subsection:pregluing}. For $i=0,1$ we also have maps
$$
\mathcal{H}_1(\dot F_i,N_{v_i})\to \mathcal{H}_1(\dot F_0\# \dot F_1, N_{v_0\# v_1}),$$
$$
\mathcal{H}_0(\dot F_i, \Lambda^{0,1} T^*\dot F_i\otimes N_{v_i})\to \mathcal{H}_0(\dot F_0\# \dot F_1, \Lambda^{0,1} T^*(\dot F_0\# \dot F_1)\otimes N_{v_0\# v_1}),$$
$$\xi_0\mapsto \xi_0':=\beta_0 \xi_0, \quad \xi_1 \mapsto \xi_1':=\beta_1 (\xi_1)_T$$
which restrict to inclusions on $\ker D_{v_1}$ and $\op{coker} D_{v_0}$.  Here the subscript $T$ indicates a translation which depends on $T$ and $\xi_i'$ is viewed as a section of the appropriate bundle. For $\xi_i\in \mathcal{H}_1(\dot F_i,N_{v_i})$, we let $\widehat\xi_i$ be the $L^2$-orthogonal projection of $\xi'_i$ to $\ker \widetilde D$. By abuse of notation, let $\bdry_\tau$ be the pullback of $\bdry_\tau$ under the projection $\pi:\mathcal{M}^0\to[0,1]$. If $v_0 \cup v_1$ is on $\bdry\mathcal{M}^0$ such that $\partial_\tau$ points away from (resp.\ towards) $v_0 \cup v_1$ and towards (resp.\ away from) the interior of $\mathcal{M}^0$, then $\partial_\tau$ corresponds to $-\widehat\partial_s$ (resp.\ $\widehat\bdry_s$). By the gluing property for coherent orientations we have
$$\mathfrak{o}(\widetilde D)=\op{sgn}(v_0)\op{sgn}(v_1)[\widehat\bdry_s]=\op{sgn}(v_0)\op{sgn}(v_1)[-{\bdry}_\tau]$$
(resp.\ $\mathfrak{o}(\widetilde D)=\op{sgn}(v_0)\op{sgn}(v_1)[{\bdry}_\tau]$) and the sign $\op{sgn}(v_0)\op{sgn}(v_1)$ assigned to $v_0\cup v_1$ agrees with the boundary orientation of $\mathcal{M}^0$.

\subsubsection{The case $k>1$}

Next we consider the case $k>1$. Define $\op{sgn}(v_0)$ and $\op{sgn}(v_1)$ by
$$\mathfrak{o}(D_{v_0})=[\op{sgn}(v_0)\sigma_k^*\wedge \cdots \wedge \sigma_1^*], \quad \mathfrak{o}(D_{v_1})=[\op{sgn}(v_1)e_1\wedge \cdots \wedge e_{k}],$$
where $e_i\in \ker D_{v_1}$, $1\leq i \leq k$, corresponds to $f_i$. Then
$$\mathfrak{o}(\widetilde D_{v_0})=[\op{sgn}(v_0)a \otimes\sigma_k^*\wedge \cdots \wedge \sigma_1^*]=[\op{sgn}(v_0)\sigma_{k-1}^*\wedge \cdots \wedge \sigma_1^*].$$
Notice that $\op{sgn}(v_1)$ is locally constant.

We will now define $\mathfrak{o}(\widetilde D_{v_0})\widetilde\otimes \mathfrak{o}(D_{v_1})$.  Let
$$F=\R\langle e_1',\dots,e'_k\rangle,\quad E=\R\langle \sigma_1',\dots,\sigma_{k}'\rangle.$$
We are assuming that $T\gg 0$ so that $D_{v_0\# v_1}$ is an isomorphism and the composition of $\widetilde D|_F$ and the $L^2$-projection $p:\widetilde D(F)\to E$ is also an isomorphism. Then we set $\widehat \xi= \xi' - D_{v_0\# v_1}^{-1}((1-p)\widetilde D(\xi'))$ and
\begin{equation} \label{theodore}
\mathfrak{o}(\widetilde D_{v_0})\widetilde\otimes \mathfrak{o}(D_{v_1}):=\op{sgn}(v_0)\op{sgn}(v_1) [\widehat{e}_1\wedge \cdots \wedge \widehat{e}_k \otimes (\sigma'_{k-1})^*\wedge \cdots \wedge (\sigma'_1)^*],
\end{equation}
which makes sense in light of Lemma~\ref{lemma: comparison}.

The gluing property for coherent orientations (in a slightly more general form than that of \cite{BM}) implies:

\begin{lemma} \label{dumbledore}
For $T\gg 0$, $\mathfrak{o}(\widetilde D)=\mathfrak{o}(\widetilde D_{v_0})\widetilde\otimes \mathfrak{o}(D_{v_1}).$  Hence
$$\mathfrak{o}(\widetilde D)=\op{sgn}(v_0)\op{sgn}(v_1) [\widehat{e}_1\wedge \cdots \wedge \widehat{e}_k \otimes (\sigma'_{k-1})^*\wedge \cdots \wedge (\sigma'_1)^*].$$
\end{lemma}

Let $\widetilde D_v$ be the operator for the glued curve $v$, given as follows:
$$\widetilde D_v: \mathcal{H}_1(\dot F_0\# \dot F_1,N_v)\oplus \R\to \mathcal{H}_0(\dot F_0\#\dot F_1, \Lambda^{0,1} T^*(\dot F_0\#\dot F_1)\otimes N_v),$$
$$ (\xi,c)\mapsto D_v\xi + c\beta_0 Y',$$
where $D_v$ is the linearized normal $\overline\bdry$-operator for $v$. Then the analog of Lemma~\ref{dumbledore} also holds for $\widetilde D_v$.

\subsubsection{Comparison with signs of zeros of $d\mathfrak{s}$} \label{subsub: guava}

Next we compare $\mathfrak{o}(\widetilde D_v)$ with the signs of zeros of $d\mathfrak{s}$. As before we write $\mathcal{M}=\mathcal{M}_{J_+}^{\op{ind}=k,\op{cyl}}(\gamma_+,\gamma'_+)$.
Assume that $T\gg 0$ is sufficiently generic so that $\mathfrak{s}:\{T\}\times \mathcal{M}/\R\to \mathcal{O}$ is transverse to the zero section. Let $v_1\in \mathfrak{s}^{-1}(0)$.  In Section~\ref{subsub: guava} only, we write $v_+=v_1$ so that we have agreement with Section~\ref{section: gluing}.
We orient $T_{v_+}(\mathcal{M}/\R)$ by
$$-[e_1\wedge \cdots \wedge e_{k}\otimes \partial_s^* ]=-[e^\sharp_1\wedge\dots\wedge e^\sharp_{k-1}\wedge \bdry_r\otimes \bdry_s^*]=[e^\sharp_1\wedge\dots\wedge e^\sharp_{k-1}],$$
where $\bdry_r$ is the outward radial vector field on $\R^k$ and $e^\sharp_1,\dots,e^\sharp_{k-1}$ are tangent to $S^{k-1}\subset \R^k$ at $\widetilde{ev}^k(v_+)$ so that $[e^\sharp_1\wedge\dots\wedge e^\sharp_{k-1}\wedge \bdry_r]=[e_1\wedge \cdots \wedge e_{k}]$,
and orient the fiber $\mathcal{O}_{T,v_+}$ by $[\sigma_{k-1}^*\wedge \cdots \wedge \sigma_1^*].$

\s\n
{\bf Sign of $d\mathfrak{s}$.} We define $\op{sgn}d\mathfrak{s}(v_+)\in \{\pm1\}$ as the sign of $\det (\pi_{\mathcal{O}}\circ d\mathfrak{s}(v_+))$, where $\pi_{\mathcal{O}}$ is the projection to the fiber $\mathcal{O}_{T,v_+}$.

\s
Let
$$F^\sharp=\R\langle (e_1^\sharp)',\dots,(e_{k-1}^\sharp)'\rangle,\quad E^\sharp=\R\langle \sigma_1',\dots,\sigma_{k-1}' \rangle,$$
and let $D_v^\sharp: F^\sharp\to E^\sharp$ be the composition of $\widetilde D_v|_{F^\sharp}$ and the $L^2$-projection to $E^\sharp$. Also we orient $F^\sharp$ by $[(e_1^\sharp)'\wedge\dots\wedge(e_{k-1}^\sharp)']$ and $E^\sharp$ by $[\sigma_1'\wedge\dots\wedge \sigma_{k-1}']$; these are analogous to the orientations for $T_{v_+}(\mathcal{M}/\R)$ and $\mathcal{O}_{T,v_+}$.

\begin{lemma} \label{lemma: relate orientations}
$\op{sgn}d\mathfrak{s}(v_+)=\op{sgn} \op{det} D_v^\sharp$.
\end{lemma}

\begin{proof}
Using the notation from Section~\ref{section: gluing}, given
$$v=\op{exp}_{v_*}(\beta_+\psi_+ + \beta_-\psi_-),$$
we consider $\overline\bdry_{\overline{J}^\tau} v$, which equals the left-hand side of Equation~\eqref{eqn: d-bar rewritten}. Here $\beta_0=\beta_-$ and $\beta_1=\beta_+$.  For a zero $v_+$ of $\mathfrak{s}$, there exist $(\psi_+,\psi_-,\tau)$ so that $\overline\bdry_{\overline{J}^\tau} v=0$; in other words, $(\psi_+,\psi_-,\tau)$ solves Equations~\eqref{equation: D+1-Pi}--\eqref{equation: Pi=0}.

Consider the variation $\beta_+\phi_+$ with $\phi_+\in \ker D_+$. We claim that
\begin{equation} \label{equivalence}
d\mathfrak{s}(v_+)(\phi_+)(\sigma)=\left\langle \sigma, \widetilde D_v(\beta_+(\phi_+ + \psi_+^\flat) +\beta_-\psi_-^\flat,\tau^\flat)\right\rangle,
\end{equation}
where $(\psi_+ + \phi_+ + \psi_+^\flat,\psi_- + \psi_-^\flat, \tau+\tau^\flat)$ satisfy Equations~\eqref{equation: D+1-Pi} and \eqref{equation: one more}, with $(\psi_+,\psi_-,\tau)$ replaced by $(\psi_+ + \phi_+ + \psi_+^\flat,\psi_- + \psi_-^\flat, \tau+\tau^\flat)$, and $\sigma\in\ker D^*_- /\R\langle Y\rangle $.  Although strictly speaking not necessary, we write out the left-hand side of Equation~\eqref{eqn: d-bar rewritten} for
$$v_{\widetilde\phi_+}=\op{exp}_{v_*}(\beta_+(\psi_+ +\phi_++\psi_+^\flat) + \beta_-(\psi_-+\psi_-^\flat))$$
and $\tau+\tau^\flat$, using the fact that $\overline\bdry_{\overline{J}^\tau} v=0$:
\begin{align*}
&\beta_-\left(D_-\psi_-^\flat +\tau^\flat Y' + {\bdry\beta_+\over \bdry s}(\phi_+ +\psi_+^\flat) +\Delta \mathcal{R}'_-(\psi_- +\psi_-^\flat,\tau-\tau_0+\tau^\flat)\right)\\
+ & \qquad  \beta_+\left(D_+(\phi_+ +\psi^\flat_+) +{\bdry\beta_-\over \bdry s}(\psi_-^\flat) +\Delta\mathcal{R}'_+(\psi_+ +\phi_+ +\psi_+^\flat)\right),
\end{align*}
up to first order in $\phi_+,\psi_+^\flat,\psi_-^\flat,\tau^\flat$.  Here $\Delta\mathcal{R}'_+$ has terms of the form $B(\psi_-, \psi_-^\flat)$, $B(\psi_- ,\tau^\flat)$, $B(\tau-\tau_0,\psi_-^\flat)$, and $B(\tau-\tau_0,\tau^\flat)$; and $\Delta\mathcal{R}'_+$ has terms of the form $B(\psi_+, \phi_+ +\psi_+^\flat)$. By $B(r_0,r_1)$ we mean a term that is linear in $r_1$ with coefficients that are functions of $r_0$. Then, by Equations~\eqref{equation: D+1-Pi} and \eqref{equation: one more}, $d\mathfrak{s}(v_+)(\phi_+)$ is given by
\begin{equation} \label{one}
\sigma\mapsto \left\langle \sigma,{\bdry\beta_+\over \bdry s}(\phi_+ +\psi_+^\flat)+\Delta\mathcal{R}'_-(\psi_- +\psi_-^\flat,\tau-\tau_0+\tau^\flat)\right\rangle,
\end{equation}
whereas $\langle \sigma,\widetilde D_v(\beta_+(\phi_+ + \psi_+^\flat) +\beta_-\psi_-^\flat,\tau^\flat)\rangle$ is given by
\begin{align} \label{one and a half}
\sigma  \mapsto & \left \langle \sigma,\beta_- \left({\bdry\beta_+\over \bdry s}(\phi_+ +\psi_+^\flat)+\Delta\mathcal{R}'_-(\psi_- +\psi_-^\flat,\tau-\tau_0+\tau^\flat)\right)\right\rangle\\
\nonumber & = \left \langle \sigma,{\bdry\beta_+\over \bdry s}(\phi_+ +\psi_+^\flat)+\Delta\mathcal{R}'_-(\psi_- +\psi_-^\flat,\tau-\tau_0+\tau^\flat)\right\rangle,
\end{align}
since $\beta_-=1$ on the support of ${\bdry\beta_+\over \bdry s}$ and $\Delta\mathcal{R}'_-$. This proves the claim.

The claim, together with Lemma~\ref{dumbledore} for $\widetilde D_v$, implies the lemma.
\end{proof}

\subsubsection{From $d\mathfrak{s}$ to $d\mathfrak{s}_0$}

The homotopy $\mathfrak{s}_\zeta$, $\zeta\in[0,1]$, from $\mathfrak{s}=\mathfrak{s}_1$ to $\mathfrak{s}_0$ gives an oriented cobordism from $\mathfrak{s}^{-1}(0)$ to $\mathfrak{s}^{-1}_0(0)$.  In view of Lemmas~\ref{dumbledore} and \ref{lemma: relate orientations} and the facts that $\op{sgn}(v_0)$ is fixed and $\op{sgn}(v_1)$ is locally constant, it suffices to keep track of $\op{sgn} d\mathfrak{s}_\zeta$ as we go from $\mathfrak{s}$ to $\mathfrak{s}_0$.

Let $v_1\in \mathfrak{s}^{-1}_0(0)$.  We will determine $\op{sgn}d\mathfrak{s}_0(v_1)$, using Equation~\eqref{eqn: expression for s zero}. There are two cases: $v_1=v_1^+$ or $v_1^-$, where $\widetilde{ev}^k_-(v_1^\pm)=(0,\dots,0,\pm 1)$.

\begin{lemma} \label{lemma: signs}
$\op{sgn}d\mathfrak{s}_0(v_1^\pm)=\pm 1$.
\end{lemma}

\begin{proof}
The tangent space $T_{v_1^\pm}(\mathcal{M}/\R)$ is oriented by $\pm[e_1\wedge \cdots \wedge e_{k-1}]$. For $i=1,\dots,k-1$, $d\mathfrak{s}_0(v_1^+)$ maps $e_i$ to some positive multiple of $\sigma_i^*$ by Equation~\eqref{eqn: expression for s zero}.  Hence $\op{sgn}d\mathfrak{s}_0(v_1^+)=+1$. Similarly, $\op{sgn}d\mathfrak{s}_0(v_1^-)=-1$.
\end{proof}

Let $\widetilde D_{v_0\# v_1^\pm}$ be the analog of $\widetilde D_v$ for $v_0\# v_1^\pm$, with one modification: $D_{v_0\# v_1^\pm}$ is the linearization of $\overline\bdry_{\overline{J}^\tau} (v_0\# v_1^\pm)$, where the term $F_-(\psi_-,\tau-\tau_0)$ of $\mathcal{R}_-$ is set to zero. This is analogous to changing $\mathfrak{s}$ to $\mathfrak{s}_0$.

By Lemmas~\ref{dumbledore} and ~\ref{lemma: relate orientations} adapted to $\widetilde D_{v_0\# v_1^\pm}$ and Lemma~\ref{lemma: signs},
\begin{align}
\label{dumbledore 2}
\mathfrak{o}(\widetilde D_{v_0\# v_1^\pm}) &= \op{sgn}(v_0)\op{sgn}(v_1^\pm) [\widehat e_1\wedge\dots \wedge \widehat e_k \otimes (\sigma'_{k-1})^* \wedge \dots\wedge (\sigma'_1)^*]\\
\nonumber &= (-1)^{k-1} \op{sgn}(v_0) \op{sgn}(v_1^\pm)[\pm \bdry_r \wedge \widehat e_1\wedge\dots \wedge \widehat e_{k-1} \otimes (\sigma'_{k-1})^* \wedge \dots\wedge (\sigma'_1)^*]\\
\nonumber &= (-1)^{k-1} \op{sgn}(v_0) \op{sgn}(v_1^\pm)(\pm 1)\op{sgn} d\mathfrak{s}_0(v_1^\pm) [\pm \bdry_r]\\
\nonumber &= (-1)^{k} \op{sgn}(v_0) \op{sgn}(v_1^\pm)[\mp\bdry_r]
\end{align}
where $\bdry_r$ is the radial vector field for $\R^k$. Note that $e_k=\bdry_r$ at $(0,\dots,0,1)$ and $e_k=-\bdry_r$ at $(0,\dots,0,-1)$.

\subsubsection{Orientations over $\nu$}

Recall that the embedded arc $\nu\subset S^{k-1}$ is oriented from $(0,\dots,0,1)$ to $(0,\dots,0,-1)$. Let $\widetilde\nu$ be a connected component of $(\overline{ev}_-^k)^{-1}(\nu)$.

Suppose $v_1^\pm$ are the endpoints of $\widetilde\nu$ over $(0,\dots,0,\pm 1)$. We can view $\dot \nu$ as the continuation of the tangent vector field $-\bdry_r$ at $v_1^+$ and $\bdry_r$ as the continuation of $\dot \nu$ at $v_1^-$. Since $\op{sgn}(v_0)$ is constant and $\op{sgn}(v_1^+)=\op{sgn}(v_1^-)$, the orientations of $v_0\cup v_1^+$ and $v_0\cup v_1^-$ that come from Equation~\eqref{dumbledore 2} are consistent with the boundary orientation of $v_0\cup \bdry \widetilde\nu$ (up to an overall sign), i.e., the signs are opposite.

Next consider
$$w_1\cup w_2 \in \left.\left( \mathcal{M}_+^{k-1}(\zeta_+,\gamma'_+)/\R \times \mathcal{M}_{J_+}^{\op{ind}=1,\op{cyl}}(\gamma_+,\zeta_+)/\mathbb{R}  \right)\right/ \sim, $$
i.e., $w_1 \cup w_2$ is a boundary point of $\widetilde \nu$ that lies over the interior of $\nu$. We define $\op{sgn}(w_1)$ by $$\mathfrak{o}(D_{w_1})=[\op{sgn}(w_1)\frak{o}(N)\wedge \partial_s^1],$$
where $N$ is the normal bundle to $\nu$ inside $S^{k-1}$ and $\frak{o}(N)$ is (a representative of) the orientation for $N$ such that $[\frak{o}(N) \wedge \dot \nu  \wedge \partial_r]=[e_1\wedge \cdots \wedge e_k]$. Also define $\op{sgn}(w_2)$ by
$$\mathfrak{o}(D_{w_2})=[\op{sgn}(w_2)\partial_s^2].$$
Then
$$\mathfrak{o}(D_{w_1\#w_2})=\mathfrak{o}(D_{w_1})\widetilde\otimes \mathfrak{o}(D_{w_2})=\op{sgn}(w_1)\op{sgn}(w_2)[\widehat{\frak{o}(N)}\wedge \widehat\partial_s^1 \wedge \widehat\partial_s^2],$$
where $\widetilde\otimes$ is defined as in the discussion after Equation~\eqref{widetilde otimes}.

If $\dot{\nu}$ points from the interior of $\widetilde\nu$ towards $w_1\cup w_2$, then $\widehat\partial_s^1 +\widehat\partial_s^2$ corresponds to $\partial_s^{12}$ or $-\bdry_r$, and $-\widehat\partial_s^1+\widehat\partial_s^2$ corresponds to $\dot{\nu}$. Hence
\begin{align} \label{eqn: uno}
\mathfrak{o}(D_{w_1\#w_2}) & =-\op{sgn}(w_1)\op{sgn}(w_2)[\widehat{\frak{o}(N)}\wedge \dot{\nu} \wedge \partial_s^{12}]\\
\nonumber & =-\op{sgn}(w_1)\op{sgn}(w_2)[\widehat{\frak{o}(N)}\wedge \dot{\nu} \wedge -\bdry_r]\\
\nonumber & =\op{sgn}(w_1)\op{sgn}(w_2)[e_1\wedge\dots \wedge e_k].
\end{align}
In particular, $\op{sgn}(w_1)\op{sgn}(w_2)= \op{sgn}(v_1^\pm)$, if $v_1^\pm$ is in the same component of $\overline{\mathcal{M}/\R}$ as $w_1\cup w_2$.
On the other hand, if $\dot{\nu}$ points from $w_1\cup w_2$ towards the interior of $\widetilde\nu$, then $\widehat\partial_s^1 +\widehat\partial_s^2$ corresponds to $\partial_s^{12}$ or $-\bdry_r$, and $\widehat\partial_s^1-\widehat\partial_s^2$ corresponds to $\dot{\nu}$. Hence
\begin{align} \label{eqn: dos}
\mathfrak{o}(D_{w_1\#w_2}) &=-\op{sgn}(w_1)\op{sgn}(w_2)[\widehat{\frak{o}(N)}\wedge \dot{\nu} \wedge \bdry_r]\\
\nonumber &=-\op{sgn}(w_1)\op{sgn}(w_2)[e_1\wedge\dots \wedge e_k].
\end{align}
In particular, $-\op{sgn}(w_1)\op{sgn}(w_2)= \op{sgn}(v_1^\pm)$, if $v_1^\pm$ is in the same component of $\overline{\mathcal{M}/\R}$ as $w_1\cup w_2$.
The signs of $\op{sgn}(w_1)\op{sgn}(w_2)$ in Equations~\eqref{eqn: uno} and \eqref{eqn: dos} are opposite, as desired.

Finally, we compare the signs of Equations~\eqref{eqn: uno} and \eqref{eqn: dos} with those of Equation~\eqref{dumbledore 2}. We contract the right-hand sides of Equations~\eqref{eqn: uno} and \eqref{eqn: dos} with $e_{k-1}^*\wedge\dots\wedge e_1^*$ to obtain
$$(-1)^k \op{sgn}(w_1)\op{sgn}(w_2)[-\bdry_r] \quad \mbox{and} \quad (-1)^{k-1}\op{sgn}(w_1)\op{sgn}(w_2)[\bdry_r].$$
These agree with
$$(-1)^k \op{sgn}(v_1^+)[-\bdry_r] \quad \mbox{and} \quad (-1)^k \op{sgn}(v_1^-)[\bdry_r],$$
which are obtained from Equation~\eqref{dumbledore 2} by dividing by $\op{sgn}(v_0)$.

This finishes the proof of Theorem \ref{thm: chain homotopy}.

\end{document}